\documentclass[11pt]{amsart}
\usepackage{amscd}
\usepackage{amsfonts}
\usepackage{color}
\usepackage{amsmath,amssymb,amsthm,latexsym}
\usepackage{amscd}

\usepackage{enumerate}

\newtheorem{theorem}{Theorem}[section]
\newtheorem{proposition}[theorem]{Proposition}
\newtheorem{definition}[theorem]{Definition}
\newtheorem{corollary}[theorem]{Corollary}
\newtheorem{lemma}[theorem]{Lemma}
\usepackage{amsmath}
\usepackage{amsfonts}
\usepackage{amsmath,amsthm}
\usepackage{amscd}
\usepackage[all]{xy}
\numberwithin{equation}{section}
\theoremstyle{remark}
\newtheorem{remark}[theorem]{Remark}
\newtheorem{example}[theorem]{\bf Example}
\newcommand{\R}{\mathbb{R}}
\newcommand{\C}{\mathbb{C}}
\newcommand{\D}{\mathbb{D}}
\newcommand{\E}{\mathbb{E}}

\newcommand{\FC}{\mathcal{C}}

\begin{document}

\title[Willmore surfaces in spheres via loop groups]{\bf{Willmore surfaces in spheres via loop groups  $I$:
generic cases and some examples}}
\author{Josef F. Dorfmeister, Peng Wang }

\date{}
\maketitle

\begin{center}
{\bf Abstract}
\end{center}

 {\small
In this paper we deal with the global properties of Willmore surfaces in spheres via the harmonic conformal Gauss map using loop groups.

We first derive a  global description of those harmonic maps which can be realized as  conformal Gauss maps of some Willmore surfaces (Theorem \ref{normalizationlemma},  Theorem \ref{th-Willmore-harmonic-U} and Theorem \ref{th-Willmore-harmonic}).

Then we introduce the DPW procedure for these harmonic maps, and state appropriate versions of the Iwasawa decomposition and the Birkhoff decomposition Theorems. In particular, we show how the harmonic maps associated with Willmore surfaces can be constructed in terms of loop groups.

The third main result, which has many implications for the case of Willmore surfaces in spheres, shows that every  harmonic map into some non-compact inner symmetric space $G/K$ induces a harmonic map into the compact dual inner symmetric space $U/{(U \cap  K^\mathbb{C})}$.
From this correspondence we obtain additional information about the global properties of harmonic maps into non-compact inner symmetric spaces.

As an illustration of the theory developed in this paper we list examples (some of which were worked out in separate papers by following the theory of the present paper). In particular, we present an explicit, unbranched (isotropic) Willmore sphere in $S^6$ which is not S-Willmore, and thus does not have a dual Willmore surface. This example gives a negative answer to a long open problem (originally posed by Ejiri).\\
}

\vspace{0.5mm}  {\bf \ \ ~~Keywords:}  Willmore surface;  conformal Gauss map; normalized potential;  non-compact symmetric space; Iwasawa decomposition.     \vspace{2mm}

{\bf\   ~~ MSC(2010): \hspace{2mm} 53A30, 53C30, 53C35}

\tableofcontents

\section{Introduction}

\subsection{Motivation for and outline of our work}
A Willmore surface in $S^{n+2}$ is a critical surface of the Willmore functional
$\int_M(|\vec{H}|^2-K+1)dM$
with $\vec{H}$  and $K$ being the mean curvature vector and the Gauss curvature respectively. It is well-known that the Willmore functional is invariant under conformal transformations of $S^{n+2}$. Moreover, by Blaschke \cite{Blaschke}, Bryant \cite{Bryant1988}, Ejiri \cite{Ejiri1988}, and Rigoli \cite{Rigoli1987}, a conformal immersion is Willmore if and only if its conformal Gauss map is harmonic.

Later, H\'{e}lein's important observation \cite{Helein,Helein2}, generalized by Xia-Shen \cite{Xia-Shen} and also developed  in a different direction by Xiang Ma \cite{Ma2006}, indicates that Willmore surfaces are also related to other  harmonic maps (into inner  symmetric spaces). Moreover, a loop group theory was built for this kind of harmonic maps in \cite{Helein,Xia-Shen}. However, this new type of harmonic maps may have singularities, which makes it very unclear how to derive global properties of Willmore surfaces this way.

Therefore, in this paper we will use the conformal Gauss map to study the global geometry of Willmore surfaces. For this purpose, one has two things to do.
\begin{enumerate}
\item To characterize the harmonic maps which can be the conformal Gauss map of a Willmore surface, and to devise a procedure which produces from such harmonic maps  a Willmore surface.
\item To derive a DPW type construction for Willmore surfaces. This entails an appropriate statement of the Birkhoff and   the Iwasawa decomposition Theorem for the corresponding loop groups.
\end{enumerate}

After these preparations, it is natural to consider more specific examples of Willmore surfaces in spheres. It turns out to particularly important to discuss Willmore surfaces of finite uniton type ( see below and Section 5.2 for a more extensive discussion). Now there exists a detailed study of harmonic maps of finite uniton type into compact inner symmetric spaces \cite{BuGu}.
But the symmetric target space of the conformal Gauss map is not compact.
In order to be able to apply the work \cite{BuGu} it is therefore necessary to associate to a harmonic map into a non-compact inner symmetric space a harmonic map into a compact inner symmetric space. This is our third main result.

 In particular, using a result of Uhlenbeck we can show this way  that all Willmore spheres are of finite uniton type.
Moreover, after Ejiri's work \cite{Ejiri1988} the classification of Willmore $2-$spheres in $S^{n+2}$, $n>2$,  stays an open problem.
We also would like to remark that for minimal $2-$spheres as special Willmore $2-$spheres,  recently in \cite{Fernandez} Fern\'{a}ndez provided a characterization of the module space of minimal $2-$spheres  with fixed energy in  $S^{n+2}$. 
 Therefore it is natural to give a priority to the discussion of Willmore 2-spheres in $S^{n+2}$.  In this spirit Willmore 2-spheres in $S^{n+2}$ have been investigated in \cite{Wang-1,Wang-iso}.
 Moreover, in \cite{Wang-1} we derive a coarse classification of Willmore 2-spheres via normalized potentials (a loop group tool).  As a concrete example (worked out in \cite{Wang-iso} by following the method explained in this paper) we state explicitly a new Willmore 2-sphere in $S^6$, giving a counterexample to Ejiri's expectation.
\vspace{3mm}



 Clearly, the next  class of surfaces to be considered are Willmore surfaces of genus $g =1$.  Willmore tori can be viewed as doubly periodic Willmore planes, i.e., Willmore planes with special symmetries. Also other higher genus Willmore surfaces can be viewed as Willmore surfaces with symmetries. So it is  natural to consider independently Willmore surfaces with symmetries. This leads to other applications of our method \cite{DoWa-sym1,DoWa-sym2}. Along this way, equivariant Willmore surfaces (Willmore surfaces with a one-parameter group of symmetries) and homogeneous Willmore surfaces will also be investigated.  While we can so far produce some examples of Willmore tori following our theoretical description, we expect that our work will eventually make it possible to generalize Bohle's result \cite{Bo} and to classify all Willmore tori in $S^{n+2}$.

In addition to the construction of (maybe branched) immersions of Riemann surfaces, equally well non-orientable surfaces can be constructed.
The basic theory for this is in   \cite{DoWa-sym2}.
 Following this theory, we have constructed (see loc.cit.) Willmore  immersions  of Moebius bands and of $\R P^2$. The case of Willmore Klein bottles we plan to investigate in the future.

\subsection{Main results of this paper}
To state the main results, let us first recall briefly some basic ideas of Willmore surfaces and the DPW construction.

\subsubsection{Willmore surfaces and the DPW construction}

A surface $y$ in $S^{n+2}$  has a mean curvature 2-sphere at every point. The conformal Gauss map $Gr$ maps each point to the oriented mean curvature 2-sphere at this point. Since in conformal geometry an oriented mean curvature 2-sphere of $S^{n+2}$  is identified with an oriented 4-dimensional Lorentzian subspace of the oriented $n+4$ dimensional Lorentz-Minkowski space $\R^{n+4}_1$, $Gr$ is a map into $Gr_{3,1}(\R^{n+4}_{1})=SO^+(1,n+3)/SO^+(1,3)\times SO(n)$.
Here by $SO^+(1,n+3)$ we denote the connected component of the special linear isometry group of $\mathbb{R}^{n+4}_1$
which contains the identity element. Here $``\ ^+"$ comes from  the fact that  $SO^+(1,n+3)$ preserves
the forward light cone.
 The Lie algebra of $SO^+(1,n+3)$ is
\begin{equation}\label{eq-so(1,n+3)}
\mathfrak{so}(1,n+3)=\mathfrak{g}=\{X\in gl(n+4,\mathbb{R})|X^tI_{1,n+3}+I_{1,n+3}X=0\}, ~\hbox{ with }~I_{1,n+3}=\hbox{diag}(-1,1,\cdots,1).
\end{equation}

The DPW construction for harmonic maps into a symmetric space $G/K$ is to relate these harmonic maps to some special meromorphic or holomorphic 1-forms, via two loop group decompositions, the Birkhoff decomposition and Iwasawa decomposition. Roughly speaking, we have the following mutually inverse procedures \cite{DPW,Wu} (assuming we fix the values of all frames at a base point).
\begin{enumerate}
\item  For a harmonic map $f$ into $G/K$, the Maurer-Cartan form of a lift of $f$ admits an $S^1-$symmetry, which allows one to introduce a loop parameter $\lambda\in S^1$. By integration one obtains {\em an extended frame} depending on $\lambda$. The Birkhoff decomposition of the extended frame w.r.t. $\lambda$ gives a meromorphic frame and the M-C form of this frame provides a special (loop algebra valued) meromorphic 1-form, called the {\em normalized potential} of the harmonic map.
\item Conversely, we start from such a special meromorphic 1-form (normalized potential), requiring that the ode with the normalized potential as coefficient matrix yields  a meromorphic frame. The Iwasawa decomposition of this frame provides an extended frame which, after projection to $G/K$, gives the desired harmonic map.
\end{enumerate}

 It is worthwhile pointing out that for the construction of conformally harmonic maps of finite uniton type  the last step is particularly easy, at least in principle, since in this case the normalized potential takes value in a fixed nilpotent Lie algebra, so that the ode solution and the loop group decompositions become much simpler. We will show that all conformally harmonic maps defined on $S^2$ are of finite uniton type and we will discuss this type of harmonic maps in more detail in \cite{DoWa2}.

In our case we have $G/K=SO^+(1,n+3)/SO^+(1,3)\times SO(n)$.
As a consequence, for the DPW construction, a basic thing we need to do is to derive the Birkhoff Decomposition Theorem and the Iwasawa Decomposition Theorem  for the natural groups describing the  symmetric space used in this paper (Section 4 and Appendix).
Note, in the literature these decomposition theorems are proven for simply connected Lie groups. For the construction of Willmore surfaces we need to use $G = SO^+(1,n+3)$, which is not simply connected. Therefore our decomposition theorems are slightly different from the ones that can be  found in the literature.

\subsubsection{Statement of the main results}

Since our overall goal is to provide a theory which permits to construct
in a concrete way as many Willmore surfaces as possible (with all kinds of fundamental groups) in high dimensional spheres, the primary question is, whether the potentials of the loop group theory can be chosen in a way so that one obtains Willmore surfaces of the required type.

Our first main result is the characterization of all harmonic maps which are the conformal Gauss map of some Willmore surface. To this end, we need to
 provide two procedures:
 \begin{enumerate}
\item  {\it From a Willmore surface to its conformally  harmonic  Gauss map (the simple direction).}

We first derive a description of the Maurer-Cartan form $\alpha$ of a natural frame associated with the conformal Gauss map of a conformal surface in $S^{n+2}$. This yields a very specific structure of the matrices occurring (see Proposition \ref{frame}). Observing that the conformal Gauss map of a Willmore surface is harmonic we show (see Corollary \ref{nilpotency}) that the submatrix
$B_1$  of $\alpha$ satisfies the nilpotency condition $B_1^t I_{1,3} B_1 = 0$, where $I_{1,3} = diag(-1,1,...,1)$.

\item
{\it From a harmonic map to a Willmore surface, where it is possible (the difficult direction).}

 In Theorem \ref{lemma-B-12} we show that the nilpotency condition for the submatrix $B_1$ of the Maurer-Cartan form of some harmonic map  is sufficient to assure that this Maurer-Cartan form
can be normalized -- without changing the harmonic map -- such that the new submatrix matrix $B_1$ has exactly the form  the Maurer-Cartan form of the natural frame of the conformal Gauss map of some conformal map into $S^{n+2}$.  Now the question is when such a conformally harmonic map is the conformal Gauss map of some conformal immersion (which is equivalent to being the conformal Gauss map of some Willmore immersion).
 It turns out, see Theorem \ref{th-Willmore-harmonic-U}, that locally  the non-vanishing and the vanishing of the simple function $h_A = a_{13} + a_{23}$, where $a_{ij}$  are the corresponding coefficients of the matrix $A_1$ in  \eqref {eq-B0}, divides the conformally harmonic maps (with nilpotency condition for $B_1$) into  two groups, namely those which do not contain a constant light-like vector and those which do contain a constant lightlike vector respectively.
Harmonic maps in the first group always come from some conformal map, while those in the second group not always correspond to conformal maps, but if they do, then the corresponding surface is conformal to a minimal surface in $\R^{n+2}$ and exactly all minimal surfaces in $\R^{n+2}$ are in this group.
 In Theorem \ref{th-potential-light}, we quote a result of
\cite{Wang-Min} which shows that all these harmonic maps correspond exactly to ``normalized potentials'' of a very specific special form. Therefore, if one wants to construct Willmore surfaces not conformally equivalent to minimal surfaces in  $\R^{n+2}$, one can do this by avoiding those normalized potentials of a very special form (see \eqref{eq-w-minimal}).
 Altogether one obtains a procedure which permits to construct all Willmore surfaces in $S^{n+2}$, which are not conformally equivalent to minimal surfaces in  $\R^{n+2}$, in a, at least theoretically, satisfactory way from normalized potentials.

In Theorem \ref{th-Willmore-harmonic} we show that the local existence of a Willmore surface describing a given conformally harmonic map
already implies a unique global existence.
 \end{enumerate}

Our second main result is a loop group description of the harmonic maps which correspond to Willmore surfaces. This includes two decomposition theorems (Iwasawa and Birkhoff) for the corresponding loop groups, the DPW procedure for our harmonic maps, as well as a (generalized) DPW construction for harmonic two-spheres (Theorem \ref{th-potential-sphere}).

The third main result  relates a harmonic map into a non-compact inner symmetric space to a harmonic map  into a compact symmetric space as follows (See Section 5 for details).
\begin{theorem}\label{thm-noncompact}
  Let $f: \tilde{M}\rightarrow \tilde{G}/\tilde{K}$  be a harmonic map from a simply connected Riemann surface $\tilde M$ into an inner, non-compact, symmetric space $\tilde{G}/\tilde{K},$  where
 $\tilde{G}$ is chosen to be simply connected.
Then there exists
a maximal compact Lie subgroup $\tilde{U}$ of
$\tilde{G}^\mathbb{C}$ satisfying $( \tilde{U} \cap \tilde{K}^{\mathbb{C}})^{\mathbb{C}}=\tilde{K}^{\mathbb{C}}$
and there exists a  harmonic map  $f_ {\tilde{U}}: \tilde{M} \rightarrow  \tilde{U}/ (\tilde{U} \cap \tilde{K}^{\mathbb{C}})$
into the compact, inner  symmetric space $\tilde{U}/ (\tilde{U} \cap \tilde{K}^{\mathbb{C}} )$ which has the same normalized potential as $f$.
The map $f_{\tilde{U}}$ is induced from $f$ via the Iwasawa decomposition of  the extended frame $F$ of $f$ relative to $\tilde{U}$.

Conversely, let $h_{\tilde{U}}:\D\rightarrow\tilde{U}/ (\tilde{U} \cap \tilde{K}^{\mathbb{C}} )$ be a harmonic map from the unit disk $\D$ with $h_{\tilde{U}}|_{z=0}=e$. Then there exists a neighbourhood $\D_0\subset \D$ of $0$ and a harmonic map $h:\D_0\rightarrow \tilde{G}/\tilde{K}$ which has the same normalized potential as $h_{\tilde{U}}$. The map $h$ is induced from $h_{\tilde{U}}$ via the Iwasawa decomposition of the extended frame $H$ of  $h_{\tilde{U}}$  relative to $\tilde{G}$.
\end{theorem}
Note that the difference in the above two directions is due the fact that the Iwasawa decomposition for the loop groups of  compact Lie groups is  global while the Iwasawa decomposition for loop groups of non-compact Lie groups is not global.

This theorem turns out to permit to apply certain well-known results for finite uniton harmonic maps into compact inner symmetric spaces \cite{BuGu} directly to harmonic maps into non-compact inner  symmetric spaces, which basically is  the  content of the follow-up paper \cite{DoWa2}.

As a consequence,  in  \cite{Wang-iso} a concrete example of a Willmore 2-sphere in $S^6$ is computed, which gives a negative answer to an open problem of Ejiri stated at  the end of his paper \cite{Ejiri1988}.
Here is the example.
\begin{theorem} \label{thm-example}(\cite{Wang-iso}) Let
\[\eta=\lambda^{-1}\left(
                      \begin{array}{cc}
                        0 & \hat{B}_1 \\
                        -\hat{B}_1^tI_{1,3} & 0 \\
                      \end{array}
                    \right)dz,\ ~ \hbox{ with } ~\ \hat{B}_1=\frac{1}{2}\left(
                     \begin{array}{cccc}
                       2iz&  -2z & -i & 1 \\
                       -2iz&  2z & -i & 1 \\
                       -2 & -2i & -z & -iz  \\
                       2i & -2 & -iz & z  \\
                     \end{array}
                   \right).\]
Then the associated family of Willmore two-spheres $y_{\lambda}$, $\lambda\in S^1$, corresponding to $\eta$, is \begin{equation}\label{example1}
\begin{split}y_{\lambda}&=\frac{1}{ \left(1+r^2+\frac{5r^4}{4}+\frac{4r^6}{9}+\frac{r^8}{36}\right)}
\left(
                          \begin{array}{c}
                            \left(1-r^2-\frac{3r^4}{4}+\frac{4r^6}{9}-\frac{r^8}{36}\right) \\
                            -i\left(z- \bar{z})(1+\frac{r^6}{9})\right) \\
                            \left(z+\bar{z})(1+\frac{r^6}{9})\right) \\
                            -i\left((\lambda^{-1}z^2-\lambda \bar{z}^2)(1-\frac{r^4}{12})\right) \\
                            \left((\lambda^{-1}z^2+\lambda \bar{z}^2)(1-\frac{r^4}{12})\right) \\
                            -i\frac{r^2}{2}(\lambda^{-1}z-\lambda \bar{z})(1+\frac{4r^2}{3}) \\
                            \frac{r^2}{2} (\lambda^{-1}z+\lambda \bar{z})(1+\frac{4r^2}{3})  \\
                          \end{array}
                        \right) \\
  \end{split}
\end{equation}
with $r=|z|$.
  Moreover, $y_{\lambda}:S^2\rightarrow S^6$ is a Willmore immersion in $S^6$, which is non S-Willmore, full, and totally isotropic. In particular, $y_\lambda$ does not have any branch points.

Note that all the surfaces $y_{\lambda}$, $\lambda\in S^1$, are isometric to each other by rotations by matrices of $SO(7)$, since $y_{\lambda}=D_{\lambda} \cdot y_1,$ with $y_1=y_{\lambda}|_{\lambda=1}$ and
\[ D_{\lambda}=\left(
    \begin{array}{ccccccc}
      I_3 &   0 & 0 & 0 & 0 \\
      0 &   \frac{\lambda+\lambda^{-1}}{2} &  \frac{\lambda-\lambda^{-1}}{-2i} & 0 & 0 \\
      0 &   \frac{\lambda-\lambda^{-1}}{2i} &  \frac{\lambda+\lambda^{-1}}{2} & 0 & 0 \\
      0 &  0 & 0& \frac{\lambda+\lambda^{-1}}{2} &  \frac{\lambda-\lambda^{-1}}{-2i} \\
      0 &   0 & 0&  \frac{\lambda-\lambda^{-1}}{2i} &  \frac{\lambda+\lambda^{-1}}{2} \\
    \end{array}
  \right)\in SO(7).\]
\end{theorem}
For the meaning of $\eta$ and $\hat B_1$ we refer to Section 3 and Section 4 below. It is a straightforward computation to verify that $y$ in \eqref{example1} has the desired properties. We refer to \cite{Wang-iso} for detailed proofs and further discussions. This result and the characterization of the normalized potentials of all Willmore 2-spheres in $S^{n+2}$ indicates that our approach is workable for the study of global geometry of Willmore surfaces in terms of the DPW method for their conformal Gauss map.

Finally we would like to point out that by permitting extended frames to have (up to two) singular points  it is possible to apply the usual loop group formalism (basically without any changes) to the case  of
harmonic maps from $S^2$ to any inner symmetric space (see Theorem \ref{th-potential-sphere}). Hence the results of this paper (unless explicitly excluded)  also make sense for the case $\tilde{M} = S^2$.

\subsection{Organization of the paper}

This paper is organized as follows: in Section 2 we recall the moving frame treatment of Willmore surfaces, following the method of \cite{BPP}, relating a Willmore surface to its conformal Gauss map. We also briefly compare our treatment with H\'{e}lein's framework \cite{Helein}.

Then we introduce the basic facts about harmonic maps and apply them, in Section 3,  to describe the conformal Gauss maps.
To be concrete, we describe a global correspondence
between Willmore surfaces and conformally   harmonic maps of a special type (See Theorem \ref{th-Willmore-harmonic-U} and Theorem \ref{th-Willmore-harmonic}).

 In Section 4, we first recall the DPW method for harmonic maps into symmetric spaces. The Birkhoff and Iwasawa Decomposition Theorems concerning our non-compact groups are presented  in Section 4.1. And the existence of normalized potentials for harmonic two-spheres is provided as well (Theorem \ref{th-potential-sphere}).
It is important to note here that the normalized potential and the meromorphic frame also make sense globally for Willmore 2-spheres, although the extended frame (in this case and only in this case) requires to admit some singularities globally.

 In Section 5 we show that harmonic maps into a non-compact  inner symmetric space induce  harmonic maps into their compact dual inner symmetric space (Theorem \ref{thm-noncompact}).

Section 6 is devoted to some applications of our main results.
 This includes a description of the potentials corresponding to Willmore surfaces which are conformally equivalent to minimal surfaces in $\mathbb{R}^{n+2}$. We also show in this section that isotropic Willmore surfaces in $S^4$ are Willmore surfaces of finite uniton type. Moreover, homogeneous Willmore surfaces admitting a transitive abelian group action are characterized and some examples of homogeneous Willmore tori  are discussed here.

In Appendix A we firstly show that there are two open ``big Iwasawa cells"  for the twisted loop groups used in this paper. Secondly we show that, in our case, in the complexified stabilizer group $K^{\C}$
there exists a solvable subgroup $S$ such that the group multiplication map
$ K \times S \rightarrow K\cdot S $ is a diffeomorphism  onto an open subset of  $K^\C$.
Finally, Appendix B ends this paper with a proof of  Theorem \ref{normalizationlemma}.

\section{Willmore surfaces in $S^{n+2}$}

In \cite{BPP}, a natural and simple treatment of the conformal geometry of surfaces in $S^{n+2}$ is presented.
Here  we will use the same set-up and give a description of a conformal surface in $S^{n+2}$ by using the Maurer-Cartan form of some lift.
We will review first the projective light cone model of the conformal geometry of
$S^{n+2}$ and derive the surface theory in this model. In view of our goal to describe Willmore surfaces, we reformulate the equations on the Lie algebra level and apply it to the well-known description of Willmore surfaces.


\subsection{Conformal surface theory in the projective light cone model}

Let $\mathbb{R}^{n+4}_1$ denote Minkowski space, i.e. we consider $\mathbb{R}^{n+4}$ equipped with the Lorentzian metric
$$<x,y>=-x_{0}y_0+\sum_{j=1}^{n+3}x_jy_j=x^t I_{1,n+3} y,\ \hspace{3mm}  I_{1,n+3}=diag(-1,1,\cdots,1).$$
Let $\mathcal{C}_+^{n+3}= \lbrace x \in \mathbb{R}^{n+4}_{1} |<x,x>=0 , x_0 >0 \rbrace $
denote the forward light cone of $\mathbb{R}^{n+4}_{1}$.
It is easy to see that the projective light cone
$
Q^{n+2}=\{\ [x]\in\mathbb{R}P^{n+3}\ |\ x\in \mathcal{C}_+^{n+3}
\}$
with the induced conformal metric, is conformally equivalent to $S^{n+2}$.
Moreover, the conformal group of
$Q^{n+2}$ is exactly the projectivized orthogonal group $O(1,n+3)/\{\pm1\}$ of
$\mathbb{R}^{n+4}_1$, acting on $Q^{n+2}$ by
\[
T([x])=[Tx],\,\,\ T\in O(1, n+3).
\]

Let $y:M\rightarrow S^{n+2}$ be a conformal immersion from a Riemann surface $M$.
 Let $U\subset M$ be a contractible open subset. A local
lift of $y$ is a map $Y:U\rightarrow \mathcal{C}_+^{n+3} $ such
that $\pi\circ Y=y$. Two different local lifts differ by a scaling,
thus they induce the same conformal metric on $M$.
Here we call $y$ a {\em conformal} immersion, if
$\langle Y_{z},Y_{z} \rangle =0$ and
$\langle Y_{z},Y_{\bar{z}} \rangle >0$ for any local lift $Y$ and any complex
coordinate $z$ on $M$.
Noticing
$\langle Y,Y_{z\bar{z}}\rangle =-\langle Y_{z},Y_{\bar{z}}\rangle <0$, we see that
\begin{equation}
V={\rm Span}_{\mathbb{R}}\{Y,{\rm Re}Y_{z},{\rm Im}Y_{z},Y_{z\bar{z}}\}
\end{equation}
is an oriented rank-4 Lorentzian sub-bundle over $U$, and
there is a natural decomposition
of the oriented trivial bundle $U\times \mathbb{R}^{n+4}_{1}=V\oplus V^{\perp}$, where $V^{\perp}$ is the orthogonal complement of $V$ with an induced natural orientation.
Note that both, $V$ and $V^{\perp}$,  are independent of the choice of $Y$
and $z$, and therefore are conformally invariant. In fact, we obtain a global conformally invariant bundle decomposition
$M\times \mathbb{R}^{n+4}_{1}=V\oplus V^{\perp}$. For any $p\in M$, we denote by $V_p$ the fiber of $V$ at $p$. And
the complexifications of  $V$ and $V^{\perp}$ are denoted by $V_{\mathbb{C}}$ and
$V^{\perp}_{\mathbb{C}}$ respectively.

Since $Y$ takes values in the forward light cone $\mathcal{C}_+^{n+3}$, we are only
interested in conformal transformations which are contained in  $SO^+(1,n+3)$.

Fixing a local coordinate $z$ on $U$, there exists a unique local lift $Y$ in $\mathcal{C}^{n+3}$ satisfying
$|{\rm d}Y|^2=|{\rm d}z|^2$, called the canonical lift (with respect
to $z$).
Clearly, for a canonical lift we have
$\langle Y_{z},Y_{\bar{z}}\rangle = \frac{1}{2}.$

 Given a canonical lift $Y$ we choose the frame $\{Y,Y_{z},Y_{\bar{z}},N\}$ of
$V_{\mathbb{C}}$, where $N$ is the uniquely determined section of $V$ over $U$ satisfying
\begin{equation}\label{eq-N}
\langle N,Y_{z}\rangle=\langle N,Y_{\bar{z}}\rangle=\langle
N,N\rangle=0,\langle N,Y\rangle=-1.
\end{equation}
Note that $N$ lies in the forward light cone $\mathcal{C}_+^{n+3}$
 and that
$N\equiv 2Y_{z\bar{z}}\!\!\mod Y$ holds.

Next we define \emph{the conformal Gauss map} of $y$.

\begin{definition} \label{def-gauss} $($\cite{Bryant1984,BPP,Ejiri1988,Ma}$)$
Let $y:M\to S^{n+2}$ be a conformal immersion from a Riemann surface $M$. The  \emph{conformal Gauss map} of $y$ is defined by
\begin{equation}\begin{array}{ccccc}
                 Gr : & M &\rightarrow&
Gr_{1,3}(\mathbb{R}^{n+4}_{1}) &= SO^+(1,n+3)/SO^+(1,3)\times SO(n)\\
                \ & p\in M & \mapsto & V_p &\ \\
                \end{array}
\end{equation}
Moreover, let $Y$ be the canonical
lift of $y$ with respect to a local coordinate $z = u + i v$. Embedding $Gr_{1,3}(\mathbb{R}^{n+4}_{1})$ into the exterior product $\Lambda^4\mathbb{R}^{n+4}_1$,
we have
\[
Gr=Y\wedge Y_{u}\wedge Y_{v}\wedge N=-2i\cdot Y\wedge Y_{z}\wedge
Y_{\bar{z}} \wedge N
\]
where $N$ is the frame vector
determined in \eqref{eq-N}.

Note that in this case, the $4-$dimensional Lorentzian subspace $\hbox{Span}_{\R}\{Y,N,Y_u,Y_v\}$ has a well-defined natural orientation. This implies that $Gr$ maps to  the symmetric space $SO^+(1,n+3)/SO^+(1,3)\times SO(n)$ instead of $SO^+(1,n+3)/S(O^+(1,3)\times O(n))$, which is usually used in the literature.
Note that $Gr$ depends on the conformal immersion $y$ as well as the chosen complex structure of the Riemann surface $M$. We will therefore sometimes write $Gr = Gr_y$ to emphasize this, in particular, since the latter fact is
usually not pointed out in the literature.
\end{definition}

Given a (local) canonical lift $Y$ we note that $Y_{zz}$ is orthogonal to
$Y$, $Y_{z}$ and $Y_{\bar{z}}$. Therefore there exists a complex valued function $s$
and a section $\kappa\in \Gamma(V_{\mathbb{C}}^{\perp})$ such that
\begin{equation}
Y_{zz}=-\frac{s}{2}Y+\kappa.
\end{equation}
This defines two basic invariants of $y$:
$\kappa$, called \emph{the conformal Hopf differential} of $y$ ,
and $s$, called \emph{the Schwarzian} of $y$.
Clearly, $\kappa$ and $s$ depend on the
coordinate $z$ (for a more detailed discussion, see \cite{BPP,Ma}).
Let $D$
denote the $V_{\mathbb{C}}^{\perp}$ part of the natural connection of
$\mathbb{C}^{n+4}$. Then for any section $\psi\in
\Gamma(V_{\mathbb{C}}^{\perp})$ of the normal bundle $V_{\mathbb{C}}^{\perp}$
and any (local) canonical lift $Y$ of some conformal immersion $y$ into $S^{n+2}$
we obtain
the structure equations (\cite{BPP}, \cite{Ma2006}):
\begin{equation}\label{eq-moving}
\left\{\begin {array}{lllll}
Y_{zz}=-\frac{s}{2}Y+\kappa,\\
Y_{z\bar{z}}=-\langle \kappa,\bar\kappa\rangle Y+\frac{1}{2}N,\\
N_{z}=-2\langle \kappa,\bar\kappa\rangle Y_{z}-sY_{\bar{z}}+2D_{\bar{z}}\kappa,\\
\psi_{z}=D_{z}\psi+2\langle \psi,D_{\bar{z}}\kappa\rangle Y-2\langle
\psi,\kappa\rangle Y_{\bar{z}}.
\end {array}\right.
\end{equation}
For these structure equations the integrability conditions are
the conformal Gauss, Codazzi and Ricci equations respectively (\cite{BPP}, \cite{Ma2006}):
\begin{equation}\label{eq-integ}
\left\{\begin {array}{lllll} \frac{1}{2}s_{\bar{z}}=3\langle
\kappa,D_z\bar\kappa\rangle +\langle D_z\kappa,\bar\kappa\rangle,\\
{\rm Im}(D_{\bar{z}}D_{\bar{z}}\kappa+\frac{\bar{s}}{2}\kappa)=0,\\
R^{D}_{\bar{z}z}=D_{\bar{z}}D_{z}\psi-D_{z}D_{\bar{z}}\psi =
2\langle \psi,\kappa\rangle\bar{\kappa}- 2\langle
\psi,\bar{\kappa}\rangle\kappa.
\end {array}\right.
\end{equation}

Choosing an oriented orthonormal frame $\{\psi_j,j=1,\cdots,n\}$ of the normal bundle $V^{\perp}$ over $U$,
we can write the normal connection as $D_z\psi_j=\sum_{l=1}^{n}b_{jl}\psi_l$ with $b_{jl}+b_{lj}=0.$
Then, the conformal Hopf differential $\kappa$ and its derivative $D_{\bar{z}}\kappa$ are of the form
\begin{equation} \label{defkappabeta}
\kappa=\sum_{j=1}^{n}k_j\psi_j,\ D_{\bar{z}}\kappa=\sum_{j=1}^{n}\beta_j\psi_j,\
\hbox{with }\beta_j=k_{j\bar{z}}- \sum_{j=1}^{n}\bar{b}_{jl}k_l,\ j=1,\cdots,n.\end{equation}
Finally, set
$\phi_1=\frac{1}{\sqrt{2}}(Y+N),\ \phi_2=\frac{1}{\sqrt{2}}(-Y+N),\ \phi_3= Y_z+Y_{\bar{z}},\ \phi_4=i(Y_z-Y_{\bar{z}}),\   k^2=\sum_{j=1}^{n}|k_j|^2,$
and set
\begin{equation}\label{F}
F:=\left(\phi_1,\phi_2,\phi_3,\phi_4,\psi_1,\cdots,\psi_n\right).
\end{equation}
\begin{proposition}\label{frame} Let $y:M\rightarrow  S^{n+2}$ be a conformal immersion
and $Y$ its canonical lift over the open contractible set $U \subset M$.
Then the frame $F$ attains values in $SO^+(1,n+3)$,
and  the Maurer-Cartan form $\alpha=F^{-1}dF$ of $F$ is of the form $$\alpha=\left(
                   \begin{array}{cc}
                     A_1 & B_1 \\
                     B_2 & A_2 \\
                   \end{array}
                 \right)dz+\left(
                   \begin{array}{cc}
                     \bar{A}_1 & \bar{B}_1 \\
                     \bar{B}_2 & \bar{A}_2 \\
                   \end{array}
                 \right)d\bar{z},$$
with
\begin{equation}A_1=\left(
                             \begin{array}{cccc}
                               0 & 0 & s_1 & s_2\\
                               0 & 0 & s_3 & s_4 \\
                               s_1 & -s_3 & 0 & 0 \\
                               s_2 & -s_4 & 0 & 0 \\
                             \end{array}
                           \right),\   A_2=\left(
                             \begin{array}{cccc}
                               b_{11} & \cdots &  b_{n1} \\
                               \vdots& \vdots & \vdots \\
                               b_{1n} &\cdots & b_{nn} \\
                             \end{array}
                           \right),\ \end{equation}
\begin{equation}\label{s}
\left\{\begin{split}&s_1=\frac{1}{2\sqrt{2}}(1-s-2k^2),\ s_2=-\frac{i}{2\sqrt{2}}(1+s-2k^2),\\
&s_3=\frac{1}{2\sqrt{2}}(1+s+2k^2),\ s_4=-\frac{i}{2\sqrt{2}}(1-s+2k^2),\\
\end{split}\right.
\end{equation}
\begin{equation} \label{B1}
B_1=\left(
      \begin{array}{ccc}
         \sqrt{2} \beta_1 & \cdots & \sqrt{2}\beta_n \\
         -\sqrt{2} \beta_1 & \cdots & -\sqrt{2}\beta_n \\
        -k_1 & \cdots & -k_n \\
        -ik_1 & \cdots & -ik_n \\
      \end{array}
    \right),  \ \
B_2=\left(
      \begin{array}{cccc}
        \sqrt{2} \beta_1& \sqrt{2} \beta_1 & k_1& i k_1 \\
        \vdots & \vdots & \vdots & \vdots\\
      \sqrt{2} \beta_n& \sqrt{2} \beta_n & k_n & ik_n \\
      \end{array}
    \right)=-B_1^tI_{1,3}.
 \end{equation}
                       \begin{equation*} \end{equation*}

 Conversely, assume we have some frame $F=(\phi_1,\cdots,\phi_4,\psi_1,\cdots,\psi_{n+4}):U\rightarrow SO^+(1,n+3)$
such that the Maurer-Cartan form $\alpha=F^{-1}dF$  of $F$ is of the above form,
then
\begin{equation}\label{eq-y-from-map}y=\pi_0(F)=:\left[ (\phi_1-\phi_2)\right]
\end{equation}
is a conformal immersion from $U$ into $Q^{n+2}\cong S^{n+2}$ (with canonical lift $\frac{1}{\sqrt{2}}(\phi_1-\phi_2)$).
    \end{proposition}

\begin{remark}
This form of the Maurer-Cartan form is of great importance.
Note that for any point in $M$ the rank of $B_1$ is at most 2.
 Moreover, $B_1 ^t I_{1,3} B_1 = 0$ holds.
The case $B_1\equiv 0$ is equivalent to $y$ being conformally equivalent to a round sphere. For the non-trivial case, $B_1 \neq 0$, a detailed discussion will be given in Theorem 3.10.
\end{remark}


\subsection{Willmore surfaces and harmonicity}

The conformal
Hopf differential $\kappa$ plays an important role in the investigation of Willmore
surfaces.
A direct computation using \eqref{eq-moving} shows that the conformal Gauss map $Gr$ induces a conformally invariant (possibly degenerate )
metric
\[
g:=\frac{1}{4}\langle {\rm d}G,{\rm d}G\rangle=\langle
\kappa,\bar{\kappa}\rangle|dz|^{2}
\]
globally on $M$ (see \cite{BPP}). Note that this metric degenerates at umbilical points of $y$, which are by definition the points where $\kappa$ vanishes (For
Willmore tori with umbilical lines, see \cite{Ba-Bo}).
Nevertheless, this metric can be used to define the Willmore
functional.

\begin{definition}(\cite{BPP}, \cite{Ma2006}) The \emph{ Willmore functional} of $y$ is
defined as four  times the area of M with respect to the metric above:
\begin{equation}\label{eq-W-energy}
W(y):=2i\int_{M}\langle \kappa,\bar{\kappa}\rangle dz\wedge
d\bar{z}.
\end{equation}
An immersed surface $y:M\rightarrow S^{n+2}$ is called a
\emph{Willmore surface}, if it is a critical point of the Willmore
functional with respect to any variation (with compact support) of the map $y:M\rightarrow
 S^{n+2}$.
\end{definition}

Note that the above definition of the Willmore functional coincides
with the usual definition (See \cite{BPP}, \cite{Ma-W1}). It is well-known that Willmore surfaces can be characterized as
follows
\cite{Bryant1984,BPP,Ejiri1988,Wang1998}.

\begin{theorem}\label{thm-willmore} For a conformal immersion $y:M\rightarrow  S^{n+2}$, the following three conditions
are equivalent:
 \begin{enumerate}
\item $y$ is Willmore;

\item The conformal Gauss map $Gr$ is a conformally harmonic map into
$G_{3,1}(\mathbb{R}^{n+3}_{1})$;

\item The conformal Hopf differential $\kappa$ of $y$ satisfies the
``Willmore condition":
\begin{equation}\label{eq-willmore}
D_{\bar{z}}D_{\bar{z}}\kappa+\frac{\bar{s}}{2}\kappa=0
\end{equation}
for any contractible chart of $M$.
 \end{enumerate}
\end{theorem}

\begin{remark}
The definitions and statements given above are correct, as long as the immersion $y$ is sufficiently often differentiable.
Considering the classical Willmore functional one observes that for the functional to make sense we only need that $y$ is contained in the Sobolev space $ W^{2,2}$.
It has been shown by Kuwert and Sch{\"a}tzle (see \cite{Ku-Sch} and references therein) that already under these  weak assumptions it follows that the immersion $y$ is real analytic. As to the analytic properties of Willmore surfaces, we also refer to \cite{Rivi¨¨re} and \cite{Be-Ri} for recent progress.

As a consequence of these analytic properties, the conformal Gauss map of any Willmore immersion is real analytic as well. We will therefore always assume w.l.g. that our immersions all are real analytic.
Since we have shown just above that these Gauss maps are conformally harmonic, in this paper we will exclusively consider real analytic harmonic maps. Note also that a general result of Eells and Sampson \cite{Ee-S} states that harmonic maps from surfaces with  Riemannian metric into Riemannian manifolds are real analytic.
\end{remark}

Now we introduce the notion of the so-called ``dual Willmore surface", which is of essential importance in Bryant's and Ejiri's description of Willmore two-spheres.
\begin{definition}\label{def-dual} (\cite{Bryant1984}, page 399 of \cite{Ejiri1988})
Let $y:M\rightarrow S^{n+2}$ be a Willmore surface and $Gr$  its conformal Gauss map.  Let $M_0$ denote the  set of umbilical points of $y$.  A conformal map $\hat{y}:M\setminus M_0\rightarrow S^{n+2}$ which does not coincide with $y$ is called a ``dual surface" of $y$, if either $\hat y$ reduces to a point or on an open dense subset of $M\setminus M_0$ the map  $\hat y$  is an immersion  and the conformal Gauss map $Gr_{\hat{y}}$ of $\hat{y}$ spans at all points the same subspace as $Gr_y$ (Clearly,  $\hat{y}$  then also is a Willmore surface).
\end{definition}
  There exist many Willmore surfaces (\cite{Ba-Bo}, \cite{Bryant1982}, \cite{Bryant1984}, \cite{Ejiri1988}, \cite{Le-Pe-Pin}, etc.) which admit dual Willmore surfaces.
But in general a Willmore surface in $S^{n+2}$ may not admit a dual surface. To describe Willmore surfaces having dual surfaces,
Ejiri introduced the so-called {\em S-Willmore surfaces}  in \cite{Ejiri1988}.
 For this paper it is convenient to define S-Willmore surfaces as follows (see also \cite{Ma2005}):
\begin{definition} (\cite{Ejiri1988})
A Willmore immersion $y:M\rightarrow S^{n+2}$ is called an S-Willmore surface if on any open subset $U$, away from the umbilical points, the conformal Hopf differential $\kappa$ of $y$ satisfies
$D_{\bar{z}}\kappa || \kappa,\
~\hbox{ i.e.  }\  D_{\bar{z}}\kappa+\frac{\bar{\mu}}{2}\kappa=0~\hbox{ for some } \mu:U\rightarrow \mathbb{C}. $
\end{definition}

\begin{corollary}\label{S-frame} Let $y$ be a Willmore surface which is not totally umbilical. Then $y$ is S-Willmore  if and only if the
(maximal) rank of $B_1$ in Proposition \ref{frame} is $1$.
\end{corollary}

\begin{theorem}\label{thm-dual gauss map} \
 \begin{enumerate}
 \item $($\cite{Bryant1984}, Theorem 7.1 of \cite{Ejiri1988}$)$ A (non totally umbilical) Willmore surface $y$ is S-Willmore  if and only if it has a unique dual (Willmore) surface $\hat y$ on $M\setminus M_0$.

Moreover, if $y$ is S-Willmore, the dual map $\hat y$ can be extended to $M$.
\item $($Theorem 2.9 of \cite{Ma}$)$ If the dual surface $\hat y$ of $y$ is immersed at $p\in M$,  then $Gr_{\hat{y}}(p)$ spans the same subspace as $Gr_y(p)$, but its orientation  is  opposite to the one of $Gr_y$.
\end{enumerate}
\end{theorem}

Note that the following theorem can also be derived from  \cite{Ejiri1988} and \cite{Ma}. We include a proof here to be used in the construction of Willmore surfaces from harmonic maps discussed below.
\begin{theorem} \cite{Ejiri1988}, \cite{Ma2005} \label{Conformal-Gauss-map} Let $f:M\rightarrow SO^+(1,n+3)/SO^+(1,3)\times SO(n)$ be  the (oriented)  conformal Gauss map of a non totally umbilical Willmore surface $y:M\rightarrow S^{n+2}$. If $f$ is also the conformal Gauss map of a surface $\tilde{y}$ on an open subset $U$ of $M$, then
$\tilde{y}$ is equal to $y$ on $U$.
   \end{theorem}
\begin{proof} Since $y$ is not totally umbilical,  the set $M_1$ of points of $M$, where $y$ has a non-zero conformal Hopf differential, is open and dense in $M$.

Suppose that $\tilde{y}$ is different from $y$ on some open  subset $U_1$ of $U$. For any open contractible subset $U_2\subset U_1$ we consider the canonical lift $(Y,z)$ of $y$. Since $y$ and $\tilde{y}$ are different on $U_1$, we can choose on $U_2\cap M_1$ a  (not necessarily canonical) lift of $\tilde{y}$ which is of the form (see also (4.1) in \cite{Ma2006})
\begin{equation}\label{eq-dual-Y}
 \tilde Y=N+\bar\mu Y_z+\mu Y_{\bar{z}}+\frac{|\mu|^2}{2}Y,
\end{equation}
i.e., $\tilde{y}=[\tilde Y]$. By  (2.5),  a straightforward computation shows that (See (4.2), (4.3) of \cite{Ma2006})
\begin{equation}\label{eq-dual-Yz}
\tilde Y_{ z}=\frac{\mu}{2}\tilde Y+\rho(Y_z+\frac{\mu}{2}Y)+\theta(Y_{\bar z}+\frac{\bar\mu}{2}Y)+2D_{\bar{z}}\kappa+ \bar\mu \kappa,
\end{equation}
with $\rho=\mu_{\bar z}-2\langle\kappa,\kappa\rangle,$ $\theta=\mu_{z}-\frac{\mu^2}{2}-s$.

 Next we distinguish two cases.

(a) If $y$ is S-Willmore, there exists a unique $\mu$ such that $2D_{\bar{z}}\kappa+ \bar\mu \kappa=0$ on $U_2\cap M_1$. The Willmore equation now is equivalent to $\theta=0$. Hence we have  $\tilde Y_{ z}=\frac{\mu}{2}\tilde Y+\rho(Y_z+\frac{\mu}{2}Y)\in\hbox{Span}_{\C}\{Y,Y_z,Y_{\bar{z}}, Y_{z\bar z}\}$ and
\[\tilde Y_{ z\bar z }=(\frac{\mu}{2}\tilde Y)_{\bar z}+\rho_{\bar{z}}(Y_z+\frac{\mu}{2}Y)+\rho(Y_z+\frac{\mu}{2}Y)_{\bar{z}}\in\hbox{Span}_{\C}\{Y,Y_z,Y_{\bar{z}}, Y_{z\bar z}\},\]
i.e., $[\tilde Y]$ is exactly the dual Willmore surface of $y$. Now Theorem \ref{thm-dual gauss map} shows that the orientation of $Gr_{\tilde{y}}$ is opposite to the orientation of $Gr_y$. This contradiction shows that $y$ and $\tilde{y}$ can not differ on an open subset of $M$.

(b) If $y$ is  not an S-Willmore  surface, then by definition, there exists no $\mu$ such that $ D_{\bar{z}}\kappa+\frac{\bar\mu}{2}\kappa=0$ holds.
Hence for any $\mu$ we obtain  $\tilde Y_z\not\in\hbox{Span}_{\C}\{Y,Y_z,Y_{\bar{z}}, Y_{z\bar z}\}$. So $\tilde{y}=[\tilde Y]$ can not have $f$ as its conformal Gauss map.
\end{proof}

\begin{remark}\
\begin{enumerate}
\item Note that in terms of Proposition \ref{frame} the condition in (a) above is equivalent to $rank B_1 = 1$ and the condition in (b) is
equivalent to $rank B_1= 2$.

\item In Case (a), although $f$ is not equal to  the conformal Gauss map $Gr_{\tilde y}$ of $\tilde{y}$, if we change the orientation of $U$, then we will have $f=Gr_{\tilde y}$.
\end{enumerate}
\end{remark}

We will say ``the conformal Gauss map contains a constant lightlike vector $Y_0 $'' if there exists a non-zero constant lightlike vector $Y_0$ in $\mathbb{R}^{n+4}_1$ satisfying $Y_0\in V_{p}$ for all $p\in M$.  Then a well-known fact states (one can find a proof in \cite{Helein} or \cite{Ma-W1})
\begin{theorem}\label{minimal} A Willmore surface $y$ is conformally equivalent to a minimal surface in $R^{n+2}$  if and
only if its conformal Gauss map $Gr$ contains a constant lightlike vector.
\end{theorem}

There exist Willmore surfaces which fail to be immersions at some points. To include surfaces of this type, we introduce the notion of {\em Willmore maps} and {\em strong Willmore maps}.

\begin{definition}
A smooth map $y$ from a Riemann surface $M$ to $S^{n+2}$ is called a Willmore map if it is a conformal Willmore immersion on an open dense subset $\hat{M}$ of $M$.  If $\hat{M}$ is maximal, then the points in $M_0=M\backslash \hat{M}$ are called branch points of $y$, at which points $y$ fails to be an immersion.

Moreover, $y$ is called a strong Willmore map if it is a Willmore map and if the conformal Gauss map $Gr: \hat{M}\rightarrow SO^+(1,n+3)/SO^+(1,3)\times SO(n)$ of $y$ can be extended smoothly (and hence real analytically) to $M$.
\end{definition}
\begin{remark}
It is an interesting (open) problem under which conditions a  Willmore map will be a strong Willmore map.
\end{remark}

\begin{example}\
 \begin{enumerate}
 \item Let $x: \hat{M}\rightarrow \mathbb{R}^n$ be a complete minimal surface with finite total curvature. By the classical theory of minimal surfaces, $x$ is a minimal surface with finitely many ends $\{p_1,\cdots, p_r\}$. And by the inverse  of the stereographic projection $x$ becomes a smooth map $y$ from a compact Riemann surface $M=\hat{M}\cup\{p_1,\cdots, p_r\}$ to $S^n$. If all the ends of $x$ are embedded planar ends, $y$ will be a Willmore immersion (\cite{Bryant1984}, \cite{Bryant1988}). If some planar ends $\{p_{j_1},\cdots, p_{j_t}\}$ fail to be embedded, $y$ will be a strong Willmore map with branch  points $\{p_{j_1},\cdots, p_{j_t}\}$. If some ends $\{\tilde{p}_{1},\cdots,\tilde{p}_{l}\}$ fail to be planar, $y$ will be a Willmore map with branch  points $\{\tilde{p}_{1},\cdots,\tilde{p}_{l}\}$ and its conformal Gauss map can not be extended to these points, that is, $y$ is not a strong Willmore map.

\item  Another interesting type of Willmore surfaces consists  of the so-called
isotropic (or super-conformal \cite{Bryant1982}) surfaces in $S^4$,
which can be lifted to holomorphic or anti-holomorphic curves in the twistor bundle of $S^4$.
It is well known (\cite{Ejiri1988,Mus1,Mon,BFLPP}) that such surfaces of genus $0$, together with  all complete minimal surfaces of genus 0 in $\mathbb{R}^4$
with embedded planar ends, provide all the possibilities of Willmore
 immersions  of
two spheres in $S^4$.

\item It is well-known that all minimal surfaces in Riemannian space forms
 can be considered to be  Willmore surfaces in some $S^m$ (\cite{Bryant1984,Weiner}).
These surfaces are basic examples of Willmore surfaces.
Moreover, they are S-Willmore surfaces, see \cite{Ejiri1988,Ma}.
It will therefore be of  particular interest and importance to construct non-S-Willmore surfaces.

\item The first non-minimal Willmore surface was given by Ejiri in \cite{Ejiri1982}.
This non-S-Willmore Willmore surface is a homogeneous torus in $S^5$.
Later, using the Hopf bundle, Pinkall produced a family of non-minimal Willmore tori in $S^3$ via elastic curves (\cite{Pinkall1985}).\\
 \end{enumerate}
\end{example}

\begin{remark} ({\em H\'{e}lein and Ma's harmonic maps})

In \cite{Helein}, H\'{e}lein extended the treatment of Bryant \cite{Bryant1984} to deal with Willmore surfaces in $S^3$ by using
a loop group method \cite{DPW}. He used two kinds of harmonic maps: the conformal Gauss map and the ones he  called  ``roughly harmonic maps". In terms of the notation used here, for a Willmore immersion $y$ in $S^3$ with
a local lift $Y$,  choose $\hat{Y}\in \Gamma (V)$ such that $\langle \hat{Y}, \hat{Y}\rangle=0,$  and $\langle Y, \hat{Y}\rangle=-1.$ Then  H\'{e}lein's roughly harmonic map is defined by
\begin{equation}\mathfrak{H}=Y\wedge \hat{Y}: M\rightarrow Gr_{1,1}(R^{5}_1).
\end{equation}
The reason of the name ``roughly harmonic" is that although $\mathfrak{H}$ may not be harmonic in general, it really provides another family of flat connections (see (36) page 350 in \cite{Helein} for details). If one assumes furthermore that  $\hat{Y}$ satisfies
\begin{equation}\label{H-M} \hat{Y}_{z}\in \hbox{Span}_{\C}\{\hat{Y}, Y, Y_z\} \mod V^{\perp}_{\C} ~ \ \hbox{ for all } z,
\end{equation}
$\mathfrak{H}$ will be a harmonic map. Especially, for a Willmore surface $y$ in $S^3$, there always exists a dual surface (Recall Definition \ref{def-dual} or see \cite{Bryant1984}).  When $\hat{Y}$ is chosen as the lift of the dual surface $\hat{y}$ of $y$, one obtains an interesting harmonic map connecting the original surface and its dual surface.
It is straightforward to generalize H\'{e}lein's notion of  roughly harmonic maps to the case of  Willmore immersions into $S^{n+2}$, since the definition above does not involve the co-dimensional information.
Such natural generalizations following  H\'{e}lein have been worked out in \cite{Xia-Shen}, by using the treatment of \cite{Wang1998} on Willmore submanifolds.

In a different development, in \cite{Ma2006}, Ma considered the generalization of the notion of  a dual surface for a Willmore surface $y$ in $S^{n+2}$.  Let $\hat{Y}\in \Gamma (V)$ with $\langle \hat{Y},\hat{Y}\rangle=0,$ and $ \langle Y, \hat{Y}\rangle=-1.$
Ma found that if $\hat{Y}$ satisfies
\begin{equation} \begin{split}
&\hat{Y}_{z}\in \hbox{Span}_{\C}\{\hat{Y}, Y, Y_z\} \mod V^{\perp}_{\C}  ~~\hbox{and }~~ \langle \hat{Y}_{z},\hat{Y}_{z}\rangle=0~ \ \hbox{ for all } z,\
\end{split}
\end{equation}
then $[\hat{Y}]$ is a new Willmore surface (which may degenerate to a point, see \cite{Ma2006}).
In this case $[\hat{Y}]$ is called  ``an adjoint surface" of $y$. Different from dual surfaces, the adjoint surface $[\hat{Y}]$ is in general  not unique (a detailed discussion on this can be found in \cite{Ma2006}). Moreover, Ma showed that for an adjoint surface
 $[\hat{Y}]$ the map  $\mathfrak{H}=Y\wedge \hat{Y}: M\rightarrow Gr_{1,1}(R^{n+3}_1)$ is a conformal harmonic map.
Obviously this harmonic map defined by Ma is just a special case of H\'{e}lein's harmonic maps used in  \cite{Helein}, \cite{Xia-Shen}, but it may be a particularly natural case (See \cite{Helein2}).

Note that  for H\'{e}lein's harmonic map as well as for Ma's adjoint surfaces, it is usually not possible to prove
 global existence, since the solution of the equation \eqref{H-M} may have singularities. So it does not seem to be easy  to use this approach to discuss the global problem directly. To be more concrete, first we would like to point out that \eqref{H-M} is exactly the Riccati equation \[\mu_{z}-\frac{\mu^2}{2}-s=0.\]
Note that for S-Willmore surfaces, $\mu$ may take the value $\infty$ at some points. Therefore, using the expression of the dual surface \eqref{eq-dual-Y} as given in the proof of Theorem \ref{Conformal-Gauss-map}, at the points where $\mu$ approaches $\infty$, we have $[\hat{Y}]=[Y]$. This implies that the 2-dimensional Lorentzian bundle $Span_{\R}\{Y, \hat{Y}\}$ defined by $Y$ and $\hat{Y}$ reduces to a 1-dimensional lightlike bundle at these points. It stays unknown how to deal with the global properties for this kind of harmonic maps by using H\'{e}lein's approach.
 This is one of the reasons why we use the conformal Gauss map to study Willmore
surfaces, although the computations using H\'{e}lein's harmonic map would perhaps be somewhat easier.

The relation between our approach and H\'{e}lein's is very interesting, in particular in view of Ma's contributions.
We hope to be able to pursue this in a subsequent publication.
\end{remark}


\section{Conformally harmonic maps into $SO^+(1,n+3)/SO^+(1,3)\times SO(n)$}

In this section, we first review the basic description of harmonic maps.
Then we apply it to the harmonic maps into
 $SO^+(1,n+3)/SO^+(1,3)\times SO(n)$. We have seen above that Willmore surfaces are related to
conformally harmonic maps with special Maurer-Cartan forms. Since not every conformally harmonic map
is the conformal Gauss map of some  strong Willmore map, we give a
necessary and sufficient condition for a conformally harmonic map to be the conformal Gauss map of a strong Willmore map.


\subsection{Harmonic maps into the symmetric space G/K}

Let $N=G/K$ be a symmetric space with involution $\sigma: G\rightarrow G$ such
that $G^{\sigma}\supset K\supset(G^{\sigma})_0$. Let $\mathfrak{g}$ and $\mathfrak{k}$ denote the
Lie algebras of $G$ and $K$ respectively. The involution $\sigma$ induces the Cartan decomposition
\[\mathfrak{g}=\mathfrak{k}\oplus\mathfrak{p},\hspace{5mm}  [\mathfrak{k},\mathfrak{k}]\subset\mathfrak{k},
\hspace{5mm} [\mathfrak{k},\mathfrak{p}]\subset\mathfrak{p},\hspace{5mm}
[\mathfrak{p},\mathfrak{p}]\subset\mathfrak{k}.\]
Let $\pi:G\rightarrow G/K$ denote the projection of $G$ onto $G/K$.

Let $f:M\rightarrow G/K$ be a conformally harmonic map from a connected Riemann surface $M$.
Let $U\subset M$ be an open contractible subset.
Then there exists a frame $F: U\rightarrow G$ such that $f=\pi\circ F$ on $U$.
Let $\alpha$ denote the Maurer-Cartan form of $F$. Then $\alpha$ satisfies the Maurer-Cartan equation
and altogether we have
\begin{equation*}F^{-1}d F= \alpha,\ \hbox{ with }\ d \alpha+\frac{1}{2}[\alpha\wedge\alpha]=0.
\end{equation*}
Decomposing $\alpha$ with respect to $\mathfrak{g}=\mathfrak{k}\oplus\mathfrak{p}$ we obtain
\begin{equation*}\alpha=\alpha_{ \mathfrak{k}  } +\alpha_{ \mathfrak{p} },\ \hbox{ with }\
\alpha_{\mathfrak{k  }}\in \Gamma(\mathfrak{k}\otimes T^*M),\
\alpha_{ \mathfrak{p }}\in \Gamma(\mathfrak{p}\otimes T^*M).
\end{equation*}
Moreover
 \begin{equation*}
\left\{\begin{array}{ll}
 & d\alpha_{ \mathfrak{k }}+\frac{1}{2}[\alpha_{ \mathfrak{k }}\wedge\alpha_{ \mathfrak{k }}]+
\frac{1}{2}[\alpha_{ \mathfrak{p }}\wedge\alpha_{ \mathfrak{p }}]=0,\\
&d \alpha_{ \mathfrak{p }}+[\alpha_{ \mathfrak{k }}\wedge\alpha_{ \mathfrak{p }}]=0,
\end{array}\right.
\end{equation*}
holds. Next we decompose $\alpha_{\mathfrak{p}}$ further into the $(1,0)-$part $\alpha_{\mathfrak{p}}'$ and the $(0,1)-$part $\alpha_{\mathfrak{p}}''$,
and set  \begin{equation}\label{eq-harmonic-lam}
\alpha_{\lambda}=\lambda^{-1}\alpha_{\mathfrak{p}}'+\alpha_{\mathfrak{k}}+\lambda\alpha_{\mathfrak{p}}'', \hspace{5mm}  \lambda\in S^1.
\end{equation}

\begin{lemma}\label{lemma-harmonic} $($\cite{DPW}$)$ The map  $f:M\rightarrow G/K$ is harmonic if and only if
\begin{equation}\label{integr}d
\alpha_{\lambda}+\frac{1}{2}[\alpha_{\lambda}\wedge\alpha_{\lambda}]=0,\ \ \hbox{for all}\ \lambda \in S^1.
\end{equation}
\end{lemma}

\begin{definition} Let $f:M\rightarrow G/K$ be harmonic
and $\alpha_{\lambda}$  the differential one-form defined above. Since, by the lemma,
$\alpha_{\lambda}$ satisfies the integrability condition \eqref{integr}, we consider
on any contractible open subset $U \subset M$
the solution $F(z,\lambda)$ to the equation
$$d F(z,\lambda)= F(z, \lambda)\alpha_{\lambda}$$
with the initial condition $F(z_0,\lambda)=e$,
where $z_0$ is a fixed base point $z_0 \in U$, and $e$ is the identity element in $G$. The map $F(z, \lambda)$
is called the {\em extended frame}
 of the harmonic map $f$ normalized at the base point $z=z_0$. Note that $F$ satisfies $F(z,\lambda =1)=F(z)$.
 \end{definition}

Consider $TM^{\mathbb{C}}=T'M\oplus T''M$ and write $d=\partial+\bar{\partial}$.
Then Lemma \ref{lemma-harmonic} can be restated as
\begin{lemma} $($\cite{DPW}$)$ The map $f:M\rightarrow G/K$ is harmonic if and only if
\begin{equation}
\left\{\begin{array}{ll}
 & d\alpha_{\mathfrak{k}}+\frac{1}{2}[\alpha_{\mathfrak{k}}\wedge\alpha_{\mathfrak{k}}]+\frac{1}{2}[\alpha_{\mathfrak{p}}\wedge\alpha_{\mathfrak{p}}]=0,\\
& \bar{\partial}\alpha_{\mathfrak{p}}'+[\alpha_{\mathfrak{k}}\wedge\alpha_{\mathfrak{p}}']=0.
\end{array}\right. \end{equation}
\end{lemma}


\subsection{Harmonic maps into $SO^+(1,n+3)/SO^+(1,3)\times SO(n)$}

 Let's consider again $\mathbb{R}^{n+4}_1$, with the metric introduced in Section 2 and  the group   $SO^+(1,n+3)$ together with its Lie algebra $\mathfrak{so}(1,n+3)$ defined in \eqref{eq-so(1,n+3)}.
Consider the involution
 \begin{equation}
\begin{array}{ll}
\sigma:  SO^+(1,n+3)& \rightarrow SO^+(1,n+3)\\
 \ \ \ \ \ \ \ A&\mapsto D^{-1} A D,
\end{array}\end{equation}
with
 $$D=\left(
         \begin{array}{ccccc}
             -I_{4} & 0 \\
            0 & I_n \\
         \end{array}
       \right),
       $$
       where $I_k$ denotes the $k\times k$ identity matrix. Then the fixed point group $SO^+(1,n+3)^{\sigma}$ of $\sigma$ contains $SO^+(1,3)\times SO(n)$,  where
$SO^+(1,3)$ denotes the connected component of $SO(1,3)$ containing $I$.
Moreover we have $SO^+(1,n+3)^{\sigma}\supset SO^+(1,3)\times SO(n) = (SO^+(1,n+3)^{\sigma})^0$,
where the superscript $0$ denotes  the connected component containing the identity element. On the Lie algebra level we obtain
   \begin{equation*}\begin{split}&\mathfrak{g}=
\left\{\left(
                   \begin{array}{cc}
                     A_1 & B_1 \\
                     -B_1^tI_{1,3} & A_2 \\
                   \end{array}
                 \right)
 |\ A_1^tI_{1,3}+I_{1,3}A_1=0, \ \ A_2+A_2^t=0\right\},\\
&\mathfrak{k}=\left\{\left(
                   \begin{array}{cc}
                     A_1 &0 \\
                     0 & A_2 \\
                   \end{array}
                 \right)
 | \ A_1^tI_{1,3}+I_{1,3}A_1=0,\ \ A_2+A_2^t=0\right\},\ \mathfrak{p}=\left\{\left(
                   \begin{array}{cc}
                   0 & B_1 \\
                     -B_1^tI_{1,3} & 0 \\
                   \end{array}
                 \right)
\right\}.
\end{split}
\end{equation*}

Now let $f: M\rightarrow SO^+(1,n+3)/SO^+(1,3)\times SO(n)$
 be a harmonic map with local frame $F: U\rightarrow SO^+(1,n+3)$ and Maurer-Cartan form $\alpha$ on some contractible open subset $U$ of $M$.
Let $z$ be a local complex coordinate on $U$. Writing
\begin{equation}\label{eq-B0}
    \alpha_{\mathfrak{k}}'=\left(
                   \begin{array}{cc}
                     A_1 &0 \\
                     0 & A_2 \\
                   \end{array}
                 \right) d z,\ \hbox{ and }\ \alpha_{\mathfrak{p}}'=\left(
                   \begin{array}{cc}
                     0 & B_1 \\
                     -B_1^tI_{1,3} & 0 \\
                   \end{array}\right)dz,\end{equation}
the harmonic map equations can be rephrased equivalently in the form
\begin{equation} \label{harmonic}
\left\{\begin{split}
 & Im \left( A_{1\bar{z}} +\bar{A}_1A_1-\bar{B}_1B_1^tI_{1,3}\right)=0,\\
 & Im \left(  A_{2\bar{z}} +\bar{A}_2A_2-\bar{B}_1^tI_{1,3}B_1\right)=0,\\
 &  B_{1\bar{z}} +\bar{A}_1B_1-B_1\bar{A}_2=0.
\end{split}\right. \end{equation}

In Section 2 we have seen that the Maurer Cartan form of the frame associated with a
Willmore surface in $S^{n+2}$ has a very special form. Fortunately, it is easy to detect when such a special form can be obtained by gauging.
We will see  that a crucial part of our paper is the following result.

\begin{theorem}\label{normalizationlemma} Let $B_1$, as in \eqref{eq-B0}, be part of the Maurer-Cartan form of some non-constant harmonic map which is defined on the open, contractible Riemann surface $U$.
Then
\begin{equation}\label{eq-B1-2} B_1^tI_{1,3}B_1=0,
\end{equation}
if and only if
there exists  a real analytic map $\mathbb{A}:  U\rightarrow SO^+(1,3)$ such that
\begin{equation}\label{eq-B1-standard}
\mathbb{A}B_1=\left(
      \begin{array}{ccc}
        \sqrt{2}\beta_1 & \cdots &\sqrt{ 2}\beta_n \\
        -\sqrt{2}\beta_1 & \cdots &  -\sqrt{ 2}\beta_n \\
        -k_1 & \cdots & -k_n \\
        -ik_1 & \cdots & -ik_n \\
      \end{array}
    \right) ,\ \hbox{or }~~
    \mathbb{A}B_1=\left(
      \begin{array}{ccc}
        \sqrt{2}\beta_1 & \cdots &\sqrt{ 2}\beta_n \\
        -\sqrt{2}\beta_1 & \cdots &  -\sqrt{ 2}\beta_n \\
        -k_1 & \cdots & -k_n \\
        ik_1 & \cdots & ik_n \\
      \end{array}
    \right)\ \hbox{ on }    U.
\end{equation}
\end{theorem}
\begin{proof}See Appendix B.\end{proof}
\begin{remark}  Recall  from  \eqref{defkappabeta}  that $k_j\equiv 0$ for all $j$ on an open subset implies $\kappa = 0$ and hence the  surface and its whole associated family  is umbilical. In particular, such surfaces have $B_1\equiv0$ and thus describe surfaces conformally equivalent to a round sphere. Such surfaces are not of interest in this paper and therefore will not be considered.
\end{remark}

\begin{lemma} Let $U$ be a contractible open Riemann surface. Let $f: U\rightarrow SO^+(1,n+3)/SO^+(1,3)\times SO(n)$ be a non-constant harmonic map
 with two frames
$F,\ \hat{F}:U \rightarrow SO^+(1,n+3)$ and Maurer-Cartan forms $\alpha,\hat{\alpha}$.
Using a local complex coordinate $z$ on $U$, we write
$$  \alpha_{\mathfrak{p}}'=\left(
                   \begin{array}{cc}
                     0 & B_1 \\
                     -B_1^tI_{1,3} & 0 \\
                   \end{array}\right)dz, \hspace{1cm} \hat\alpha_{\mathfrak{p}}'=\left(
                   \begin{array}{cc}
                     0 & \hat{B}_1 \\
                     -\hat{B}_1^tI_{1,3} & 0 \\
                   \end{array}\right)dz.$$
Then
$ B_1^t I_{1,3} B_1 = 0~~ \hbox{ if and only if }~~  \hat{B}_1^t I_{1,3} \hat{B}_1 = 0.$

Moreover, under the above condition, we have
$\bar B_1^t I_{1,3} B_1 = 0~~ \hbox{ if and only if }~~  \bar{\hat{B}}_1^t I_{1,3} \hat{B}_1 = 0.$

\end{lemma}

\begin{proof}
Since $F$ and $\hat{F}$ are lifts of the same harmonic map $f$, there exists $F_0=\hbox{diag}(   F_{01},F_{02}): U\rightarrow SO^+(1,3)\times SO(n)
$ such that $\hat{F}=F\cdot F_0$. Then
$\hat{\alpha}=F_0^{-1}\alpha F_0+F_0^{-1}dF_0$, yielding $\hat{B}_1=F_{01}^{-1}B_1F_{02}$. So
$$\hat{B}_1^t I_{1,3} \hat{B}_1= F_{02}^{-1}B_1^tF_{01}^{-1,t}I_{1,3}F_{01}^{-1}B_1F_{02}= F_{02}^{-1}B_1^tI_{1,3}B_1F_{02}.$$
The last statement comes from the fact that $F_{01}$ and $F_{02}$ are real matrices.
\end{proof}

\begin{definition} \label{stronglyconfharm}
Let $f: M\rightarrow SO^+(1,n+3)/SO^+(1,3)\times SO(n)$ be a harmonic map. We call $f$ a {\bf strongly conformally harmonic map} if for any point $p\in M$, there exists a neighborhood $U_p$ of $p$ and a frame $F$ (with Maurer-Cartan form  $\alpha$) of $y$ on $U_p$ satisfying
\begin{equation}\label{eq-Willmore harmonic} B_1^t I_{1,3} B_1 = 0, ~ \hbox{where }~ \alpha_{\mathfrak{p}}'=\left(
                   \begin{array}{cc}
                     0 & B_1 \\
                     -B_1^tI_{1,3} & 0 \\
                   \end{array}\right)dz.\end{equation}
 \end{definition}
\begin{remark}
 Note that for a
 harmonic map to be conformally harmonic, one only needs
\[\langle\alpha_{\mathfrak{p}}'(\frac{\partial}{\partial z}),\alpha_{\mathfrak{p}}'(\frac{\partial}{\partial z})\rangle =
tr \left( \left(\alpha_{\mathfrak{p}}'(\frac{\partial}{\partial z})\right)^tI_{1,3}\alpha_{\mathfrak{p}}'(\frac{\partial}{\partial z})\right)=0.\]
But this follows immediately from the condition on $B_1$ we have assumed. The same condition now shows that $f$ even is strongly conformally harmonic.
\end{remark}

Applying the definition above to Willmore surfaces, we derive
in view of equation \eqref{B1} immediately
\begin{corollary} \label{nilpotency} The conformal Gauss map of a strong Willmore map is a strongly conformally harmonic map.
\end{corollary}

\begin{theorem} \label{lemma-B-12}
Let $U$ be a contractible open Riemann surface with local complex coordinate $z$. Let $f: U\rightarrow SO^+(1,n+3)/SO^+(1,3)\times SO(n)$ be a strongly conformally harmonic map
 with  frame
$F:U \rightarrow SO^+(1,n+3)$ and Maurer-Cartan form $\alpha$. Set
$$\alpha_{\mathfrak{k}}'=\left(
                   \begin{array}{cc}
                     A_1 &0 \\
                     0 & A_2 \\
                   \end{array}
                 \right) d z,\ \ \ \alpha_{\mathfrak{p}}'=\left(
                   \begin{array}{cc}
                     0 & B_1 \\
                     -B_1^tI_{1,3} & 0 \\
                   \end{array}\right)dz.$$
Then $B_1$ has, after some gauge or a gauge and a change of orientation if necessary, the form
\begin{equation}\label{eq-B1}
B_1=\left(
      \begin{array}{ccc}
        \sqrt{2}\beta_1 & \cdots &\sqrt{ 2}\beta_n \\
        -\sqrt{2}\beta_1 & \cdots &  -\sqrt{ 2}\beta_n \\
        -k_1 & \cdots & -k_n \\
        -ik_1 & \cdots & - ik_n \\
      \end{array}
    \right).\end{equation}

    \end{theorem}

\begin{proof}
The conformal harmonicity of $f$ ensures that $B_1$ is a real analytic matrix function (\cite{Ee-S}, \cite{Ku-Sch}).
By Theorem \ref{normalizationlemma}, there exists $A:U \rightarrow SO^+(1,3)$ such that
{\small \begin{equation*}
AB_1=\left(
      \begin{array}{ccc}
        \sqrt{2}\beta_1 & \cdots &\sqrt{ 2}\beta_n \\
        -\sqrt{2}\beta_1 & \cdots &  -\sqrt{ 2}\beta_n \\
        -k_1 & \cdots & -k_n \\
        -ik_1 & \cdots & -ik_n \\
      \end{array}
    \right),\ \hbox{or }
    AB_1=\left(
      \begin{array}{ccc}
        \sqrt{2}\beta_1 & \cdots &\sqrt{ 2}\beta_n \\
        -\sqrt{2}\beta_1 & \cdots &  -\sqrt{ 2}\beta_n \\
        -k_1 & \cdots & -k_n \\
        ik_1 & \cdots & ik_n \\
      \end{array}
    \right)\ \hbox{ on } U.
\end{equation*}}

For the first case, setting
$\hat F= F\cdot\hbox{diag}(A,I_n)
$ we obtain $\hat{B}_1= AB_1$  on $U$.
For the second case, setting $w=\bar{z}$ induces an opposite orientation on $U$ and $U$ is also a Riemann surface for this new coordinate. Now $A\bar{B}_1$ is of the desired form.\end{proof}

For $f$ in the above lemma, assume that  $f(p)=V_p,$ $ p\in M$,
with $V_p\subset\mathbb{R}^{n+4}_1$.
Similar to Theorem \ref{minimal}, we will say that {\em $f$ contains a constant light-like vector $Y_0$ }if there exists a non-zero constant lightlike vector $Y_0$ in $\mathbb{R}^{n+4}_1$ satisfying $Y_0\in V_{p}$ for all $p\in M$.
\begin{theorem}\label{th-Willmore-harmonic-U}
Let $U$ be a contractible open Riemann surface with local complex coordinate $z$. Let $f: U\rightarrow SO^+(1,n+3)/SO^+(1,3)\times SO(n)$ be a harmonic map with  frame
$F=(e_0,\hat{e}_0,e_1,e_2,\psi_1,\cdots,\psi_n):U \rightarrow SO^+(1,n+3)$ and Maurer-Cartan form $\alpha$. Set
\[\alpha_{\mathfrak{k}}'=\left(
                   \begin{array}{cc}
                     A_1 &0 \\
                     0 & A_2 \\
                   \end{array}
                 \right) d z,\ \ \alpha_{\mathfrak{p}}'=\left(
                   \begin{array}{cc}
                     0 & B_1 \\
                     -B_1^tI_{1,3} & 0 \\
                   \end{array}\right)dz.\]
Assume moreover that $f$ is a strongly conformally harmonic map on $U$, that is, assume that $B_1^tI_{1,3} B_1 =0$ holds.
Then w.l.g., $B_1$ has the form \eqref{eq-B1} on $U$, after a change of orientation of $U$ if necessary.
Writing $A_1$ in the form
\begin{equation}
A_1=\left(
         \begin{array}{ccccc}
         0 &  a_{12} &  a_{13} & a_{14} \\
          a_{12}  & 0 &  a_{23} & a_{24} \\
            a_{13} & -a_{23} & 0 & a_{34} \\
            a_{14}  & -a_{24}& -a_{34}  & 0 \\
         \end{array} \right),
\end{equation}
we distinguish two cases:

 \begin{enumerate}[$(a)$]
\item{ $a_{13}+a_{23}\not\equiv0$ on $U$:}
 In this case, there exists an open dense subset $U\backslash U_0$ such that $a_{13}+a_{23}\neq0$ on $U\backslash U_0$ and  $a_{13}+a_{23}=0$ on $ U_0$.  And $f$ is the conformal Gauss map of the unique Willmore surface $y=[e_0-\hat {e}_0]:U\setminus U_0\rightarrow S^{n+2}$. Moreover, $y$ has a conformal extension to $U$,  but is not an immersion on $U_0$.
  \begin{enumerate}[$(1)$]
  \item If the maximal rank of $B_1$ is $2$,  $y$ is not S-Willmore.
\item  If
the maximal rank of $B_1$ is $1$, then $y$ is S-Willmore on $U\setminus U_0$. Moreover, its dual surface $\hat y$ has a conformal extension to $U$, and after changing the orientation of $U$, $f$ will be the conformal Gauss map of the dual surface $\hat y$ on the points where $\hat y$ is immersed.\vspace{2mm}
 \end{enumerate}

\item {    $a_{13}+a_{23}\equiv 0$ on $U$:}
  In this case,  $f$ contains a constant light-like vector.
  \begin{enumerate}[$(1)$]
  \item   If the maximal rank of $B_1$ is $2$, then $f$ can not (even locally)  be the conformal Gauss map of any Willmore map.\vspace{1mm}
\item    If the maximal rank of $B_1$ is $1$, $f$ belongs to one of the following two cases:
  \begin{enumerate}[$(i)$]
  \item There exists some open and dense subset $U^*$ of $U$ such that $f$ is the conformal Gauss map of some uniquely determined Willmore surface $y^*:  U^* \rightarrow S^{n+2}$, either for $U$ with the given complex structure or the conjugate complex structure.

 Moreover, $y^*$ is conformally equivalent to a minimal surface in $R^{n+2}$, and $y^*$ can be extended smoothly to $U$. The conformal map
 $ y^*$ may be branched or un-branched on $U \setminus U^*$.
  \item $f$ reduces to a conformally harmonic map into $SO^+(1,n+1)/SO^+(1,1)\times SO(n)
$ or into $SO(n+2)/SO(2)\times SO(n)$, considered as natural submanifolds of $SO^+(1,n+3)/SO^+(1,3)\times SO(n)$.
 In this case $f$ is not (even locally) the conformal Gauss map of a Willmore map.
 \end{enumerate} \end{enumerate}
\end{enumerate}
\end{theorem}

\begin{proof}

As remarked above, the condition on $B_1$ implies that $f$ is a conformally harmonic, even strongly conformally harmonic map.
Moreover, by Theorem  \ref{normalizationlemma} and Theorem  \ref{lemma-B-12}  we can assume w.l.g. (after changing the complex structure of $U$, if necessary) that $B_1$ has the form of (\ref{eq-B1}).
  The proof of parts (a) and (b) is based on an evaluation of the third of the harmonic map equations (\ref{harmonic}).
Writing this equation in terms of matrix entries we obtain:
 \begin{subequations} \label{harmonic-3}
\begin{align}
& ~~ \beta_{j\bar{z}}-\bar{a}_{12}\beta_j-\frac{\sqrt{2}}{2}(\bar{a}_{13}+i\bar{a}_{14})k_j-\sum_{l=1}^n\beta_l\bar{b}_{lj}=0       \label{harmonic-3:1A} \\
& -\beta_{j\bar{z}}+\bar{a}_{12}\beta_j-\frac{\sqrt{2}}{2}(\bar{a}_{23}+i\bar{a}_{24})k_j+\sum_{l=1}^n\beta_l\bar{b}_{lj}=0,    \label{harmonic-3:1B} \\
 &-k_{j\bar{z}}+\sqrt{2}(\bar{a}_{13}+\bar{a}_{23})\beta_j-i\bar{a}_{34}k_j+\sum_{l=1}^n k_l\bar{b}_{lj}=0, \label{harmonic-3:1C}\\
& -ik_{j\bar{z}}+\sqrt{2}(\bar{a}_{14}+\bar{a}_{24})\beta_j+\bar{a}_{34}k_j+i\sum_{l=1}^n k_l\bar{b}_{lj}=0, \ \ j=1,\cdots,n.\label{harmonic-3:1D}
\end{align}
\end{subequations}
By \eqref{harmonic-3:1A}$+$\eqref{harmonic-3:1B} and \eqref{harmonic-3:1C}$+i\cdot$\eqref{harmonic-3:1B} one obtains
 \[(a_{13}+a_{23}-i(a_{14}+a_{24})) \bar{k}_j=(a_{13}+a_{23}-i(a_{14}+a_{24}))\bar{\beta}_j=0, \ j=1,\cdots,n.\]
Since $f$ is non-constant, not all the $\beta _j$ and all the $k_j$ vanish and we infer
\begin{equation}\label{spec-cond}
a_{13}+a_{23}=i(a_{14}+a_{24}).
\end{equation}

Recall that $F=(e_0,\hat{e}_0,e_1,e_2,\psi_1,\cdots,\psi_n)$. Set $Y_0=\frac{1}{\sqrt{2}}(e_{0}-\hat{e}_{0})$ and $ N_0=\frac{1}{\sqrt{2}}(e_{0}+\hat{e}_{0})$. We have
\begin{equation}\label{eq-ez}
  \left\{\begin{split}
  e_{0z}&=a_{12}\hat e_0+a_{13}e_1+a_{14}e_2+\sqrt{2}\sum_{1\leq j\leq n}\beta_j\psi_j,\
   \hat e_{0z} =a_{12} e_0+a_{23}e_1+a_{24}e_2+\sqrt{2}\sum_{1\leq j\leq n}\beta_j\psi_j,\\
  e_{1z}&=a_{13}e_0+a_{23}\hat e_0-a_{34}e_2+\sum_{1\leq j\leq n}k_j\psi_j, \
  e_{2z} =a_{14}e_0+a_{24}\hat e_0+a_{34}e_1+\sum_{1\leq j\leq n}ik_j\psi_j.\\
  \end{split}\right.
\end{equation}
Then
\[Y_{0z}=\frac{1}{\sqrt{2}}(e_{0}-\hat{e}_{0})_z=-a_{12}Y_{0}+\frac{1}{\sqrt{2}}(a_{13}+a_{23})(e_1-ie_2)\]
follows. Now there are two possibilities:\vspace{2mm}

{\bf Case (a):} $a_{13}+a_{23}$ does not  vanish identically on $U$. Since  $a_{13}+a_{23}$ is real analytic,
there exists some subset $U_0$  of $U$ satisfying the first part of the claim. In this case one verifies directly that $[Y_{0}]$ is conformal from $U$ to $S^{n+2}$ and a conformal immersion from $U \setminus U_0$  into $S^{n+2}$.

Calculating
\[Y_{0z\bar{z}}=|a_{13}+a_{23}|^2N_0 \mod \{Y_{0},e_1,e_2 \}\]
shows that $f$ is the harmonic conformal Gauss map of  $[Y_{0}]$ on $U\setminus U_0$. As a consequence,  $[Y_0]$ is a Willmore surface on $U \setminus U_0$.

 The claims involving the   S-Willmore  condition follow from Corollary \ref{S-frame}. So we only need  to consider Case ($a2$). By Theorem \ref{thm-dual gauss map}, the dual surface $\hat y$ of $[Y_0]$ is defined on $U\setminus U_0$ and has $f$ as its conformal Gauss map if we change the complex structure of $U$,  if $\hat y$ is immersed on an open dense subset of $U\setminus U_0$.

It remains to prove in this case that $\hat{y}$ extends real analytically to all of $U$. By Theorem \ref{thm-dual gauss map}, $\hat y$ is defined on $U\setminus U_0$. To extend $\hat y$ to $U$,  consider a lift $\hat Y$ of $\hat y$. Then  $\hat{Y} = a_0 Y_0 + \check{a}_0 N_0 + a_1 e_1 + a_2 e_2$ for some real valued functions defined on $U\setminus U_0$ .
Since $\hat{Y}$ is not always a multiple of $Y_0$, it follows that $Y_0$ and $\hat{Y}$ are linearly independent on some open (and dense) subset $U'(\subset U\setminus U_0)$ of $U$, and $\check{a}_0 \neq 0$  on $U'$ follows. We can thus consider $\frac{1}{\check{a}_0}\hat Y$ on $U'$. Hence we can assume w.l.g. that
$\hat Y$ is of the form
\begin{equation}\label{hat-Y}
\hat Y =N_{0}+\mu_{1}e_{1}-\mu_{2}e_{2}+\frac{1}{2}|\mu|^2 Y_{0}
\end{equation}
 on $U'$, with $\mu=\mu_1+i\mu_2$ a real analytic complex valued function.
Using \eqref{eq-ez} we obtain
\begin{equation}
\hat Y_{z}=\sum_{j}^{n}(2\beta_j+\bar\mu k_j)\psi_j \mod \{e_{0},\hat{e}_{0},e_1,e_2\} \hbox{ on }  U'.
\end{equation}
Now the duality condition implies  $\hat Y_{z} = 0 \mod \{e_{0},\hat{e}_{0},e_1,e_2\},$ whence $\mu$ is the solution to the equations $\beta_j=-\frac{\bar\mu}{2}k_j,\ j=1,2,\cdots,n$ on $U'$.

To extend $\hat Y$ all we need to do is to extend the definition of $\mu$ real analytically to $U$. To extend $\mu$, we only need to show that $\mu$ has a well-defined (finite or infinite) limit at all points $U^*\subset U$ where $\sum|k_j|^2$ vanishes. From Corollary \ref{corollary-B1-vanish} in Appendix B we know $B_1=h_0\tilde B_1$ with $\tilde B_1$ never vanishing on $U$. So for every $p\in U^*$, there exists some $j$ such that $\beta_j=h_0\tilde \beta_j,\ k_j=h_0\tilde k_j$  with $|\tilde\beta_j(p)|^2+|\tilde{k}_j(p)|^2\neq0$.
 If    $\tilde{k}_j(p) \neq 0$, then the limit $\lim_{\tilde p\rightarrow p}\frac{\beta_j}{k_j}=\frac{\tilde{\beta_j}}{\tilde{k}_j}$ exists, is finite and the quotient function is real analytic in a neighbourhood of $p$.
In this case we can define $[\hat Y ]$ also by \eqref{hat-Y}.
If $\tilde{k}_j=0$, then $\tilde{\beta}_j \neq 0$ and the inverse of the limit is real analytic in a neighbourhood of $p$.
Then we consider  (locally) $\tilde{\hat {Y}}=\frac{1}{|\mu|^2}\hat Y$. By
the argument just given $\tilde{\hat{Y}}$ is well-defined and real analytic in a (possibly small) neighbourhood $U_p$ of $p\in U$ and $[\tilde{\hat{Y}}]=[\hat Y]=\hat{y}$ holds on $U_p\cap U'$.
  \vspace{3mm}

{\bf Case (b):} $a_{13}+a_{23}=0$ on $U$. Then we obtain
$Y_{0z}=-a_{12}Y_{0}.$ By scaling, we may assume that $Y_{0z}=0$ holds. Hence we can assume w.l.g. $a_{12} = 0$ on $U$.\vspace{2mm}

 (b1): Let's assume that the strongly conformally harmonic map $f$ is the conformal Gauss map of some conformal
immersion $\check{y}: U' \rightarrow S^{n+2},$ where $U'$ is some (possibly small) open subset of $U$. In this case, the canonical lift $\check{Y}$ of $\check{y}$ satisfies
(with $z = u + iv$)
\begin{equation} \label{checkY=f}
 f = \hbox{Span}_{\mathbb{R}} \{ \check{Y},\check{Y}_u, \check{Y}_v, \check{Y}_{z \bar{z}}\} = \hbox{Span}_{\mathbb{R}}\{e_{0},\hat{e}_{0},e_1,e_2\} = \hbox{Span}_{\mathbb{R}} \{ Y_0, N_0, e_1, e_2 \}.
\end{equation}
Thus similar to the discussions above, we can thus assume that a lift $Y_{\mu}$ of $\check y$ is of the form
\begin{equation} \label{Y-mu}
Y_{\mu}=N_{0}+\mu_{1}e_{1}-\mu_{2}e_{2}+\frac{1}{2}|\mu|^2 Y_{0}
\end{equation}
with $\mu=\mu_1+i\mu_2$ a real analytic complex valued function defined on an open and dense subset $U''$  of $U$.
Then,  since  $[Y_{\mu}]=\check{y}$ is a conformal surface with  conformal Gauss map $f$, we know that $Y_{\mu}$ satisfies
\begin{equation} \label{b2}
\ Y_{\mu z} \in
\hbox{Span}_{\mathbb{C}}\{e_{0},\hat{e}_{0},e_1,e_2\}.
\end{equation}
Using \eqref{eq-ez} we obtain
\begin{equation} \label{Ymuz1}
 Y_{\mu z}=\sum_{j}^{n}(2\beta_j+\bar\mu k_j)\psi_j \mod \{e_{0},\hat{e}_{0},e_1,e_2\} \hbox{ on }  U''.
\end{equation}
Now the last two equations imply  $Y_{\mu z} = 0 \mod \{e_{0},\hat{e}_{0},e_1,e_2\},$ whence
$\beta_j=-\frac{\bar\mu}{2}k_j,\ j=1,2,\cdots,n$ on $U''$.
From this we infer that on $U''$ we have rank $B_1 \leq 1$ and
since $B_1$ is real analytic we obtain $rank(B_1)\leq 1$ on $U$ and the claim follows.\vspace{2mm}

(b2): Let's assume now that the maximal rank of $B_1$ is $1$. We distinguish two cases according to the vanishing or not of $\sum |k_j|^2$ on $U$.\vspace{2mm}

Case (b2.a): $\sum |k_j|^2 \neq 0$ and rank$(B_1) = 1$ on the open  and dense subset $U^\sharp$ of $U$:  Note that these conditions imply
$\beta_j=-\frac{\bar\mu}{2}k_j, j = 1, \dots n$ on $U^\sharp$ for some function $\mu$ on the open (and dense) subset $U^\sharp$ of $U$.
Now we consider $Y_\mu$ of the form \eqref{Y-mu} with this $\mu$. We will show that it satisfies $Y_{\mu},\ Y_{\mu z},\ Y_{\mu z\bar{z}}\in
\hbox{Span}_{\mathbb{C}}\{e_{0},\hat{e}_{0},e_1,e_2\}.$
It will turn out to be convenient to rewrite $Y_\mu$ in the form
\begin{equation} \label{Y-mu-2}
Y_{\mu}=N_{0}+\mu_{1}e_{1}-\mu_{2}e_{2}+\frac{|\mu|^2}{2}Y_{0}=
N_0+\bar\mu P+\mu \bar{P}+\frac{|\mu|^2}{2}Y_{0},
\end{equation}
with $\mu=\mu_1+i\mu_2$ and  $\ P=\frac{1}{2}(e_1-ie_2)$. Note that for $Y_\mu$ equation (\ref{Ymuz1}) holds. Substituting $\beta_j=-\frac{\bar\mu}{2}k_j$ into the expression (\ref{Ymuz1}) for
$Y_{\mu z}$ we obtain
$Y_{\mu u}, Y_{\mu v}  \in \hbox{Span}_{\mathbb{R}}
\{ e_{0}, \hat{e}_{0},e_1,e_2 \}$.

Moreover,  using $P_{\bar{z}}=-i\bar{a}_{34}P
+\frac{1}{2}(\bar{a}_{13}-i\bar{a}_{14})Y_0,$
which follows from $dF = F \alpha$ together with the special values for the entries of $\alpha$  in the case under discussion, one derives
$ Y_{\mu z\bar{z}}\in \hbox{Span}_\R \{e_{0},\hat{e}_{0},e_1,e_2\}.$
Thus we have shown that for the lightlike vector $Y_\mu$ the relation
$Y_{\mu},\ Y_{\mu z},\ Y_{\mu z\bar{z}}\in
\hbox{Span}_{\mathbb{C}}\{e_{0},\hat{e}_{0},e_1,e_2\} $ holds.

 Next we want to determine, when the map $Y_\mu$
  comes from some conformal map into $S^{n+2}$. It is easy to verify
that  $y^* = [Y_{\mu}]$ is a conformal (whence Willmore) surface with its conformal Gauss map spanning the same vector space as $f$ if and only if $Y_{\mu}$ satisfies
\begin{equation} \label{b2}
Y_{\mu},\ Y_{\mu z},\ Y_{\mu z\bar{z}}\in
\hbox{Span}_{\mathbb{C}}\{e_{0},\hat{e}_{0},e_1,e_2\},
\ \langle Y_{\mu z},Y_{\mu z}\rangle=0,\ \langle Y_{\mu z},Y_{\mu \bar{z}}\rangle>0.
\end{equation}
We have shown above that the first three conditions are satisfied for $Y_\mu$. To verify the fourth condition we evaluate the harmonicity condition for $f$.
Substituting $\beta_j=-\frac{\bar\mu}{2}k_j$
into the  first equation of  \eqref{harmonic-3} and using  $a_{13}+a_{23}=0$, $a_{12} = 0$, and the third equation of \eqref{harmonic-3} we derive
$\mu_z+\sqrt{2}(a_{13}-ia_{14})+ia_{34}\mu=0$ on $U^\sharp$.

Next we observe that the derivative of  $Y_\mu$ also satisfies
\begin{equation}
Y_{\mu z}=(\cdots)P+(\cdots)Y_0 +(\mu_z+\sqrt{2}(a_{13}-ia_{14})+ia_{34}\mu)\bar{P} = (\cdots)P+(\cdots)Y_0,
\end{equation}
whence $\langle Y_{\mu z},Y_{\mu z}\rangle=0$ follows.

As a consequence, the only condition to decide whether $f$  corresponds to a conformal  immersion or not is, whether we can satisfy
 $ \langle Y_{\mu z},Y_{\mu \bar{z}}\rangle>0$, or not.

This naturally leads to the two cases listed in the theorem.

 Case (i). There exists an  open and dense subset  $U^* \subset U^\sharp \subset U$,  where $ \langle Y_{\mu z},Y_{\mu \bar{z}}\rangle>0$ holds.
Then, for every $z\in U^*$, the subspace spanned by $f(z,\bar z)$  coincides with the one of $Gr_{y^*}(z,\bar z)$. Hence either $f(z,\bar z)=Gr_{ y^*}(z,\bar z)$ on $U^*$, or $f(z,\bar z)$ spans the same subspace as
$Gr_{y^*}(z,\bar z)$ and has an orientation which is opposite to the one of $Gr_{ y^*}(z,\bar z)$. In the latter case this just says that after a change of complex structure of $U$ the conformal Gauss map of $y^*$ coincides with $f$.
Moreover, since $f$ and hence $Gr_{y^*}$ contain the  light-like vector $Y_0$,
by the stereographic projection with respect to $Y_0$, $[Y_{\mu}]$ becomes a minimal
surface in $R^{n+2}$ (\cite{Bryant1984}, \cite{Ejiri1988},\ \cite{Mon}, \cite{Ma-W1}).

To extend ${y^*}=[Y_{\mu}]$ to  $U$, we need only to show that $\mu$ has a  well-defined ( finite or infinite)  limit at all points where $\sum|k_j|^2$ vanishes. The argument for this is very similar to the argument given  in Case $(a2)$.

From Corollary \ref{corollary-B1-vanish} in Appendix B we know   $B_1=h_0\tilde B_1$ with $\tilde B_1$ never vanishing on $U$. So for every $p\in U\backslash U^{\sharp}$, there exists some $j$ such that $\beta_j=h_0\tilde \beta_j,\ k_j=h_0\tilde k_j$  with $|\tilde\beta_j(p)|^2+|\tilde{k}_j(p)|^2\neq0$.
 If    $\tilde{k}_j(p) \neq 0$, then the limit $\lim_{\tilde p\rightarrow p}\frac{\beta_j}{k_j}=\frac{\tilde{\beta_j}}{\tilde{k}_j}$ exists, is finite and the quotient function is real analytic in a neighbourhood of $p$.
In this case we can define $[Y_{\mu}]$ also by \eqref{Y-mu-2}.
If $\tilde{k}_j=0$, then $\tilde{\beta}_j \neq 0$ and the inverse of the limit is real analytic in a neighbourhood of $p$.
Then we consider  (locally) $\tilde{Y}_{\mu}=\frac{1}{|\mu|^2}Y_{\mu}$. By
 the argument just given  $\tilde{Y}_{\mu}$ is well-defined and real analytic in a (possibly small) neighbourhood $U_p$ of $p\in U$ and $[\tilde{Y}_{\mu}]=[Y_{\mu}]={y^*}$ holds on $U_p\cap U^{\sharp}$. \vspace{1mm}

  Case (ii).
If  $ \langle Y_{\mu z},Y_{\mu \bar{z}}\rangle=0$ on $U$, then
$Y_{\mu}$ is another constant light-like vector of $f$, linearly independent from $Y_0$. So $Y_0$ and $Y_\mu$ span a constant, real, 2-dimensional Lorentzian subspace. Let $\{\tilde e_0,\tilde{\hat{e}}_0\}$ be an Euclidean oriented orthonormal basis of it and let $\{\tilde e_0,\tilde{\hat{e}}_0, \tilde e_1, \tilde e_2\}$ be an oriented orthonormal basis of $Span_{\R}\{e_0,\hat{e}_0,e_1,e_2\}$. Since $\{e_0,\hat{e}_0\}$ and
$\{\tilde{e}_0,\tilde{\hat{e}}_0\}$ span 2-dimensional Lorentzian subspaces and $\{e_1,e_2\}$ and $\tilde{e}_1, \tilde{e_2}$ span Euclidean subspaces, there exists a transformation in $SO^+(1,3)$ which maps the first basis onto the second one in the given order.
Then the lift $\tilde{F}(z,\bar z)=( \tilde e_0,\tilde{\hat{e}}_0, \tilde e_1, \tilde e_2,\psi_1,\cdots, \psi_n)$ reduces to a map into $SO(n+2)\subset SO^+(1,n+3)$, i.e, $f$ reduces to a harmonic map into $SO(n+2)/SO(2)\times SO(n)
 \subset SO^+(1,n+3)/SO^+(1,3)\times SO(n)$. \vspace{2mm}

Case $(b2.b)$: $\sum |k_j|^2 \equiv 0$: In this case the integrability condition of $\alpha$ implies, in view of the vanishing of both sides on (\ref{b2}), that the submatrix of $\alpha$ with entries $33,34,43,44$ is integrable. Hence, after gauging $\alpha$ by some matrix in $SO(2)$ we can assume w.l.g. that $a_{34} = 0$ holds.
Hence we obtain
 $e_{1z}=\sqrt{2}a_{13}Y_0$ and $\ e_{2z}=\sqrt{2}a_{14}Y_0.$  Since $Y_{0z}=0$, similarly after gauging $\alpha$ by some matrix in $SO^+(1,3)$ we can assume w.l.g. that $a_{13}=a_{14} = 0$ holds, i.e., $e_1$ and $e_2$ are constant.
Hence
$f$ reduces to a map into  $SO^+(1,n+1)/SO^+(1,1)\times SO(n)
 \subset SO^+(1,n+3)/SO^+(1,3)\times SO(n)$.
\end{proof}

\begin{remark} Note that the proof of Theorem \ref{th-Willmore-harmonic-U} shows how one can construct a Willmore surface from a strongly conformally harmonic map $f:U \rightarrow SO^+(1, n+3)/{SO^+(1,3) \times SO(n)}$  or prove that $f$ is not the conformal Gauss map of any conformal map $y: U \rightarrow S^{n+2}$:
\begin{enumerate}
  \item
 Step 1: Choose any frame $F:U \rightarrow SO^+(1, n+3)$.

  \item Step 2: Choose a gauge  $A: U \rightarrow SO^+(1,3) \times SO(n)$ such that the Maurer-Cartan form  $ \tilde{\alpha} = \tilde{F}^{-1} d \tilde{F}$
of the new frame $\tilde{F} = F A $ has the form as stated in (\ref{eq-B1}).
 For this we may need to change the complex structure on $U$.

  \item Step 3: Consider the function $\tilde{h} = \tilde{a}_{13} + \tilde{a}_{23} : U \rightarrow \C$.
\begin{enumerate}
  \item
 Step 3a: $\tilde{h}  \not\equiv 0$ on $U$:

 Then $f$ is the (harmonic) conformal Gauss map of some conformal map $\tilde{y} :\hat{U} \rightarrow S^{n+2}$. More precisely, $\tilde{y} =
\lbrack \frac{1}{\sqrt{2}} ( e_0 - \hat{e}_0 \rbrack,$ where $e_0$ and $\hat{e}_0$ are the first and second column of the extended frame $\tilde{F}$ respectively.

 \item  Step 3b: $\tilde{h}  \equiv 0$ on $U$:

 If the maximal rank of $B_1$ is 2, then $f$ is not (even locally) the conformal Gauss map of any conformal immersion.

  If the maximal rank of $B_1$ is 1,
Consider the function
$p = \sum_{j=1}^n  | \tilde{k} |^2 : U \rightarrow \R _{\geq 0}$.  If $p \equiv 0$ on $U$, then $f$ is not (even locally) the conformal Gauss map of any conformal immersion. If $p \not\equiv 0$ on $U$, we can only obtain (possibly only after changing the complex structure on U) Willmore surfaces in $S^{n+2}$ which are conformal to minimal surfaces in $\R^{n+2}$. If we are not interested in minimal surfaces in $\R^{n+2}, $ then we are done. Otherwise let $\mu =-\frac{2\bar{\tilde{\beta}}_j}{\bar{\tilde{k}}_j}$ on the points $\bar{\tilde{k}}_j\neq 0$. With such a function $\mu$ we consider $\tilde{Y}_\mu$ as in (\ref{Y-mu}).   The stereographic projection of  $\tilde{Y}_\mu$  with center $\tilde{Y}_0 $ yields a minimal surface in $\R^{n+2}$.
  \end{enumerate}
\end{enumerate}
\end{remark}
\begin{corollary} Let $f$ be a  strongly conformally harmonic map as in Theorem \ref{th-Willmore-harmonic-U} which belongs to  Case (a)
as well as to  Case (b). Then, possibly after  changing  the complex structure  of $U$, $f$ is the conformal Gauss map of some minimal surface in $R^{n+2}$ (after putting $R^{n+2}$ conformally into $S^{n+2}$), and vice versa.
\end{corollary}

\begin{corollary} Let $f$ be a strongly conformally harmonic map as in Theorem \ref{th-Willmore-harmonic-U}. Assume that $f$ does not contain any
constant lightlike vector, i.e. $f$ belongs to Case $(a)$.
\begin{enumerate}[(a)]
  \item
If  $rank B_1 = 2$, then there exists a  unique  Willmore, but not S-Willmore,  map  $y:U\rightarrow S^{n+2}$ such that on $U \setminus U_0 $ $y$ is immersed and has $f$ as  its conformal Gauss map.

\item If $rank B_1 = 1$, then there exists a pair of  S--Willmore maps $y,\hat{y}:U\rightarrow S^{n+2}$ which are dual to each other and  such that
\begin{enumerate}[$(b1)$]
  \item on an open dense subset $U_1$ of $U$, $y$ is immersed and has $f$ as  its conformal Gauss map;
  \item on an open dense subset $U_2$ of $U$, $\hat y$ is immersed and has $f$ as  its conformal Gauss map after a change of the orientation of $U$.
  \end{enumerate}
\end{enumerate}\end{corollary}

Ejiri's Willmore torus in $S^5$ (\cite{Ejiri1982}) provides an example for Case (a), and Veronese spheres in $S^{2m}$ (\cite{Mon}) provide examples for Case (b).

\begin{corollary} Let $f$ be a strongly conformally harmonic map as in Theorem \ref{th-Willmore-harmonic-U}. Assume that $f$ contains a
constant lightlike vector. Then either  $f$ does not correspond to any Willmore map,
or
 $f$ corresponds to a Willmore map which is conformally equivalent to a minimal surface into $\R^{n+2}$.
\end{corollary}

\begin{remark}  As stated in the introduction and
also stated explicitly by the corollaries above,
one can divide the set of all strongly harmonic
maps into two groups, one consisting of those
which do not contain a constant lightlike vector
and the other consisting of those which do contain
a constant lightlike vector.
The latter ones will not always represent a conformal Gauss map
of some immersion. But if they do, then the corresponding Willmore surface
is conformally equivalent to a minimal surface in Euclidean space.
Since this type of Willmore surfaces is well known and well investigated,
we are  primarily  interested in the surfaces belonging to
the first group of conformally harmonic maps.
In terms of potentials  these two cases can be distinguished
 very easily,
since (in section 6) we will describe precisely those normalized potentials,
for which the corresponding strongly conformal map contains a constant
lightlike vector. So all we need to do is to make sure our
normalized potential does not have the form stated in Theorem \ref{th-potential-light}.
In particular, we will obtain new Willmore spheres which are not S-Willmore if we assume furthermore that $B_1$ is of rank $2$ and the harmonic map is a map from $S^2$ into  $ SO^+(1,n+3)/SO^+(1,3)\times SO(n)$. Such maps are of finite uniton type, as we will show later, and can be treated by our procedure relatively easily.
\end{remark}

We also have the following
\begin{corollary}
Let $f$ be a strongly conformally harmonic map as in Theorem \ref{th-Willmore-harmonic-U}. Fix the orientation of contractible open set $U$. Then
\begin{enumerate}
\item either  $f$ can not be the conformal Gauss map of any conformal map on any open subset of $U$;
\item or there exists a unique Willmore map $y:U\rightarrow S^{n+2}$ such that $f$ is the conformal Gauss map of $y$ on an open dense subset of $U$.
\end{enumerate}
\end{corollary}
\begin{proof}
If $B_1$ has the form \eqref{eq-B1} on $U$, our corollary follows from  Theorem \ref{th-Willmore-harmonic-U}.

Now assume that $B_1=(v_1,\cdots,v_n)$ has the second form in \eqref{eq-B1-standard} on $U$. Next we distinguish two cases.

Assume that the rank of $B_1$ is $2$ on an open dense subset $U'$ of $U$. Set
\[\mathcal{N} = \C (1,-1,0,0)^t,~~ \mathcal{N}_{\pm}= \C (0,0,1, \pm i)^t.\]
Then $Span_{\C}\{v_1,\cdots,v_n\}=\mathcal{N}\oplus\mathcal{N}_-$ on $U'$. We claim that
$f$ can not be the conformal Gauss map of any conformal map on any open subset of $U$. Otherwise
  we have some Willmore immersion $y:U''\subset\rightarrow S^{n+2}$ with $f$ as its conformal Gauss map. By Proposition \ref{frame} there exists $\tilde F$ such that its Maurer-Cartan form $\tilde B_1$ has the form \eqref{eq-B1} on $U''$. This means that there exists some $L=\hbox{diag}(L_1,L_2):U''\rightarrow SO^+(1,3)\times SO(n)$ such that $\tilde F=FL$. So $\tilde B_1=(\tilde{v}_1,\cdots,\tilde{v}_n)=L_1 B_1L_2^{-1}$. Since  $Span_{\C}\{\tilde{v}_1,\cdots,\tilde{v}_n\}=\mathcal{N}\oplus\mathcal{N}_+$ on $U''$, $L_1$ maps $\mathcal{N}\oplus\mathcal{N}_-$ into $\mathcal{N}\oplus\mathcal{N}_+$. But this is not possible.

Now assume that the maximal rank of $B_1$ is $1$ on an open subset $U'$ and hence on the whole $U$. Applying Lemma \ref{lemma-null} as in the proof of Theorem \ref{normalizationlemma} in Appendix B, one has that there exists some frame $\hat F$ of $f$ on $U$ such that $\hat B_1$ has the form \eqref{eq-B1} on $U$. Our corollary follows again from  Theorem \ref{th-Willmore-harmonic-U}.

\end{proof}
Since any two points of a surface $M$ are contained in a contractible open subset of $M$ the corollary yields straightforwardly the following
\begin{theorem}\label{th-Willmore-harmonic}  Let $f: M\rightarrow SO^+(1,n+3)/SO^+(1,3)\times SO(n)$ be a non-constant strongly conformally harmonic map from a connected Riemann surface $M$. If
on a contractible open subset $U\subset M$  $f$ is the conformal Gauss map of some Willmore immersion $\tilde y:U\rightarrow S^{n+2}$, then there exists a unique conformal (Willmore) map $y:M\rightarrow S^{n+2}$ such that $f$ is the conformal Gauss map of $y$ on an open dense subset $M_1$ of $M$ and $y|_{U}=\tilde y$.
\end{theorem}


\section{Loop group theory for harmonic maps}

In this section we start by collecting the basic definitions and the basic
decomposition theorems for loop groups (\cite{DPW}, \cite{Wu}, \cite{Ba-Do}).
Next we recall the DPW method for the construction of harmonic maps.
Since we are mainly interested in strongly conformally harmonic maps,
we characterize strongly conformal harmonicity in terms of normalized potentials which satisfy the additional
condition  $B_1^t I_{1,3} B_1 = 0$.


\subsection{Loop groups and decomposition theorems}

Let $G$ be a connected real Lie group and $G^\C$ its complexification
(For details on complexifications, in particular of semi-simple Lie groups, see \cite{Hochschild}).

Let $\sigma$ denote an inner involution of $G$ and $K$ a closed subgroup satisfying $(Fix^\sigma(G))^0 \subset K \subset Fix^\sigma(G)$.
Then $\sigma$ fixes $\mathfrak{k} = Lie K$.
The extension of $\sigma$  to an involution of $G^\C$ has
 $\mathfrak{k}^\C$  as  its fixed point algebra.
By abuse of notation we put $K^\C = Fix^\sigma(G^\C)$.
It is known that $K^\C$ is in general not connected.

Here are  the basic definitions about loop groups which we will apply to any  Lie group $G$ and its  complexification
$G^\C$, assuming an inner involution $\sigma$ of these groups is given:
 \begin{equation*}
\begin{array}{llll}
\Lambda G^{\mathbb{C}}_{\sigma} ~&=\{\gamma:S^1\rightarrow G^{\mathbb{C}}~|~ ,\
\sigma \gamma(\lambda)=\gamma(-\lambda),\lambda\in S^1  \},\\[1mm]

\Lambda G_{\sigma} ~&=\{\gamma\in \Lambda G^{\mathbb{C}}_{\sigma}
|~ \gamma(\lambda)\in G, \hbox{for all}\ \lambda\in S^1 \},\\[1mm]
\Omega G_{\sigma} ~&=\{\gamma\in \Lambda G_{\sigma}|~ \gamma(1)=e \},\\[1mm]

 \Lambda^{-} G^{\mathbb{C}}_{\sigma}  ~&=
\{\gamma\in \Lambda G^{\mathbb{C}}_{\sigma}~
|~ \gamma \hbox{ extends holomorphically to } |\lambda|>1 \cup\{\infty\} \},\\[1mm]

\Lambda_{*}^{-} G^{\mathbb{C}}_{\sigma} ~&=\{\gamma\in \Lambda G^{\mathbb{C}}_{\sigma}~
|~ \gamma \hbox{ extends holomorphically to } |\lambda|>1\cup\{\infty\},\  \gamma(\infty)=e \},\\[1mm]

\Lambda^{+} G^{\mathbb{C}}_{\sigma} ~&=\{\gamma\in \Lambda G^{\mathbb{C}}_{\sigma}~
|~ \gamma \hbox{ extends holomorphically to } |\lambda|<1 \},\\[1mm]

\Lambda_{L}^{+} G^{\mathbb{C}}_{\sigma} ~&=\{\gamma\in
\Lambda G^{\mathbb{C}}_{\sigma}~|~   \gamma(0)\in L \},\\[1mm]
\end{array}\end{equation*}
where $L$ denotes some Lie subgroup of $K^\C$, prescribing in which group the leading term is supposed to be contained in.
If $L = (K^\C)^0$, then we write $\Lambda_{\mathcal{C}}^{\pm} G^{\mathbb{C}}_{\sigma} $. These groups are connected.

For the decomposition theorems quoted below (which are of crucial importance for the applicability of the loop group method explained below) we need to have some topology on our loop groups.
This can be done in several ways.  We will represent $G^\C$ as a matrix group and will assume that all matrix entries
of  $\Lambda G^{\mathbb{C}}_{\sigma}$  are in the Wiener
algebra of the unit circle.  We thus obtain that $\Lambda G^{\mathbb{C}}_{\sigma}$ is a Banach Lie group.
All other groups discussed in this paper inherit a Banach Lie group structure in a natural way.

With these assumptions we obtain:
\begin{theorem} {\em(Birkhoff Decomposition Theorem for $ (\Lambda {G}_\sigma^\C )^0$)} \label{thm-birkhoff-0}
\begin{enumerate}
\item $ (\Lambda {G}_\sigma^\C )^0= \bigcup \Lambda^{-}_{\mathcal{C}} {G}^{\mathbb{C}}_{\sigma} \cdot \omega \cdot \Lambda^{+}_{\mathcal{C}} {G}^{\mathbb{C}}_{\sigma}$,
where the $\omega$'s are representatives of the double cosets.

\item The multiplication $\Lambda_{*}^{-} {G}^{\mathbb{C}}_{\sigma}\times
\Lambda^{+}_\FC {G}^{\mathbb{C}}_{\sigma}\rightarrow
\Lambda {G}^{\mathbb{C}}_{\sigma}$ is an analytic  diffeomorphism onto the
open and dense subset $\Lambda_{*}^{-} {G}^{\mathbb{C}}_{\sigma}\cdot
\Lambda^{+}_\FC {G}^{\mathbb{C}}_{\sigma}$  of $ (\Lambda G^\C_\sigma)^0$
{\em (big Birkhoff cell)}.
\end{enumerate}
\end {theorem}

\begin{remark} \
\begin{enumerate}
\item For the case of Willmore surfaces we consider  the inner symmetric space $G/K$, where $G = SO^+(1,n+3)$ and $K = SO^+(1,3) \times SO(n)$.
\item The inner involution  $\sigma$ is given by $\sigma = Ad I_{4,n}$ with $ I_{4,n} = diag(I_4, -I_n)$. Then $ Fix^\sigma(G) = S( O^+(1,3) \times O(n))$ and $K$ is the connected component of this fixed point group.

\item For the complexification  of $G$ we obtain
 $G^\C \cong SO(1,n+3,\C)$ and the fixed point group of $\sigma$ in $G^\C$ (called $K^\C$,  by abuse of notation, as above) is $K^\C \cong Fix^\sigma(G^\C) = S(O(1,3,\C) \times O(n,\C))$.  Moreover, we obtain that $(K^\C)^0$ is the complexification of $({Fix}^{\sigma}(G))^0$.

\item For the simply-connected covers $\tilde{G}$ and $\tilde{G}^\C$ of $G$ and  $ G^\C$ respectively we obtain
$\tilde{G} = Spin(1,n+3)^0$ and
$\tilde{G}^\C = Spin(1, n+3,\C)$  (See e.g. \cite{LM}, $(2.35)$  and \cite{Mein}, Proposition 3.1 respectively).

\end{enumerate}
\end{remark}

The discussion so far is related to part one of the loop group method (as outlined in detail below).
Part two of the loop group method requires another decomposition theorem:

\begin{theorem} {\em (Iwasawa Decomposition Theorem for
$(\Lambda G^{\C})_{\sigma} ^0 $\label{gen-Iwasawa})} \label{thm-Iwasawa-0}
\begin{enumerate}
\item
$(\Lambda G^{\C})_{\sigma} ^0=
\bigcup \Lambda G_{\sigma}^0\cdot \delta\cdot
\Lambda_\mathcal{C}^{+} G^{\mathbb{C}}_{\sigma},$
where the $\delta$'s are representatives of the double cosets.

 \item The multiplication $\Lambda G_{\sigma}^0 \times \Lambda_\mathcal{C}^{+} G^{\mathbb{C}}_{\sigma}\rightarrow
(\Lambda G^{\mathbb{C}}_{\sigma})^0$ is a real analytic map onto the connected open subset
$ \Lambda G_{\sigma}^0 \cdot \Lambda_\mathcal{C}^{+} G^{\mathbb{C}}_{\sigma}   = \mathcal{I}^{\mathcal{U}}_e \subset \Lambda G^{\mathbb{C}}_{\sigma}$.
\end{enumerate}

\end{theorem}

For the loop groups involved in the description of Willmore surfaces we add two results:

\begin{theorem}
Consider the setting $G/K = Gr_{1,3}(\mathbb{R}^{n+4}_{1}) = SO^+(1,n+3)/SO^+(1,3)\times SO(n)$.
\begin{enumerate}
\item
 There exist  two different open Iwasawa cells in the connected loop group
$(\Lambda G^{\mathbb{C}}_{\sigma})^0$,  one given by $\delta = e$ and the other one by $\delta = diag(-1,1,1,1,-1,1,1,...,1) $.

\item
There exists a closed, connected, solvable subgroup
$S \subseteq (K^\C)^0$ such that
the multiplication

$\Lambda G_{\sigma}^0 \times \Lambda^{+}_S G^{\mathbb{C}}_{\sigma}\rightarrow
(\Lambda G^{\mathbb{C}}_{\sigma})^0$ is a real analytic diffeomorphism onto the connected open subset
$ \Lambda G_{\sigma}^0 \cdot \Lambda^{+}_S G^{\mathbb{C}}_{\sigma}      \subset  \mathcal{I}^{\mathcal{U}}_e \subset(\Lambda G^{\mathbb{C}}_{\sigma})^0$.
\end{enumerate}
\end{theorem}

 Proofs for the last three theorems are given in  Appendix A.


\subsection{The DPW method and potentials }

 With the loop group decompositions as stated above, we obtain a
construction scheme of harmonic maps from a surface into any real pseudo-Riemannian symmetric space $G/K$ for which the metric is induced from a bi-invariant metric on $G$. All symmetric spaces considered in this paper are of this type.

So far we have mainly discussed Willmore surfaces and the corresponding conformally harmonic maps
defined on some open subset $U$ of $\C$ (or possibly an open subset of some surface $M$).
Since the immersions of interest are conformal, the corresponding surface has a complex structure.
We thus only consider  Riemann surfaces in this paper.
 If $M$ is such a Riemann surface, then its universal cover
$\tilde{M}$ is either $S^2$ or $\C$ or $\E$, the open unit disk in $\C$.
Every harmonic map from $M$ to some  symmetric space $G/K$ induces via composition with the natural projection
a harmonic map from the universal cover $\tilde{M}$ into $G/K$.
Therefore, to start with, we need to consider harmonic maps from $S^2$, $\C$ and $\E$ into $G/K$.

\begin{theorem}{\em(\cite{DPW})}\label{DPW}
Let $\D$ be a contractible open subset of $\C$ and $z_0 \in \D$ a base point.
Let $f: \D \rightarrow G/K$ be a harmonic map with $f(z_0)=eK.$
Then the associated family  $f_{\lambda}$ of $f$ can be lifted to a map
$F:\D \rightarrow \Lambda G_{\sigma}$, the extended frame of $f,$ and we can assume  w.l.g. that $F(z_0,\lambda)= e$ holds.
Under this assumption,
\begin{enumerate}
\item
  The map $F$ takes only values in
$ \mathcal{I}^{\mathcal{U}}_e \subset \Lambda G^{\mathbb{C}}_{\sigma}$.

\item There exists a discrete subset $\D_0\subset \D$ such that on $\D\setminus \D_0$
we have the decomposition
$$F(z,\lambda)=F_-(z,\lambda)\cdot F_+(z,\lambda),
$$where $$ F_-(z,\lambda)\in\Lambda_{*}^{-} G^{\mathbb{C}}_{\sigma}
\hspace{2mm} \mbox{and} \hspace{2mm} F_+(z,\lambda)\in (\Lambda^{+} G^{\mathbb{C}}_{\sigma})^0.$$
Moreover $F_-(z,\lambda)$ is meromorphic in $z \in \D$ and $F_-(z_0,\lambda)=e$ holds and the Maurer-Cartan form $\eta$ of $F_-$
$$\eta= F_-(z,\lambda)^{-1} d F_-(z,\lambda)$$
is a $\lambda^{-1}\cdot\mathfrak{p}^{\mathbb{C}}-\hbox{valued}$ meromorphic $(1,0)-$
form with poles at points of $\D_0$ only.

\item Conversely, any harmonic map  $f: \D\rightarrow G/K$ can be derived from a
$\lambda^{-1}\cdot\mathfrak{p}^{\mathbb{C}}-\hbox{valued}$ meromorphic $(1,0)-$ form $\eta$ on $\D$.

\item  Spelling out the converse procedure in detail we obtain:
Let $\eta$ be a  $\lambda^{-1}\cdot\mathfrak{p}^{\mathbb{C}}-\hbox{valued}$ meromorphic $(1,0)-$ form for which the solution
to the ODE
\begin{equation}
F_-(z,\lambda)^{-1} d F_-(z,\lambda)=\eta, \hspace{5mm} F_-(z_0,\lambda)=e,
\end{equation}
is meromorphic on $\D$, with  $\D_0$ as set of possible poles.
Then on the open set $\D_{\mathcal{I}} = \lbrace z \in \D; F(z,\lambda)
\in \mathcal{I}^{\mathcal{U}} \rbrace$ we
define $\tilde{F}(z,\lambda)$ via the factorization
 $\mathcal{I}^{\mathcal{U}}_e =  ( \Lambda G_{\sigma})^0 \cdot \Lambda_S^{+} G^{\mathbb{C}}_{\sigma}
\subset  \Lambda G^{\mathbb{C}}_{\sigma}$:
\begin {equation}\label{Iwa}
F_-(z,\lambda)=\tilde{F}(z,\lambda)\cdot \tilde{F}_+(z,\lambda)^{-1}.
\end{equation}
 This way one obtains an extended frame
$$ \tilde{F}(z,\lambda)=F_-(z,\lambda)\cdot \tilde{F}_+(z,\lambda)\\$$
of some harmonic map from $\D_{\mathcal{I}} \subset \D$ to $G/K$ satisfying  $\tilde{F}(z_0,\lambda)=e$.

Moreover, the two constructions outlined above  are inverse to each other (on appropriate domains of definition).
\end{enumerate}
\end{theorem}

\begin{definition}{\em(\cite{DPW})}
The $\lambda^{-1}\cdot \mathfrak{p}^{\mathbb{C}}-\hbox{valued}$ meromorphic $(1,0)$
form $\eta$ is called the {\em normalized potential} for the harmonic
map $f$ with the point $z_0$ as the reference point.
\end{definition}

\begin{remark}\
\begin{enumerate}
\item
Note that the normalized potential is uniquely determined once a base point is chosen.
However, if we conjugate a normalized potential by  some $z-$independent element  $k$ of $K$,
then the procedure outlined in the theorem produces a new harmonic map (and correspondingly a new Willmore surface)
which differs from the original one by the rigid motion induced by $k$.
Since we usually do not care about how the harmonic map (or the Willmore surface) sits
in  space, we sometimes use elements of $K$ to simplify or further normalize the normalized potential.
 \item
In the converse procedure, part $(4)$ above, since in our case the symmetric space $G/K$ is not compact,
the Iwasawa splitting \eqref{Iwa} will in general not be possible for all $z \in \D$.
Thus $\tilde{F}$, as well as the harmonic map $\tilde{f}$ will have singularities on $\D$. There are two types of singularities. One type stems from poles in the potential $\eta$ and the other type occurs,
when $F_-$   touches or crosses the boundary of an  open Iwasawa cell (See \cite{B-R-S}).
(There are at least two open Iwasawa cells, as pointed out above).
In the new example of a Willmore sphere in $S^6$  \cite{Wang-iso} (also see Theorem \ref{thm-example}  already mentioned in the introduction)  it happens
that the frame of the harmonic map has  singularities, but the Willmore immersion and hence the harmonic map does not have any singularity
 (\cite{Wang-iso}).  In this case the frame is the product of a matrix without singularities with a singular scalar factor. The projection of the frame to the harmonic map cancels out this singular factor. So the appearance of singularities is only due to the choice of frame.

\end{enumerate}\end{remark}

So far we have only introduced the ``normalized potential". However, in many applications it is much more convenient to use potentials which contain, in their Fourier expansion, more than one power of $\lambda$.
The normalized potential is usually meromorphic in $z$. Since it is uniquely determined, there is no way to change this. However, when permitting many (maybe infinitely many) powers of $\lambda$, then one can obtain potential with holomorphic coefficients,  which will be called {\em holomorphic potential}.

\begin{theorem} \label{pot-hol} Let $\D$ be a contractible open subset of $\C$.
Let $F(z,\lambda)$ be the frame of some harmonic map
into $G/K$. Then there exists some $V_+ \in \Lambda^{+} G^{\mathbb{C}}_{\sigma} $ such that $C(z,\lambda) =
F V_+$ is holomorphic in $z$ and in $\lambda \in \mathbb{C}^*$.
Then the Maurer-Cartan form $\eta = C^{-1} dC$ of $C$ is a holomorphic $(1,0)-$form on $\D$ and it is easy to verify that $\lambda \eta$ is holomorphic for $\lambda \in \C$.
Conversely, any harmonic map  $f: \D\rightarrow G/K$ can be derived from
such a holomorphic $(1,0)-$ form $\eta$ on $\D$ by the same steps as in the previous theorem.
\end{theorem}
The proof can be taken verbatim from the appendix of
\cite{DPW} and will  be omitted here.
 We would like to point out that the proof works for all harmonic maps into any Riemannian or pseudo-Riemannian symmetric space (actually, even more generally, for primitive harmonic maps into $k-$symmetric spaces).
In particular, the proof is independent of the results of the previous sections.

\begin{remark}\
\begin{enumerate}
 \item
Again, since the Iwasawa splitting is not global in our case, even when starting with a holomorphic potential, the resulting harmonic map will generally have singularities.

\item
Let $\eta_1$ and $\eta_2$ be any two potentials producing the same harmonic map by the procedure outlined above. Then there exists a gauge
$W_+:\D \rightarrow\Lambda^+G^{\C}_{\sigma}$ transforming one potential into the other. For a proof consider the frames $F_1 = C_1 V_{+1}$ and
$F_2 = C_2 V_{+2}$ constructed as outlined above. Since we assume that the two potentials induce the same harmonic map, these frames only differ by some gauge: $F_1 = F_2 T$ where $T \in K$. This implies
$C_1 V_{+1} = C_2 V_{+2} T$. Thus $W_+ = V_{+2} T V_{+1}^{-1}$ is the desired gauge.
\end{enumerate}\end{remark}
\vspace{1mm}

So far we have only discussed potentials for harmonic maps defined on some
contractible open subset of $\C$. Let now $M$ denote a Riemann surface which is either non-compact or compact of positive genus.
Then the universal cover $\tilde{M}$ of $M$ can be realized as a contractible open subset of $\C$. Moreover, if $f:M \rightarrow G/K$ is a harmonic map, then the composition $\tilde{f}$ of $f$ with the canonical projection from $\tilde{M}$ onto $M$ is also harmonic. Therefore to $\tilde{f}$ we can construct normalized potentials and holomorphic potentials as outlined above. These potentials for $\tilde{f}$
will also be called potentials for $f$.
The converse procedure as outlined in the last two theorems produces harmonic maps defined on some open subsets (containing the base point) of $\D$.
For these harmonic maps to descend to $M$ ``closing conditions" need to be satisfied.

\begin{remark}\
\begin{enumerate}
 \item If $M = S^2$, then it is not clear a priori that the procedure discussed in Theorem \ref{DPW} works as well. However, if the  symmetric target space actually is a real Lie group $G,$ considered as a symmetric space $G \cong (G \times G)/\Delta$, where $\Delta$ denotes the subgroup
$\Delta = \lbrace (g,g) \in G \times G, g \in G \rbrace$ and one uses the natural projection $(g,h) \rightarrow  gh^{-1},$ then the same procedure works.
In this case one can  lift a harmonic map $f: S^2 \rightarrow G$  to $G \times G$
by $F = (f,e)$. This way one obtains, as in the previous cases, a normalized potential. It has the form
$\xi = (\lambda^{-1} \eta , -\lambda^{-1} \eta)$. Harmonic maps into Lie groups (as symmetric spaces) have been discussed in \cite{Do-Es}, Section 9
. Note, however, that the formula given in \cite{Do-Es} for the normalized potential shows a wrong $\lambda-$ dependence.

 \item On the other hand, one does not obtain a holomorphic potential for $M=S^2$, since $S^2$ does not carry any non-trivial holomorphic $(1,0)-$forms. The proof of \cite{DPW} which was  used in the proof of the theorem above  is not applicable to the case $M = S^2$, of course.
\end{enumerate}\end{remark}
Let's consider now the case that we have a harmonic map $f$ from $M = S^2$
into some general symmetric space $G/K$. Since $\pi_2 (S^2) = \mathbb{Z}$ and  $\pi_2 (G) = 0$, it will generally not be possible to lift the smooth map $f$ to a smooth map $F$  from $S^2$ to $G$ such that the map $F$ composed with the natural projection from $G$ to $G/K$ is $f$. But one can find some way around this non-lifting obstacle.
\begin{theorem}\label{th-potential-sphere}
Every harmonic map from $S^2$ to any  Riemannian or pseudo-Riemannian symmetric space $G/K$ admits an extended frame with at most two singularities and it admits  a global meromorphic extended frame. In particular, every harmonic
 map from $S^2$ to any  Riemannian or pseudo-Riemannian symmetric space $G/K$
can be obtained from some meromorphic normalized potential.
\end{theorem}

\begin{proof}Let $f:S^2\rightarrow G/K$ be a harmonic map. Set $$\mathcal{U}_1=S^2\backslash \{\hbox{ north pole }\},
\ ~\mathcal{U}_2=S^2\backslash \{\hbox{ south pole }\}$$
and $f_1=f|_{\mathcal{U}_1}$, $f_2=f|_{\mathcal{U}_2}$. Since $\mathcal{U}_1\cong\mathcal{U}_2\cong\mathbb{C}$, there exist frame lifts $F_j:\mathcal{U}_j\rightarrow G$ of $f_j$, $j=1,2$, by \cite{DPW}.

We can assume w.l.g. $F_1(p_0)=F_2(p_0)=e$ where $p_0$ is a fixed base point in $\mathcal{U}_1\cap\mathcal{U}_2$ and $f(p_0)=e \mod K$.
Also we have $F_2=F_1\mathcal{K}$ on $\mathcal{U}_1\cap\mathcal{U}_2$. Introducing $\lambda$ yields ($\sigma-$twisted) $F_1$ and $F_2$ and again
$F_2=F_1\mathcal{K}$, where $F_j=F_j(z,\bar{z},\lambda)$ and $\mathcal{K}=\mathcal{K}(z,\bar{z})$. By \cite{DPW}, there exist discrete subsets $D_j\subset \mathcal{U}_j$, $j=1,2$ such that
$$F_j=F_{j-}F_{j+}, \ ~~j=1,2$$
on $\mathcal{U}_j\backslash D_j$. Moreover, $F_{j-}$ extends to a meromorphic map on $\mathcal{U}_j$ by \cite{DPW}.

On $(\mathcal{U}_1\cap\mathcal{U}_2)\backslash (D_1\cup D_2)$ we have
$F_{2-}V_{2+}=F_{1-}V_{1+}\mathcal{K}$, where $F_{j-}=I+\mathcal{O}(\lambda^{-1})$.
Hence
$$F_{2-}=F_{1-} \ ~\hbox{ on }\ ~ (\mathcal{U}_1\cap\mathcal{U}_2)\backslash (D_1\cup D_2).$$
Therefore this meromorphic map on $\mathcal{U}_1\cap\mathcal{U}_2$ extends meromorphically to $S^2$.
Set $\eta=F_-^{-1}dF_-.$
Then $\eta$ is a meromorphic $(1,0)-$form on $S^2$ of the form $\eta=\lambda^{-1}\eta_{-1}dz.$   As usual, $\eta$ will be called ``normalized potential'' of $f$. Moreover, by reversing the steps above we see that the map $f$ can be obtained from $\eta$ as usual, however,  we need to admit (up to two) singularities for the  extended  frame defined by $\eta$. Thus $\eta$ is justifiedly called the  normalized potential of $f$.
\end{proof}

\begin{remark}\
 \begin{enumerate}
  \item By removing just one point of $S^2$, like the north pole, one obtains a meromorphic map on $\mathcal{U}_1$ which, however, could have an essential singularity at the north pole. The use of $\mathcal{U}_1$ and $\mathcal{U}_2$ as above shows that $F_{1-}=F_{2-}$ on $\mathcal{U}_1\cap\mathcal{U}_2$ actually extends as a meromorphic frame to all points of $S^2$.
 \item From the proof above it is clear that the original harmonic map $f$ can be  reconstructed  from $\eta$ by the usual steps  (see Theorem \ref{DPW}). Note that the two procedures just discussed are inverse to each other.

 \item The theorem just proven can be used to construct all harmonic maps from $S^2$ into any Riemannian or pseudo-Riemannian symmetric space:

\emph{ Consider a meromorphic $(1,0)-$form on $S^2$ of the form $\eta=\lambda^{-1}\eta_{-1}dz $ which has a meromorphic solution $F_-$ on $S^2$ to the ode $ F_- \eta=dF_-, F_-(p_0,\lambda) = e.$}

Now an Iwasawa decomposition of $F_-$ makes sense for all points, where $F_-$ is in the open Iwasawa cell containing $e$, producing a ``frame" $F$ which is an actual frame on the set of non-singular points of $F_-$. Let $\tau$ denote the anti-holomorphic involution of $G^{\C}$  defining $G$.
Since $F$ is obtained via a Birkhoff decomposition of $\tau(F_{-})^{-1}F_-$ in the form
$\tau(F_{-})^{-1}F_-=\tau(V_{+})V_+^{-1}$, we obtain  $F=F_-V_+=\tau(F_-V_+)$, and its matrix entries are rational functions in the entries of $\tau(F_{-})^{-1}F_-$. In particular, the matrix entries of $F$ are rational functions in $u,v$, $z=u+iv$. Now a harmonic map is obtained by $f = F \mod K$.

 \item Since for pseudo-Riemannian spaces the Iwasawa splitting is not global in general, not every $\eta$ as above will yield a singularity free harmonic map on all of $S^2$.
The domain of definition of $f$ will need to be discussed separately in each case.
\end{enumerate}
\end{remark}

\begin{corollary} Let $f: S^2\rightarrow G/K$ be a harmonic map and $\eta$  its normalized potential with reference point $z_0$. Then away from the (finitely many) poles of $\eta$ there exists an extended frame $F$ for $f$ and a global Iwasawa splitting $F=F_-F_+$, $F_-^{-1}\frac{d}{dz}F_-=\eta$.
Moreover, $F_-$ is meromorphic on $S^2$.
\end{corollary}

Clearly, the normalized potential just discussed lives on $S^2 = M$.
 For general surfaces $M$ the potentials for harmonic maps only live on the universal cover $\tilde{M}$ of $M$. Therefore, if one wants to construct harmonic maps from some arbitrary Riemann surface $M$ into some  symmetric space $G/K$, one would have at least some indication for where to find an appropriate potential, if one would know that for every harmonic map from $M$ to  $G/K$ there is some potential  on $\tilde{M}$ which is the pullback of some differential one-form defined on $M$.
So far there is known \cite{Do-Ha5}, Theorem 3.2

\begin{theorem}
If $M$ is non-compact, then for every harmonic map from $M$ to any
 Riemannian or pseudo-Riemannian
symmetric space there exists a holomorphic potential defined  on the universal cover
$\tilde{M}$ of $M$ which is invariant under the fundamental group of $M$.
\end{theorem}
 \begin{remark}\
\begin{enumerate}
 \item
 By abuse of notation we sometimes say in the situation described above that the potential is defined on $M$.

\item For the case of compact surfaces $M$ we conjecture

\vspace{1.5mm}
{ \it Every harmonic map from any compact Riemann surface $M$ to any pseudo-Riemannian symmetric space
can be obtained from some meromorphic potential defined on $M$.}\vspace{1.5mm}\\
 In \cite{DoWa-sym1} we will prove this conjecture for all compact Riemann surfaces and for the  pseudo-Riemannian symmetric space occurring in our Willmore setting.
\end{enumerate}
\end{remark}

\subsection{The normalized potential for strongly conformally harmonic maps
and Wu's formula}

From the definition of the normalized potential (see Theorem \ref{DPW}) we can read off that
it is obtained from the $\lambda^{-1}-$part of the Maurer-Cartan form
of $F$ by conjugation by some matrix function with values in $K^\mathbb{C}$.
For known examples one can write down the normalized potential much more specifically.
In \cite{Wu}, Wu showed how one can determine locally the normalized potential
from the Maurer-Cartan form of the harmonic map $f$.

In this subsection we will make this explicit for the case of primary interest to this paper. As an immediate consequence of Theorem \ref{DPW} we obtain:

\begin{theorem} \label{normalized-potential-W} Let $\D$ be a contractible open subset of $\C$ and $0 \in \D$ a base point.

 Let $f: \D\rightarrow SO^+(1,n+3)/SO^+(1,3)\times SO(n)$ be a strongly
conformally harmonic map with $f(0)=eK$ and
$F:\D \rightarrow (\Lambda G_{\sigma})^0$ an extended frame of $f$
such that $F(0,\lambda) = I$. Then the normalized potential of $f$ with respect to the base point $z = 0$ is of the form
\begin{equation}\label{eq-potential-W}
\eta= \lambda^{-1}\eta_{-1}dz,\  \hbox{ with }\ \eta_{-1}=\left(
    \begin{array}{cc}
      0 & \hat{B}_1 \\
      -\hat{B}_1^tI_{1,3} & 0 \\
    \end{array}
  \right)dz,\hspace{3mm} \mbox{with} \hspace{3mm} \hat{B_1}^tI_{1,3}\hat{B}_1=0,
\end{equation}
where $\hat{B}_1dz$ is a meromorphic $(1,0)-$ form on $\D$ and $0$ is not a pole of $\hat{B}_1$.

Conversely,  any  normalized potential defined on $\D$ and satisfying  (\ref{eq-potential-W}) induces a strongly conformally harmonic
map from an open subset $0 \in \D_\mathcal{I} \subset \D$ into $SO^+(1,n+3)/SO^+(1,3)\times SO(n)$.
\end{theorem}

\begin{remark}
Using Theorem \ref{th-potential-sphere} one can formulate an analogous
result for $\D$ replaced by $S^2$.
\end{remark}

Similarly we obtain as an immediate consequence of Theorem \ref{pot-hol}:

\begin{theorem}\label{Holo-Willmore}
Let $\D$ be a contractible open subset of $\C$ and $0 \in \D$ a base point.
 Let $f: \D\rightarrow SO^+(1,n+3)/SO^+(1,3)\times SO(n)$ be a strongly
conformally harmonic map with $f(0)=eK$ and
$F:\D \rightarrow (\Lambda G_{\sigma})^0$ an extended frame of $f$
such that $F(0,\lambda) = I$. Then there exists a holomorphic potential for $f$ and each holomorphic potential for $f$ is of the form
\begin{equation} \label{eq-potential-H}
\xi=(\lambda^{-1}\xi_{-1}+\sum_{j\geq0}\lambda^j\xi_j)dz,\  \hbox{ with }\ \xi_{-1}=\left(
    \begin{array}{cc}
      0 & \hat{B}_1 \\
      -\hat{B}_1^tI_{1,3} & 0 \\
    \end{array}
  \right),\hspace{3mm} \mbox{and} \hspace{3mm} \hat{B_1}^tI_{1,3}\hat{B}_1=0,
\end{equation}
where $\xi_{j}dz$, $j=-1,0,\cdots,\infty$, are holomorphic $(1,0)-$ forms on $\D$.

Conversely,  any holomorphic potential $\eta$ defined on $\D$ and satisfying \eqref{eq-potential-H} induces  a strongly conformally harmonic
map from an open subset $0 \in \D_\mathcal{I} \subset \D$ into $SO^+(1,n+3)/SO^+(1,3)\times SO(n)$.
\end{theorem}

The matrix function $\hat{B_1}$ in the previous theorem (normalized potential) can be made much more explicit.

\begin{theorem} \label{Wu-W} ( Wu's Formula for Strongly Conformally Harmonic Maps)

Let $\D$ be a contractible open subset of $\C$ and $0 \in \D$ a base point.

 Let $f: \D\rightarrow SO^+(1,n+3)/SO^+(1,3)\times SO(n)$ be a strongly
conformally harmonic map with $f(0)=eK$ and
$F:\D \rightarrow (\Lambda G_{\sigma})^0$ an extended frame of $f$
such that $F(0,\lambda) = I$. Consider  $ F^{-1} dF = \alpha = \lambda^{-1} \alpha'_{\mathfrak{p}} + \alpha_{\mathfrak{k} }+ \lambda \alpha^{\prime \prime}_{\mathfrak{p}}$ and let $\delta_1$ denote the sum of the
holomorphic terms in the Taylor expansion of $\alpha'_{\mathfrak{p}}(\frac{\partial}{\partial z})$ about $0$, considered
as a form depending on $z$ and $\bar{z}$. The form $\delta_1$ is called the {\em holomorphic part}
of $\alpha'_{\mathfrak{p}}(\frac{\partial}{\partial z})$. Similarly, denote by $\delta_0$ the holomorphic part of $\alpha'_{\mathfrak{k}}(\frac{\partial}{\partial z})$.

Then the  normalized potential $\eta$ of $f$ with the origin as the reference point is given by
\begin{equation}
\eta=\lambda^{-1}\eta_{-1}dz\ \hbox{ with }\eta_{-1}=F_0(z)\delta_1F_0(z)^{-1},
\end{equation}
where $F_0: \D \rightarrow  G/K$ is  the solution to the equation $F_0(z)^{-1}dF_0(z)=\delta_0,\  F_0 (0) = I.$
\end{theorem}
\begin{proof}
 By Theorem  \ref{normalizationlemma} we know  the form of
$\alpha'_{\mathfrak{p}}(\frac{\partial}{\partial z}).$
Since the holomorphic part  $\delta_1$ of  $\alpha'_{\mathfrak{p}}(\frac{\partial}{\partial z})$ has the same form, we obtain by setting $\tilde{B}_1 (z, \bar{z} = 0)$
$$\delta_1=
\left(
    \begin{array}{cc}
      0 & \widetilde{B}_1 \\
      -{\widetilde{B} }_1^tI_{1,3} & 0 \\
    \end{array}
  \right)dz,\ \hbox{ with }\ {\widetilde{B} }_1^tI_{1,3}{\widetilde{B} }_1=0.
$$
Let
$F_0=\hbox{diag}(\hat{A}_1,\hat{A}_2):\D \rightarrow  SO^+(1,3,\mathbb{C})\times SO(n,\mathbb{C})$ be the solution to the equation
$F_0(z)^{-1}dF_0(z)=\delta_0,\ F_0(0)=I,$ where $\delta_0$ is the holomorphic part of $\alpha'_{\mathfrak{k}}(\frac{\partial}{\partial z})$.
Then Wu's Formula \cite{Wu} implies for the normalized potential
$$\eta_{-1}=F_0(z)\delta_1 F_0(z)^{-1}=\lambda^{-1} \left(
    \begin{array}{cc}
      0 & \hat{A}_1\widetilde{B} _1\hat{A}_2^{-1} \\
      -\hat{A}_2{\widetilde{B} }_1^tI_{1,3}\hat{A}_1^{-1} & 0 \\
    \end{array}
  \right)dz=\lambda^{-1}\left(
    \begin{array}{cc}
      0 & \hat{B}_1 \\
      -\hat{B}_1^tI_{1,3} & 0 \\
    \end{array}
  \right)dz.$$
Moreover,  from ${\widetilde{B} }_1^tI_{1,3}{\widetilde{B} }_1=0$ we obtain $\hat{B}_1^tI_{1,3}\hat{B}_1=0$.
\end{proof}
\begin{remark}
\
\begin{enumerate}
\item Note that one can assume w.l.g. that $\alpha$ has the special form stated in
Theorem \ref{normalizationlemma}. However this will not imply in general that $\eta_{-1}$ has such a special form. Later, in Section 6, we will show that only very special harmonic maps admit such kinds of normalized potentials.

\item
It is straightforward to verify that $\eta_{-1}$ in \eqref{eq-potential-W} satisfies
$$\eta_{-1}^3=0.$$ So $\eta$ is pointwise nilpotent as a Lie algebra-valued function.
However this does not imply that $\eta$ attains all values in a fixed nilpotent Lie subalgebra.
As a consequence, in general the corresponding conformally harmonic map is not of finite uniton type.
A standard example for this is the Clifford torus in $S^3$, which is of finite type and not of finite uniton type.

 \item In the last theorem we have considered local expansions of real analytic functions into power series in $z$ and $\bar{z}$ about $z=0$ and set $\bar{z} = 0$. So the factors entering into the formula for $\eta$ will in general  only be defined locally. However, $\eta$ itself is defined and meromorphic globally on $\D$.

If one wants to find globally defined factors for the representation of $\eta$ above, then one needs to analyze the proof of the corresponding result
of \cite{Do-Ko}.
\end{enumerate}
\end{remark}


\section{ A duality theorem for  harmonic maps into non-compact symmetric spaces}

The main goal of this section is to prove Theorem \ref{thm-noncompact}, which states that a harmonic map into a non-compact symmetric space induces naturally a harmonic map into the  compact dual symmetric space. This result is of great importance for the discussion of Willmore surfaces of finite uniton type \cite{DoWa2}, since it permits to apply the work of Burstall and Guest \cite{BuGu}, originally only applicable to harmonic maps into compact symmetric spaces, to the investigation of Willmore surfaces in spheres.

We will see that all Willmore 2-spheres are of finite uniton type
(the monodromy matrices all are trivial and all quantities of geometric interest are Laurent polynomials). Hence the work of \cite{BuGu} can be applied. Along these lines, in \cite{Wang-iso}  a new  Willmore 2-sphere in $S^6$ is produced  (see Theorem \ref{thm-example} ) which  solves a long open problem posed by Ejiri \cite{Ejiri1988}.

Note that for higher genus some of the Willmore surfaces are of finite uniton type, but in general others are not.

\subsection{Proof of Theorem \ref{thm-noncompact}}
Since in this section we discuss a general statement we return to the general setting  discussed in Section 4.1.
Hence we consider a connected, non-compact, semi-simple, real  Lie group $G$
and a non-compact,  inner, pseudo-Riemannian symmetric space $G/K$ defined by $\sigma$. Set $\mathfrak{g} =Lie(G)$ and
$\mathfrak{g} ^\C=\mathfrak{g}\otimes \C$. Denote by $\tau$ the complex anti-linear involution defining $\mathfrak{g}$ in  $\mathfrak{g} ^\C$. Obviously $\sigma\tau=\tau\sigma$.
Let $\theta $ be some Cartan involution of $\mathfrak{g} ^\C$ commuting with $\sigma$ and $\tau$.

In this section we will  represent the homogeneous space $G/K$ also in the form $G/K = \tilde{G}/\tilde{K}$, where $\tilde{G}$  is simply connected.
Let  $\tilde{G}^\C$ denote the complexification of $\tilde{G}$ \cite{Hochschild}. The Lie algebra of $G^\C$ is $\mathfrak{g} ^\C$ and  $\tilde{G}^\C$ is simply connected and  $\sigma, \tau$ and $\theta$ have extensions to pairwise commuting involutive group homomorphisms of  $\tilde{G}^\C$.
Since we start from a semi-simple Lie group $G$, we can assume that the natural image of $\tilde{G}$ in $\tilde{G}^\C$ is a closed subgroup of $\tilde{G}^\C$ \cite{Hochschild}.

Let $f:M\rightarrow G/K = \tilde{G}/\tilde{K}$ be a harmonic map with
an extended frame
$F:\tilde{M}  \rightarrow (\Lambda G_{\sigma})^0\subset\Lambda G^{\mathbb{C}}_{\sigma}$. Then we also have an extended frame
$\tilde{F}:\tilde{M}  \rightarrow \Lambda \tilde{G}_{\sigma}\subset\Lambda \tilde{G}^{\mathbb{C}}_{\sigma}$. (Note that here for $G = SL(2,\R)$ the last two ``inclusions''  actually describe inclusions of the image of $\tilde{G}$ under the natural homomorphism.
 Also recall that by Theorem \ref{th-potential-sphere} in the case $M = S^2$ the frame $F$ is permitted/required  to have singular points.)

To relate $f$ to a harmonic map $\hat{f}$ into a compact inner symmetric space, let $\tilde{U} = Fix ^{\theta}(\tilde{G}^\C)$.
Then $\tilde{U}$ is a maximal compact subgroup of  $\tilde{G}^\mathbb{C}$, and $\tilde{U}$ is connected and simply connected \cite{Aom}. Moreover, observe that
$\tilde{K}^\mathbb{C} = Fix ^{\sigma }(\tilde{G}^\C) \subset \tilde{G}^\C$ is a  connected complex Lie group    satisfying
$\tilde{K}^\C \cap \tilde{G} = \tilde{K}$.
Let $\mathfrak{g} = \mathfrak{k} + \mathfrak{p}$
be the decomposition of $\mathfrak{g}$ relative to $\sigma$ and $\mathfrak{g} = \mathfrak{h} + \mathfrak{m}$
the decomposition of $\mathfrak{g}$ relative to $\theta$.
Then $$\mathfrak{g} = \mathfrak{k} \cap  \mathfrak{h} + \mathfrak{k} \cap \mathfrak{m} +
\mathfrak{p} \cap\mathfrak{h} + \mathfrak{p} \cap \mathfrak{m} $$ as a direct sum of vector spaces.
Moreover, for the Lie algebra $\mathfrak{u}$ of $\tilde{U}$   we have
\begin{equation*}\begin{split}\mathfrak{u} & = \mathfrak{k} \cap  \mathfrak{h} + \mathfrak{p} \cap \mathfrak{h} +
i \left( \mathfrak{k} \cap  \mathfrak{m} + \mathfrak{p} \cap \mathfrak{m} \right) \\
&=
\left(\mathfrak{k} \cap  \mathfrak{h} + (i \mathfrak{k}) \cap (i\mathfrak{m}) \right) +
\left(\mathfrak{p} \cap  \mathfrak{h} + (i \mathfrak{p}) \cap (i\mathfrak{m}) \right)\\
&=
\mathfrak{k^{\mathbb{C}}} \cap \mathfrak{u}  + \mathfrak{p^{\mathbb{C}}} \cap   \mathfrak{u} .\end{split}\end{equation*}
It is easy to see now that $ \left( \mathfrak{k^{\mathbb{C}}} \cap
\mathfrak{u}\right)^{\mathbb{C}} = \left(\mathfrak{k} \cap  \mathfrak{h} + (i\mathfrak{k}) \cap (i \mathfrak{m}) \right)^{\mathbb{C}} = \mathfrak{k^{\mathbb{C}}} $ holds.
As a consequence, for the maximal compact Lie subgroup
$\tilde{U}$ of $\tilde{G}^{\mathbb{C}}$ constructed above we obtain
$$(\tilde{U}\cap \tilde{K}^{\mathbb{C}})^{\mathbb{C}}=\tilde{K}^{\mathbb{C}}$$ which follows, since both sides represent connected Lie subgroups of
$\tilde{G}^{\mathbb{C}}$ and have the same Lie algebra (Springer-Steinberg Theorem and  \cite{Helgason}, chapter VII, Theorem 7.2). Since $\sigma$ is inner, we obtain
$rank( \mathfrak{g}) = rank( \mathfrak{k})  $ and then also by using the last equation above
$rank( \mathfrak{u}) = rank( \mathfrak{u} \cap \mathfrak{k}^\C)$.
Hence

\begin{lemma}\label{lemma-inner}
The symmetric space $\tilde{U}/(\tilde{U}\cap \tilde{K}^{\mathbb{C}})$ is an inner symmetric space.
\end{lemma}

\begin{example} For a strongly conformally harmonic map associated with a strong Willmore map $f:M\rightarrow S^{n+2}$ we have
$G=SO^+(1,n+3)$,  $\tilde{G} \cong Spin(1, n+3)^0$,
 $ G^{\mathbb{C}} \cong  SO(1, n+3, \C)$ and $ \tilde{G}^\C \cong Spin(1,n+3,\mathbb{C})$.
Moreover, we have
$ K=SO^+(1,3)\times SO(n)$ and $ \tilde{K}^\C = Spin (1,3,\C) \times Spin (n,\C).$
 Hence $\tilde{U}\cong Spin(n+4),$ and
$\tilde{U} \cap \tilde{K}^{\mathbb{C}}=Spin(4)\times Spin(n)$. On the Lie algebra level, we have
 \[Lie(\tilde{U})=\mathfrak{u}=\{A\in \mathfrak{so}(1,n+3,\mathbb{C})|A=(a_{jk}),\ ia_{1j}\in\mathbb{R},j=1,\cdots,n+4\},\]
  \[(\mathfrak{u}\cap\mathfrak{k}^{\mathbb{C}})^{\mathbb{C}}=\mathfrak{so}(1,3,\mathbb{C})\times \mathfrak{so}(n,\mathbb{C})=\mathfrak{k}^{\mathbb{C}}.\]
\end{example}

\ \\
{\em Proof of Theorem \ref{thm-noncompact}}

For the harmonic map $f : \tilde{M} \rightarrow \tilde{G}/\tilde{K}$ we  consider the non-global
Iwasawa decomposition in $\Lambda \tilde{G}^{\mathbb{C}}_{\sigma} $ relative to $\Lambda \tilde{G}_{\sigma}.$ 
 Restricting $\sigma$ to $\tilde{U}$, we can consider the
twisted loop group $\Lambda \tilde{U}^{\mathbb{C}}_{\sigma}$ and the
corresponding Iwasawa decomposition of $\Lambda \tilde{U}^{\mathbb{C}}_{\sigma}$ relative to $\Lambda \tilde{U}_{\sigma}.$
By our construction, complexifying $\tilde{G}$ and complexifying $\tilde{U}$ yields the same complex Lie group $\tilde{G}^{\mathbb{C}}$
and the holomorphic extensions of $\sigma$, considered as an involution of $\tilde{G}$ or considered as an involution of  $\tilde{U},$ yield the same involution of $\tilde{G}^{\mathbb{C}}$.
Therefore the  complex twisted loop groups, constructed by starting from $\tilde{G}$ or starting from $\tilde{U}$ are the same, that is,
\[\Lambda \tilde{G}^{\mathbb{C}}_{\sigma} =\Lambda \tilde{U}^{\mathbb{C}}_{\sigma},\ \hbox{ and } \ \Lambda^{+} \tilde{G}^{\mathbb{C}}_{\sigma}=\Lambda^{+} \tilde{U}^{\mathbb{C}}_{\sigma}.\]
Applying this we can also perform the following Iwasawa decomposition of $\Lambda \tilde{G}^{\C}_{\sigma}$:
\begin{equation} \label{otherIwasawa}
\Lambda \tilde{U}_{\sigma} \cdot\Lambda^{+} \tilde{G}^{\mathbb{C}}_{\sigma}=\Lambda \tilde{G}^{\mathbb{C}}_{\sigma}.
\end{equation}

Now let us turn to the harmonic maps. First we assume that $\tilde{M}=\mathbb{D}$ is a contractible open subset of $\C$. Then we derive a global extended frame $F(z,\bar{z},\lambda)$ of $f$. Applying the decomposition (\ref{otherIwasawa}) to the frame $F$ we obtain
\begin{equation}\label{eq-noncompact}
F=F_{\tilde{U}}\cdot W_{+},\hspace{3mm} F_{\tilde{U}}\in \Lambda \tilde{U}_{\sigma},\hspace{3mm}  W_{+}\in \Lambda^{+}\tilde{U}^{\mathbb{C}}_{\sigma}. \end{equation}
Writing as usual  $\alpha=F^{-1}dF=\lambda^{-1}\alpha_{\mathfrak{p}}'+
\alpha_0+\lambda\alpha_{\mathfrak{p}}'',$
we obtain
\[F_{\tilde{U}}^{-1}dF_{\tilde{U}}=\alpha_{\tilde{U}}=W_{+}\alpha W_{+}^{-1}-dW_{+}W_{+}^{-1}=
\lambda^{-1}W_0\alpha_{\mathfrak{p}}'W_0^{-1}+
\alpha_{\tilde{U},0} + \lambda \alpha_{\tilde{U},1}+ \dots.\]
Since $W_+\in \Lambda^+ \tilde{U}^{\mathbb{C}}_{\sigma}=\Lambda^+ \tilde{G}^{\mathbb{C}}_{\sigma}$, we have
$W_0\in \tilde{K}^{\mathbb{C}},$  and $ \sigma( W_0 )= W_0.$
Since $\alpha_{\mathfrak{p}}'\in   \mathfrak{p}^{\mathbb{C}}$,
we obtain moreover $$\sigma(\alpha_{\mathfrak{p}}')=-\alpha_{\mathfrak{p}}', ~~ \hbox{ and }~~ \sigma (W_0\alpha_{\mathfrak{p}}'W_0^{-1})=-W_0\alpha_{\mathfrak{p}}'W_0^{-1}.$$
Since $\alpha_{\tilde{U}}$ is fixed by the anti-holomorphic involution $\theta$ we infer

\[\alpha_{\tilde{U}}=\lambda^{-1}\tilde{\alpha}'_{\mathfrak{p}}+
\tilde\alpha_{\mathfrak{k}}+
\lambda\tilde{\alpha}''_{\mathfrak{p}},~~\hbox{ with }~~\tilde\alpha_{\mathfrak{k}} \in \mathfrak{u} \cap
\mathfrak{k}^{\mathbb{C}},\hbox{ and  }
\tilde{\alpha}''_{\mathfrak{p}} = \theta \left( \tilde{\alpha}'_{\mathfrak{p}}\right) \in
\mathfrak{p}^{\mathbb{C}}.\]
As a consequence, $F_{\tilde{U}}$ is the frame of a harmonic map $f_{\tilde U} : \tilde{M} \rightarrow \tilde{U}/(\tilde{U}\cap \tilde{K}^{\mathbb{C}})$, where actually
\[f_{\tilde{U}}=F_{\tilde{U}} \mod \ \tilde{U} \cap\tilde{ K}^\mathbb{C}.\]
Computing the Birkhoff decomposition of $F$ as well as the Birkhoff decomposition of $F_{\tilde{U}}$ we obtain
\[F_-F_+=F=F_{\tilde{U}}W_+=F_{\tilde{U},-}\cdot F_{\tilde{U},+}\cdot W_+,\]
with
\[F_-=I+O(\lambda^{-1}), \hspace{3mm} F_{\tilde{U},-}=
I+O(\lambda^{-1})\in \Lambda^-_*\tilde{U}^{\mathbb{C}}_{\sigma}= \Lambda^-_*\tilde{G}^{\mathbb{C}}_{\sigma}.\]
This implies
$F_-=F_{\tilde{U},-},$ whence we also have  $\eta=F_-^{-1}dF_-=F_{\tilde{U},-}^{-1}dF_{\tilde{U},-}.$

Conversely, let $h_{\tilde{U}}:\D\rightarrow\tilde{U}/ (\tilde{U} \cap \tilde{K}^{\mathbb{C}} )$ be a harmonic map from the unit disk $\D$ with $h_{\tilde{U}}|_{z=0}=e$ and $H_{\tilde{U}}$ an extended frame for
$h_{\tilde{U}}$ satisfying $H_{\tilde{U}}(z=0) = I$.
Then there exists a neighbourhood $\D_0\subset \D$ of $0$ on which
 $H_{\tilde{U}} = F V_+$ holds with $F\in \Lambda G_\sigma$, i.e. where
 $H_{\tilde{U}}$ has an Iwasawa decomposition relative to $ \Lambda G_\sigma$. Then $h \equiv F \mod K $ satisfies the claim.

\hfill$\Box $

\begin{remark} \label{duality}We would like to point out that the last part of  Theorem \ref{thm-noncompact} shows that the
``duality" between the harmonic maps into compact symmetric spaces  the ones in the dual non-compact symmetric spaces is in general only local, due to the fact that the corresponding Iwasawa decompositions for non-compact symmetric spaces is in general not global. It would be interesting to understand this duality in a more global sense.
\end{remark}

\subsection{ Applications to finite uniton type harmonic maps}

The notion of {\bf finite uniton}  type of some harmonic map  was coined by Uhlenbeck in \cite{Uh}.
For this definition she required special properties of  ``extended solutions'', objects used extensively in that paper. Moreover, her definition was in the context of maps defined on $S^2$ or simply-connected subsets of $S^2$. In \cite{BuGu} the definition was extended (also by restrictions on extended solutions) to harmonic maps from arbitrary Riemann surfaces $M$ to (compact inner) symmetric spaces.

In our work we primarily use  extended frames (not extended solutions).
And in view of Theorem \ref{th-potential-sphere} this also makes sense for
$M = S^2$. Therefore we prefer to give the definition of ``finite uniton type''
in terms of extended frames.

For this purpose we introduce the notion of  an {\it algebraic loop} as meaning
 that the  Fourier expansion in $\lambda$ is a Laurent polynomial, i.e. it has only finitely many terms.
Such loops
will be denoted by the subscript $``alg"$, like  $$\Lambda_{alg} G_{\sigma},\ \Lambda_{alg} G^{\mathbb{C}}_{\sigma},\
\Omega_{alg} G_{\sigma}.$$
 We define
$$ \Omega^k_{alg} G_{\sigma}:=\{\gamma\in
\Omega_{alg} G_{\sigma}|
Ad(\gamma)=\sum_{|j|\leq k}\lambda^jT_j \}\ .$$

Now can define the notion of {\it finite uniton type}.

\begin{definition}\label{def-uni}  Let $M$ be a Riemann surface, compact or non-compact. A harmonic map $f : M\rightarrow G/K$ is said to be of {\it finite uniton type} if some extended frame $F$ of $f$,  defined on the universal cover
$\tilde{M}$ of $M$ and satisfying $F(z_0,\lambda)=e$ for some base point $z_0\in M$, has the following two properties:
\begin{enumerate}[$(U1)$]
\item $F (z,\lambda)$ descends to a map from $M$ to
$(\Lambda G^{\C}_{\sigma})/K$,
 i.e. the map $F(z,\lambda):M\rightarrow (\Lambda G^{\C}_{\sigma})/K$ is well defined on $M$ (up to two singularities in the case of $M = S^2$) for all $\lambda \in S^1$.

\item  $F(z,\lambda)$ is a Laurent polynomial in $\lambda$.
\end{enumerate}
\end{definition}
\vspace{2mm}
Hence $f$ is of finite uniton type if and only if there exists an extended frame $F$ for $f$ which has a trivial monodromy representation and is a Laurent polynomial in $\lambda$. In particular, in this case  $F(M) \subset \Omega^k_{alg} G_{\sigma}$ for some $k$.

It is also easy to verify that $f$ has finite uniton type if and only if $F_-$,  obtained from $F$ by the Birkhoff decomposition
$F = F_- F_+$ with  $F_- = I + \mathcal{O} (\lambda^{-1})$, descends
to a map $F_- : M \rightarrow \Lambda^-_*G^\C_\sigma$
and is a Laurent polynomial.

\begin{definition}
 Let $f:M \rightarrow G$  be a harmonic map of finite uniton type from $M$ into $G/K$ and $F$ an extended frame for $f$ satisfying $(U_1)$ and $(U_2)$.

We say that $f$ has {\it finite uniton number k} if
$$ F(M)\subset \Omega^k_{alg} G_{\sigma},\
\hbox{ and } F(M)\nsubseteq \Omega^{k-1}_{alg} G_{\sigma}.$$
 In this case we write  $r(f)=k$.
\end{definition}

In \cite{DoWa2} the notion of finite uniton type is investigated in much more detail. In particular, it is shown there that the definition above coincides with the definition of \cite{BuGu} and \cite{Uh}.

Combining the definition just given with Theorem \ref{thm-noncompact} we obtain:

\begin{theorem}\
\begin{enumerate}
\item Let $f:\tilde{M}\rightarrow \tilde{G}/\tilde{K}$  be a harmonic map and $f_{\tilde{U}}$ the associated harmonic map into the compact symmetric space
$\tilde{U}/(\tilde{U}\cap \tilde{K}^{\mathbb{C}})$  as in Theorem \ref{thm-noncompact}. Then $f$ is of finite uniton type if and only if $f_{\tilde{U}}$ is of finite uniton type. Moreover, we have $r(f) = r(f_{\tilde{U}})$.

\item If $\tilde M=S^2$, then $f$ always is of finite uniton type.
\end{enumerate}
\end{theorem}

\begin{proof}
 (1) Let $F$, $F_{\tilde{U}}$ be the (local) extended frame of $f$ and $f_{\tilde{U}}$ respectively and assume that $f$ or $f_{\tilde{U}}$ is of finite uniton type and F or $F_{\tilde{U}}$the corresponding frame respectively.
 By \eqref{eq-noncompact},  the Fourier expansion of $F$ contains only finitely many  powers of $\lambda^{-1}$ if and only if $F_{\tilde{U}}$ has only finitely many powers of $\lambda^{-1}$ in its Fourier expansion. Using the reality of $F$ and $F_{\tilde{U}}$, we conclude that $F$ is a Laurent polynomial of $\lambda$ if and only if  $F_{\tilde{U}}$ is a Laurent polynomial of $\lambda$.  The equality of $r(f)$ and $r(f_{\tilde{U}})$ follows, since the extremal powers of $\lambda$ occurring in $F$ and $F_{\tilde{U}}$ are the same. Finally, again by  \eqref{eq-noncompact}
it is easy to verify that $F\mod K$ is defined on $\tilde{M}$ if and only if
$F_{\tilde{U}}\mod \tilde{U} \cap \tilde{K}^\mathbb{C}$ is defined on $\tilde{M}$.

(2). In the case of $\tilde{M}=S^2$, by (1) it suffices to show that $f_{\tilde{U}}$ is of finite uniton type.
 To show this, first we claim that $\mathbb{F}=F_{\tilde{U}}(z,\bar z, -1)F_{\tilde{U}}(z,\bar z, 1)^{-1}$ is a harmonic map into $\tilde{U}$ having $\Phi(z,\bar z,\lambda)=F_{\tilde{U}}(z,\bar z, \lambda)F_{\tilde{U}}(z,\bar z, 1)^{-1}$ as its extended solution. Here $F_{\tilde{U}}(z,\bar z, \lambda)$ is the extended frame of $f_{\tilde U}$ with the Maurer-Cartan form
  $\alpha_{\tilde{U}}=\lambda^{-1} \alpha'_{\mathfrak{p}}+
\alpha_{\mathfrak{k}}+
\lambda\alpha''_{\mathfrak{p}}$ and $F_{\tilde{U}}(z_0,\bar z_0,\lambda)=e$. Straightforward computations show
\[\mathbb{A}=\frac{1}{2}\mathbb{F}^{-1}d\mathbb{F}=-F_{\tilde{U}}(z,\bar z, 1)(\alpha'_{\mathfrak{p}}+\alpha''_{\mathfrak{p}})F_{\tilde{U}}(z,\bar z, 1)^{-1}=\mathbb{A}^{(1,0)}+\mathbb{A}^{(0,1)},\]
and
 \[\Phi(z,\bar z,\lambda)^{-1}d \Phi(z,\bar z,\lambda)=(1-\lambda^{-1})\mathbb{A}^{(1,0)}+(1-\lambda)\mathbb{A}^{(1,0)}.
 \]
 The claim now follows from Theorem 2.1 of \cite{Uh}. Moreover, since $\Phi(z_0,\bar z_0,\lambda)=e$, by Theorem 11.5 of \cite{Uh} the extended solution $\Phi(z,\bar z,\lambda)$ is a Laurent polynomial of $\lambda$. Hence $F_{\tilde{U}}(z,\bar z, \lambda)$ is a Laurent polynomial of $\lambda$ and as a consequence $f_{\tilde{U}}$ is of finite uniton type.
\end{proof}


 \section{Application of Loop group theory to Willmore surfaces}

In this section we will present applications of Wu's formula for two types of harmonic maps.


\subsection{Strongly conformally harmonic maps containing a constant light-like vector}

From Theorem \ref{th-Willmore-harmonic-U}, we see that there are two kinds of conformally harmonic maps satisfying $B_1^tI_{1,3}B_1=0$:
those which contain a constant lightlike vector and those which do not contain a constant lightlike vector.
Moreover, if a conformally harmonic map $f$ does not contain a lightlike vector, then $f$ will always be the conformal Gauss map of some Willmore map.
This class of Willmore maps corresponds exactly to all those Willmore maps which are {\em not }conformal to any minimal surface in $\mathbb{R}^{n+2}$,
since minimal surfaces in $\mathbb{R}^{n+2}$ can be characterized as Willmore surfaces with their conformal Gauss map containing a constant lightlike vector.
Since minimal surfaces in  $\mathbb{R}^{n+2}$ can be constructed by direct methods,
we are mainly interested in Willmore surfaces not conformally equivalent to minimal surfaces in  $\mathbb{R}^{n+2}$.
It is therefore vital to derive a criterion to determine whether a strongly conformally harmonic map $f$ contains a lightlike vector or not.
This is the main goal of this subsection. We state the main result and refer for a proof (which uses substantially the techniques discussed in the previous sections)  to \cite{Wang-Min}.

\begin{theorem}\cite{Wang-Min} \label{th-potential-light}Let $\tilde{M}$ denote the Riemann surface $S^2, \mathbb{C}$ or the unit disk of $\mathbb{C}$. Let $f:  \mathbb{D}\rightarrow SO^+(1,n+3)/SO^+(1,3)\times SO(n)$ be a strongly conformally harmonic map which contains a constant light-like vector.  Choose a base point $p \in \tilde{M}$ and assume that $f(p)=I_{n+4}\cdot K$holds. Let  $z$ denote a local coordinate with $z(p)=0$. Then the normalized potential of $f$ with reference point $p$ is of the form
 \begin{equation}\label{eq-w-minimal}
\eta=\lambda^{-1}\left(
                     \begin{array}{cc}
                       0 & \hat{B}_1 \\
                       -\hat{B}^{t}_1I_{1,3} & 0 \\
                     \end{array}
                   \right)dz,\ \hbox{ where }\ \hat{B}_{1}=\left(
\begin{array}{cccc}
 \hat{f}_{11} & \hat{f}_{12} & \cdots &  \hat{f}_{1n} \\
 -\hat{f}_{11} &  -\hat{f}_{12} & \cdots &  -\hat{f}_{1n} \\
 \hat{f}_{31} &\hat{f}_{32} & \cdots &  \hat{f}_{3n} \\
 i\hat{f}_{31} & i\hat{f}_{32} & \cdots &  i\hat{f}_{3n} \\
 \end{array}
\right).\end{equation}
Here all  $f_{ij}$ are meromorphic functions on $\tilde{M}$.

The converse  also holds:
Let  $\eta$ be a normalized potential of the form \eqref{eq-w-minimal}. Then $B_1^tI_{1,3}B_1=0$ and we obtain a strongly conformally harmonic map $f: \tilde{M}\rightarrow SO^+(1,n+3)/SO^+(1,3)\times SO(n)$. Moreover, $f$ contains a constant light-like vector and is of finite uniton type.
\end{theorem}

\begin{remark} \
\begin{enumerate}
\item The proof of this result  requires a lengthy argument and will therefore be published in \cite{Wang-Min}.
It is not difficult to verify that $f$ is of finite uniton type.
Moreover, $f$ actually belongs to  the simplest case of finite uniton maps, the so-called $S^1-invariant$ maps (See \cite{BuGu}, \cite{Do-Es}).
For such harmonic maps, by a usually lengthy computation, one can derive the harmonic map directly without using loop groups,
since the Iwasawa splitting in this case is identical with the classical generalized Iwasawa splitting for non-compact Lie groups (see \cite{Do-Es}).

\item We would like to emphasize that the form of  $\hat{B}_1$ in equation (\ref{eq-w-minimal}) always occurs for the real analytic $B_1$ of Theorem \ref{normalizationlemma}. However, in the theorem above we obtain this form for the (meromorphic) normalized potential and these potentials  (obviously) describe a very special case  of strongly conformally harmonic maps.
\end{enumerate}
\end{remark}

\begin{corollary}
Let $f: \tilde{M}\rightarrow SO^+(1,n+3)/SO^+(1,3)\times SO(n)$ be a strongly conformally harmonic map
with its normalized potential $\eta$ of the form \eqref{eq-w-minimal} and of maximal $rank(\hat B_1)=2$. Then $f$ can not be
the conformal Gauss map of a Willmore surface. In particular, there exist strongly conformally harmonic maps which are not related to any Willmore map.
\end{corollary}
  Using the loop group method it is easy to see that harmonic maps satisfying the assumptions of the corollary always exist at least locally.
Moreover, combining Theorem \ref{th-Willmore-harmonic-U} and  Theorem \ref{th-potential-light}, and applying Wu's formula, it is straightforward to obtain the following
\begin{corollary} Let $f: \tilde{M}\rightarrow SO^+(1,n+3)/SO^+(1,3)\times SO(n)$ be a strongly conformally harmonic map
with its normalized potential $\eta$ of the form \eqref{eq-w-minimal} and of maximal $rank(\hat B_1)=1$. Then $f$ can not be
the conformal Gauss map of a Willmore surface if and only if up to a conjugation, $\hat B_1$ has one of the following forms
\begin{equation}\label{eq-w-min-2}
   \hat{B}_{1}=\left(
\begin{array}{cccc}
 0 & 0& \cdots &  0\\
 0 & 0& \cdots &  0\\
 \hat{f}_{31} &\hat{f}_{32} & \cdots &  \hat{f}_{3n} \\
 i\hat{f}_{31} & i\hat{f}_{32} & \cdots &  i\hat{f}_{3n} \\
 \end{array}\right), \hbox{ or } \hat{B}_{1}=\left(
\begin{array}{cccc}
 \hat{f}_{11} & \hat{f}_{12} & \cdots &  \hat{f}_{1n} \\
 -\hat{f}_{11} &  -\hat{f}_{12} & \cdots &  -\hat{f}_{1n} \\
 0 & 0& \cdots &  0\\
 0 & 0& \cdots &  0\\
 \end{array}\right).
\end{equation}
\end{corollary}
\begin{proof}  If $f$ can not be
the conformal Gauss map of a Willmore surface, then   by Theorem \ref{th-Willmore-harmonic-U}
it reduces to a harmonic map into $SO^+(1,n+1)/SO^+(1,1)\times SO(n)$ or $SO(n+2)/SO(2)\times SO(n)$. As a consequence, applying Wu's formula we see that the normalized potential reduces to $\Lambda \mathfrak{so}(1,n+1,\C)_{\sigma}$ or $\Lambda \mathfrak{so}(n+2,\C)_{\sigma}$. Now \eqref{eq-w-min-2} follows.
 The converse part is also straightforward. Since $\eta$ has the form stated in  \eqref{eq-w-min-2}, clearly  $f$  reduces to $SO^+(1,n+1)/SO^+(1,1)\times SO(n)$ or $SO(n+2)/SO(2)\times SO(n)$.
\end{proof}

\subsection{The conformal Gauss map of isotropic Willmore surfaces in $S^4$}

Another important class of Willmore surfaces is formed by the {\em totally isotropic} Willmore surfaces.
We recall from section 2.1 that $D$ denotes the $V_\C^\perp-$part
of the natural connection of $\C^4$. By $D_z^j$ we denote the $j-$fold iteration of $D_z$.

\begin{definition}(\cite{Ca}, \cite{Bryant1982}, \cite{Ejiri1988}) Let $y:M\rightarrow S^{n+2}$ be a conformal immersion with  $z$ a local coordinate of $M$ and $Y$ a  local lift. Then $y$ is called {\em totally isotropic} if the Hopf differential $\kappa$ of $y$ satisfies
 \begin{equation} \langle D_{z}^{j}\kappa,  D_{z}^{l}\kappa \rangle = 0, \hbox{ for } j,\ l=0,\ 1,\cdots.\end{equation}
\end{definition}
 Note that full  and totally isotropic surfaces only exist  in even dimensional spheres $S^{2m}$. They can be described as projections of holomorphic (anti-holomorphic) curves in the twistor bundle $\mathfrak{T}S^{2m}$ of $ S^{2m}$ (\cite{Ca}, \cite{Ejiri1988}).

However, in general, totally isotropic surfaces in $S^{2m}$ are not necessarily Willmore surfaces when $m>2$. Thus totally isotropic Willmore surfaces in $S^{2m}$ are of particular interest.
A much larger class of surfaces is formed by the isotropic surfaces, i.e. the surfaces, where in the definition above only the case $j = l = 0$ is required. Fairly little is known  about general isotropic  surfaces.

However, it is well-known that  all isotropic surfaces in $S^4$ are  Willmore surfaces (even S-Willmore surfaces), see \cite{Ejiri1988}, \cite{Ma2006}.

In this subsection we will characterize all isotropic Willmore surfaces in $S^4$.  An analysis of isotropic Willmore surfaces in $S^6$ will be presented in \cite{Wang-iso}. Concerning isotropic (Willmore) surfaces in $S^4$, we show
\begin{theorem}\label{th-potential-iso}
Let $y:M\rightarrow S^4$ be an isotropic surface from a simply connected Riemann surface $\tilde{M}$, with its conformal Gauss map $f=Gr$ defined in Section 2.1. Then the normalized potential of $Gr$ is of the form
\begin{equation}\label{eq-isotropic}\eta=\lambda^{-1}\left(
                     \begin{array}{cc}
                       0 & \hat{B}_1 \\
                       -\hat{B}^{t}_1I_{1,3} & 0 \\
                     \end{array}
                   \right)dz, \hbox{ with }
\hat{B}_{1}=\left(
                                                                       \begin{array}{cccc}
                                                                         \hat{f}_{11} & i\hat{f}_{11}   \\
                                                                          \hat{f}_{21} & i \hat{f}_{21}   \\
                                                                          \hat{f}_{31} &i\hat{f}_{31}   \\
                                                                         \hat{f}_{41} & i\hat{f}_{41}  \\
                                                                       \end{array}
                                                                     \right),\ -\hat{f}_{11} ^2+\hat{f}_{21} ^2+\hat{f}_{31} ^2+\hat{f}_{41} ^2=0.
\end{equation}
Moreover, $Gr$ is of finite uniton type with uniton number $r(f)$ at most 2. In particular, $f$ is $S^1$-invariant.

Conversely, let $\eta$ be defined on $\tilde{M}$ of the form \eqref{eq-isotropic} and let $f: \tilde{M}\rightarrow SO^+(1,5)/SO^+(1,3)\times SO(2)$ be the associated strongly conformally harmonic map. Then either $f$ is the conformal Gauss map of an isotropic S-Willmore surface in $S^{4}$, or $f$ takes values in   $SO^+(1,3)/SO^+(1,1)\times SO(2)$ or in $SO(4)/SO(2)\times SO(2)$ and is not the conformal Gauss map of any conformal immersion.
\end{theorem}
\begin{proof} Retaining the notation of Section 2.1  for $y$ and $Gr$, the isotropy property of $y$ shows that $\langle\kappa,\kappa\rangle=0$. Differentiating  this expression for $z$ one obtains
$\langle D_{\bar{z}}\kappa,\kappa\rangle=0.$
Noticing that the complex normal bundle is 2 dimensional and observing  that $\kappa$ is a  null vector section, it is clear that any other section perpendicular to $\kappa$  is necessarily  parallel to $\kappa$.  Hence we infer  that $D_{\bar{z}}\kappa$ is parallel to $\kappa$.
So without loss of generality, we can assume
\[\kappa=k_1\psi_1+ik_1\psi_2, \hbox{ and } D_{\bar{z}}\kappa=\beta_1\psi_1+i\beta_1\psi_2,\]
with $\psi_1,$ $\psi_2$ an orthonormal basis of sections of $V^{\perp}$ in the sense of Section 2.1.
Therefore the Maurer-Cartan form of $F(z,\bar{z},\lambda)$ w.r.t $y$ is\\
\[ F^{-1}F_z=\left(
                   \begin{array}{cc}
                     A_1 & B_1 \\
                     -B_1^tI_{1,3} & A_2 \\
                   \end{array}
                 \right),\
\hbox{ with }\
B_1=\left(
                                                                       \begin{array}{cccc}
                                                                          \sqrt{2}\beta_1 &  i\sqrt{2}\beta_1  \\
                                                                         -\sqrt{2}\beta_1 & -i\sqrt{2}\beta_1  \\
                                                                         -k_{1} & -ik_{1}   \\
                                                                         -ik_{1} & -ik_{1}  \\
                                                                       \end{array}
                                                                     \right).\]
To apply Wu's formula  (Theorem \ref{Wu-W}),
let  $\delta_{1}$, $\delta_{2}$ and  $\tilde{B}_1$  denote the ``holomorphic parts" of $ A_{1}$, $A_2$ and $B_1$ with respect to the reference point $z=0$ respectively, i.e., the part of the Taylor expansion of $A_1$, $A_2$ and $B_1$ which are independent of $\bar{z}$.
Let $F_{01}$ and $F_{02}$ be the solutions to the equations $F_{01}^{-1}dF_{01}=\delta_{1}dz,\  F_{01}|_{z=0}=I_4$ and
$F_{02}^{-1}dF_{02}=\delta_{2}dz,\  F_{02}|_{z=0}=I_2$ respectively.
By Wu's formula (Theorem \ref{Wu-W}), the normalized potential can be represented in the form
\[\eta=\lambda^{-1}\left(
                     \begin{array}{cc}
                       0 & \hat{B}_1 \\
                       -\hat{B}^{t}_1I_{1,3} & 0 \\
                     \end{array}
                   \right)dz,\ ~~\hbox{ with }\ ~~\hat{B}_{1}=F_{01}\tilde{B}_{1}F_{02}^{-1}.
\]
Noticing that here $\tilde{B}_1$ is of the form $\tilde{B}_1=\left(v_1, iv_1\right),\ \hbox{ with } \  v_1\in \mathbb{C}^4_1,$  it is immediate
to check that $F_{01}\tilde{B}_1 F_{02}^{-1}=\left(\hat{v}_1, i\hat{v}_1\right)$
holds with some vector $\hat{v}_1 \in \C^4$.

A straightforward computation shows that $\hat{B}_1^t I_{1,3}\hat{B}_1=0$ is equivalent with $  \hat{v}_1^tI_{1,3}\hat{v}_1=0$, and
\eqref{eq-isotropic} follows now.
]
In view of Definition \ref{def-uni}  of Section 5.2 the last statement is a corollary to the fact that $\eta(\frac{\partial}{\partial z})$ in \eqref{eq-isotropic} takes values in a nilpotent Lie subalgebra of
 degree of nilpotency $2$,  which shows that $F_-$ will be a polynomial in $\lambda^{-1}$ of degree at most $2$.

As to the converse part, assume that $\hat B_1=(\hat{v}_1,i\hat{v}_1)$. Let $f$ be the corresponding harmonic map with $B_1=(v_1,v_2)$. Then we have $B_1=V_{01}\hat{B}_{1}V_{02}^{-1}$ for some $V_{01}\in SO^+(1,3,\C)$ and $V_{02}\in SO(2,\C)$. So we have $v_2=iv_1$ holds. In particular $rank B_1\leq 1$. Applying Theorem \ref{th-Willmore-harmonic-U}, we see the Theorem holds except the isotropic property. But this is a consequence of the facts that $B_1$ being of the form $(v_1,iv_1)$ is independent of the choice of frames, and $B_1$ being of the form $(v_1,iv_1)$ is equivalent to the condition that the corresponding Willmore surface is isotropic (Note that these facts only hold in the case of codimension $2$).
\end{proof}

\begin{remark} \
\begin{enumerate}
\item Isotropic surfaces in $S^4$ provide another type of strongly conformally harmonic maps of finite uniton number $\leq2$,
which actually have an intersection with minimal surfaces in $\mathbb{R}^4$ (see e.g. the examples below). For more details, we refer to \cite{Mon}.
 And also note that  a Weierstrass type representation for isotropic minimal surfaces in $S^4$  has been presented in \cite{Bryant1982}.

 \item The case of isotropic Willmore surfaces in $S^6$ shows a very different situation, in particular by the fact that, in general, they can not be of finite uniton type \cite{Wang-iso}.
\end{enumerate}\end{remark}

By the classification theorems in \cite{Ejiri1988}, \cite{Mus1}, \cite{Mon},  a Willmore two-sphere in $S^4$ is either isotropic or is conformally equivalent to a minimal surface in $\mathbb{R}^4$. Applying Theorem \ref{th-potential-light} and Theorem \ref{th-potential-iso}, we obtain
\begin{corollary}The conformal Gauss map of a Willmore two-sphere in $S^4$ is of finite uniton type with $r(y)\leq2$ and hence is $S^1$-invariant.
\end{corollary}

 A main result of \cite{Le-Pe-Pin} states that a Willmore torus in $S^4$ with non-trivial normal bundle is either isotropic or is conformally equivalent to a minimal surface in $\mathbb{R}^4$.
Together with the above results we derive the following
\begin{corollary}The conformal Gauss map of a Willmore torus in $S^4$ with non-trivial normal bundle is of finite uniton number at most 2 and hence $S^1$-invariant.
\end{corollary}

We can also state  the classification theorem of Bohle \cite{Bo} on Willmore tori in $S^4$ as below. For the notion of ``finite type'' we refer to \cite{Bo}.

\begin{corollary}\cite{Bo} The conformal Gauss map of a Willmore torus in $S^4$ is either of finite type or of finite uniton number at most 2 (and hence is $S^1$-invariant in the latter case).
\end{corollary}


\subsection{On homogeneous Willmore surfaces admitting an Abelian group action}

\begin{definition}
A Willmore  immersion $y: M \rightarrow S^{n+2}$ is called homogeneous if there exists a group
$\Gamma:=\{(\gamma,R_{\gamma}):\ \gamma\in Aut(M),\ R\in SO^+(1,n+3)\}$ such that
\begin{equation}y(\gamma\cdot z)=R_{\gamma}\cdot y(z), \hbox{ for all } z\in M \hbox{ and } \ (\gamma,R_{\gamma})\in \Gamma
\end{equation}
and  the projection $\Gamma_M$ of $\Gamma$ onto the first factor acts transitively on $M$.
\end{definition}

 Clearly, with  $\Gamma$ also the closure
$\overline{\Gamma}$ in  $Aut(M)  \times SO^+(1, n+3)$ satisfies the conditions of the definition. We can thus assume that $\Gamma$ is a Lie group. Similarly, the closure
$\overline{\Gamma_M}$ in $Aut(M)$ is still abelian and transitive on $M$.

\begin{remark}
1. There are different definitions of the notion of a ``homogeneous surface". For now we will use the definition given above. We plan to discuss "homogeneous Willmore surfaces" in greater generality in a separate publication. In particular, there we will relate  our work to the
classification of Alekseevskii-Ferrand \cite{Alek}, \cite{Ferrand} .

2. We would like to point out that only the Riemann surfaces $\C, \C^*$ and (the torus)
$\mathbb{T}$ can occur as homogeneous surfaces in the sense of our definition above \cite{Kra}.
\end{remark}

\begin{theorem}
 Let $y:M \rightarrow S^{n+2}$ be a homogeneous Willmore immersion from a Riemann surface $M$.
Assume that the group $\Gamma$ is abelian. Then
\begin{enumerate}
\item  $M = \C$ and $\Gamma_M \cong$  all translations; or $M =S^1\times \R$ and
$\Gamma_M \cong S^1\times \mathbb{Z}$; or $M = \mathbb{T}$ and
$\Gamma_M \cong \mathbb{Z}^2$.

 \item  For the lift $\tilde{y}$ of $y$ to the universal cover $\C$ of $M$ there exists an  extended frame associated with the conformal Gauss map of $\tilde{y}$, which has a  constant
 Maurer-Cartan form $\alpha$ of the form stated in Proposition 2.2.
      Moreover, setting $\eta=\alpha'= \lambda^{-1} \eta_{-1} + \eta_0$, then $\eta$ is a constant, real, holomorphic potential, which generates the immersion
    $\tilde{y}$ (and hence also $y$)  and satisfies $\lbrack \eta \wedge \bar{\eta} \rbrack = 0$ and $\eta_{13} + \eta_{23} \neq 0$.

 \item Conversely, let $\eta$ be a constant real potential,  defined on $\C$ and  of the form $\eta = \lambda^{-1} \eta_{-1} + \eta_0$ satisfying $\lbrack \eta \wedge \bar{\eta} \rbrack = 0 $ and having the same form as $\alpha'$ in Proposition 2.2. Then it generates a homogeneous Willmore immersion (without branch points) for which its conformal Gauss map has an extended frame with constant Maurer-Cartan form.
\end{enumerate}
\end{theorem}
\begin{proof} (1). Since $M$ is a connected Riemann surface with  $Aut(M)$ containing
a two-dimensional abelian Lie subgroup, the universal covering of $M$ is $\C$ and the rest follows straightforwardly.

 (2). Choose the coordinate on $\C$,  we see that the group $\Gamma$ consists of all translations. Hence we obtain
for the extended frame of the conformal Gauss map the relation \[F(u+ iv, u-iv,\lambda) =
exp(u\mathfrak{X}) exp(vi\mathfrak{Y}) F(0,0,\lambda),\] where $\mathfrak{X}$ and $\mathfrak{Y}$ commute and only depend on $\lambda$.
Therefore the Maurer-Cartan form of $F$ is constant and
we obtain \[F^{-1} d F = \mathfrak{X}du + \mathfrak{Y}dv = (\mathfrak{X}+\mathfrak{Y})dz + i(\mathfrak{X}-\mathfrak{Y}) d\bar{z}.\]
So $\mathfrak{X}-\mathfrak{Y}$ only involves non-negative powers of $\lambda$ and $\mathfrak{X}+\mathfrak{Y}$ only non-positive powers of lambda. Moreover, the
matrices $\mathfrak{X}-\mathfrak{Y}$ and $\mathfrak{X}+\mathfrak{Y}$ commute.
As a consequence, assuming w.l.g. $F(0,0,\lambda) = e$, we also obtain
\[F(u+ iv, u-iv,\lambda) =
exp(u\mathfrak{X}) exp(v \mathfrak{Y})  = \exp (z (\mathfrak{X}+\mathfrak{Y})) \exp( \bar{z} i(\mathfrak{X}-\mathfrak{Y})).\]
Since this is some Birkhoff  decomposition of $\exp (z (\mathfrak{X}+\mathfrak{Y}))$
we conclude that \[ \eta=\exp (z (\mathfrak{X}+\mathfrak{Y}))^{-1}d\exp (z (\mathfrak{X}+\mathfrak{Y}))=(\mathfrak{X}+\mathfrak{Y})dz =(\lambda^{-1}\eta_{-1}+\eta_0)dz\] is a holomorphic potential for the given immersion of the type stated.

(3). Assume now $\eta$ is of the special form stated, then
$\eta=(\lambda^{-1}\mathfrak{B}+\mathfrak{A})dz$ with $\mathfrak{A},$ $\mathfrak{B}$ constant matrices satisfying
\[[\lambda^{-1}\mathfrak{B}+\mathfrak{A},\lambda\bar{\mathfrak{B}}+\bar{\mathfrak{A}}]=0.\]
Then
\[e^{z(\lambda^{-1}\mathfrak{B}+\mathfrak{A})}=e^{z(\lambda^{-1}\mathfrak{B}+\mathfrak{A})+\bar{z}(\lambda\bar{\mathfrak{B}}+\bar{\mathfrak{A}})}\cdot e^{-\bar{z}(\lambda\bar{\mathfrak{B}}+\bar{\mathfrak{A}})}\]
is an Iwasawa decomposition, producing the extended frame
\[F(z,\bar{z},\lambda)=e^{z(\lambda^{-1}\mathfrak{B}+\mathfrak{A})+\bar{z}(\lambda\bar{\mathfrak{B}}+\bar{\mathfrak{A}})}.\]
This implies that the conformally harmonic map $f = F \mod K$ is conformally homogeneous.
Since the Maurer-Cartan form of $F(z,\bar{z},\lambda)=(e_0,\hat e_{0}, e_1,e_2,\psi_1,\cdots,\psi_n)$ is of the form stated in Proposition 2.2, $F$ is the conformal Gauss map of some immersion $y=[e_0-\hat{e}_0]$. The harmonicity of the conformal Gauss map indicates that $y$ is a Willmore immersion.
\
\end{proof}

\begin{corollary} Every homogeneous Willmore torus in $S^{n+2}$ can be obtained from a constant potential of the form
\begin{equation}\label{eq-homo-int}
  \eta = \lambda^{-1} \eta_{-1} + \eta_0,~~\hbox{ with }~[\eta_{-1},\overline{\eta_0}]=0,\  [\eta_{-1},\overline{\eta_{-1}}]+ [\eta_{0},\overline{\eta_0}]=0,
\end{equation}
and $ \eta_{-1}$, $ \eta_{0}$ being of the same form  as $ \alpha_{\mathfrak{p}}'(\frac{\partial}{\partial z})$ and $\alpha_{\mathfrak{k}}'(\frac{\partial}{\partial z})$ respectively in Proposition \ref{frame}.
\end{corollary}

A special case of homogeneous strongly conformally harmonic maps is produced by ``{\bf vacuum potentials}". Recall the definition of a vacuum potential \cite{BP}
$$\eta=(\lambda^{-1}\mathfrak{B})dz,\ \hbox{ with } [\mathfrak{B},\bar{\mathfrak{B}}]=0.$$
Such a potential always produces a harmonic map $f$. For $f$ being a strongly conformally harmonic map, one needs to assume that
$$\mathfrak{B}=\left(
                   \begin{array}{cc}
                    0 & B_1 \\
                    - B_1^tI_{1,3} & 0 \\
                   \end{array}
                 \right),\ \hbox{ with }\ B_1^tI_{1,3}B_1=0.$$
Then, as shown in Lemma \ref{normalizationlemma}, there exists some $L_1\in SO^+(1,3)$ such that $L_1B_1$ is of the form stated in Lemma \ref{normalizationlemma}.
By Theorem \ref{th-potential-light}, one sees that for $f$ being the conformal Gauss map of some Willmore map $y$, the maximal rank of $B_1$ must be one.
Hence we may assume that
$$B_1=(\mathrm{v}_1,\cdots,\mathrm{v}_n) \ ~\hbox{ with }\ ~\mathrm{v}_j=(a_j+ib_j) \mathrm{v}_0,\ a_j, b_j\in\R,\ j=1,\cdots,n. $$
Here $\mathrm{v}_0\in \hbox{Span}_{\C}\{(1,-1,0,0)^t,(0,0,1,i)^t\}.$
If $\langle\mathrm{v}_0,\mathrm{v}_0\rangle=0$, then $\mathrm{v}_0\in \hbox{Span}_{\C}\{(1,-1,0,0)^t\}$.
So we see that $f$ reduces to a harmonic map into $SO(1,n)/SO(1,1)\times SO(n)$. If  $\langle\mathrm{v}_0,\mathrm{v}_0\rangle\neq0$, there exists another
$L_2\in SO^+(1,3)$ such that $L_2\mathrm{v}_0\in \hbox{Span}_{\C}\{(0,0,1,i)^t\}$.
As a consequence,  $f$ reduces to a harmonic map into $SO(n+2)/SO(2)\times SO(n)$.

In a sum, we obtain
\begin{proposition}Let $f$ be a vacuum solution which is also a strongly conformally harmonic map. Then it can not be the conformal Gauss map of any Willmore surface.
\end{proposition}

\begin{example}Let $y=[Y]: S^1\times R^1\rightarrow S^4$ be the cylinder
  \begin{equation}
Y=\left( \cosh av, \sinh av, \cos u \cos bv, \cos u \sin bv, \sin u \cos bv, \sin u \sin bv \right)^t
\end{equation}
with $a^2+b^2=1,\ a,b\in \mathbb{R}$. Note, if $a=0$ we obtain the Clifford torus in $S^3\subset S^4$, and if $b=0$ we obtain the round sphere with the north pole removed. (For a detailed discussion on Willmore tori in $S^4$, we refer to \cite{Bo} and \cite{Le-Pe-Pin}.)

A direct computation shows that $y$ is a homogeneous Willmore immersion, with a holomorphic potential
 \begin{equation}
\tilde\eta=\left(
                              \begin{array}{cccccc}
                                0 & 0 & \frac{1}{4\sqrt{2}} & \frac{-i(1+2a^2)}{4\sqrt{2}} & \frac{iab\lambda^{-1}}{2\sqrt{2}} & 0 \\
                                0 & 0 & \frac{3}{4\sqrt{2}} & \frac{-i(1+2b^2)}{4\sqrt{2}} & \frac{-iab\lambda^{-1}}{2\sqrt{2}} & 0 \\
                                \frac{1}{4\sqrt{2}} & \frac{-3}{4\sqrt{2}} & 0 & 0 & 0  & \frac{ ib \lambda^{-1}}{2}\\
                                \frac{-i(1+2a^2)}{4\sqrt{2}} & \frac{i(1+2b^2)}{4\sqrt{2}} & 0 & 0 & 0 &  \frac{ b\lambda^{-1}}{2}\\
                                 \frac{iab\lambda^{-1}}{2\sqrt{2}} & \frac{iab\lambda^{-1}}{2\sqrt{2}} & 0 & 0 & 0 & -\frac{a}{2}\\
                                0  & 0 & \frac{-ib\lambda^{-1}}{2} & \frac{-b\lambda^{-1}}{2}  & \frac{a}{2} & 0\\
                              \end{array}
                            \right)dz.
\end{equation}
\end{example}
\begin{example}\cite{Wang-Homo}  Let $y:\C \rightarrow S^5(\sqrt{1+2b^2})$
  \begin{equation}
y= \left(\cos  u \cos \frac{v}{\sqrt{3}}, \cos u \sin \frac{v}{\sqrt{3}}, \sin  u \cos \frac{v}{\sqrt{3}}, \sin u \sin \frac{v}{\sqrt{3}}, \sqrt{2}b\cos \frac{v}{\sqrt{3}b}, \sqrt{2}b\sin \frac{v}{\sqrt{3}b}\right)^t
\end{equation}
with $b\in \mathbb{R}^+$. Obviously $y$ is a torus if and only if $b\in \mathbb{Q}^+$.
A direct computation shows that $y$ is a homogeneous Willmore immersion, with a holomorphic potential
 \begin{equation}
\tilde\eta=\left(
                              \begin{array}{ccccccc}
                                0 & 0 & s_1 & s_2 & 0& 0 & \lambda^{-1}\sqrt{2}\beta_3 \\
                                0 & 0 & s_3 & s_4 & 0& 0 & -\lambda^{-1}\sqrt{2}\beta_3 \\
                               s_1& -s_3 & 0 & 0 & -\lambda^{-1}k_1 & -\lambda^{-1}k_2 & 0\\
                               s_2& -s_4 & 0 & 0 & -\lambda^{-1}ik_1 & -\lambda^{-1}ik_2 & 0\\
                               0& 0 & \lambda^{-1}k_1 & \lambda^{-1}ik_1 &0 & 0 & -a_{13}\\
                               0& 0& \lambda^{-1}k_2 & \lambda^{-1}ik_2 & 0 & 0 & -a_{23}\\
                               \lambda^{-1}\sqrt{2}\beta_3 & \lambda^{-1}\sqrt{2}\beta_3 & 0 & 0 &a_{13}  & a_{23} & 0\\
                              \end{array}
                            \right)dz,
\end{equation}
with
\[k_1=\frac{\sqrt{4b^2+2}}{12b},k_2=\frac{-i\sqrt{3}}{6}, s=\frac{4b^2-1}{18 b^2}, \beta_3=\frac{-i\sqrt{2}(4b^2-1)}{72b^2}, a_{13}=\frac{-i\sqrt{2b^2+1}}{6b}, a_{23}=\frac{\sqrt{6}}{6},\]
and
\[s_1=\frac{\sqrt{2}(20b^2+1)}{144b^2},\ s_2=\frac{-i\sqrt{2}(12b^2-1)}{48b^2},\ s_3=\frac{\sqrt{2}(52b^2-1)}{144b^2}, s_2=\frac{-i\sqrt{2}(12b^2+1)}{48b^2}.\]
Note that in this case one obtains the Ejiri's torus when $b=1$ \cite{Ejiri1982}.

Moreover, if we assume that $b=\frac{j}{l}$ with $j,l\in \mathbb{Z}^+$ and $(j,l)=1$, we obtain a torus with period $2\pi (1+il\sqrt{3})$. So the corresponding Willmore functional is
\[W(\mathbb{T}^2)=4\int_{0}^{2\pi}du\int_{0}^{2\pi l\sqrt{3}}dv\left(|k_1|^2+|k_2|^2\right)=\frac{16\pi^2\sqrt{3}}{9}\left(l+\frac{j^2}{8l}\right).\]
\end{example}
\ \\

\section{Appendix A: Two Decomposition Theorems}

In this section we discuss the basic decomposition theorems for loop groups.
We will use the notation introduced in section 4.
 Since the decomposition theorems usually are proven for loops in simply-connected groups we will assume in this section that $G$ is simply-connected and will therefore always use, to avoid confusion,  the notation $\tilde{G}$. We would like to point out that the case of the group $G = SL(2,\R)$ is included in our presentation, but needs, at places, some interpretation, since in this case $\tilde{G}$ is not a matrix group, while $\tilde{G}^\C = SL(2,\C)$ is a matrix group and ``contains''  $\tilde{G}$ as a (non-isomorphic) image of the natural homomorphism.


\subsection{Birkhoff Decomposition}

 Starting from $\tilde{G}$ and an inner involution $\sigma$,  there is a unique extension, denoted again by $\sigma$, to $\tilde{G}^\C$.
The corresponding fixed point subgroups will be denoted by $\tilde{K}^\C$. Note that the latter group is connected, by a result of Springer$-$Steinberg.

Next we will consider the twisted loop group
$\Lambda \tilde{G}^{\C}_{\sigma} $.
General loop group theory implies that, since we  consider inner involutions only, we have  $\Lambda \tilde{G}^{\C}_{\sigma}  \cong
\Lambda \tilde{G}^{\C}$.
On the other hand,  we know $\pi_0 (\Lambda H) = \pi_1 (H)$ for any connected Lie group $H$, whence we infer that
$\Lambda \tilde{G}^{\C}_{\sigma} $ is connected. And since $\tilde{K}^\C$ is connected, also the groups $\Lambda^{+} \tilde{G}^{\mathbb{C}}_{\sigma}$
and $\Lambda^{-} \tilde{G}^{\mathbb{C}}_{\sigma}$ are connected.

For the loop group method used in this paper two decomposition theorems are of crucial importance. The first is

\begin{theorem}\label{thm-birkhoff}(Birkhoff decomposition theorem) Let $\tilde{G}^{\mathbb{C}}$ denote a simply-connected complex Lie group
with connected real form $\tilde{G}$ and let $\sigma$ be an inner involution of $\tilde{G}$ and  $\tilde{G}^{\mathbb{C}}$. Then the following statements hold
\begin{enumerate}
\item  $\Lambda \tilde{G}^{\C}_{\sigma} =
\bigcup \Lambda^{-} \tilde{G}^{\mathbb{C}}_{\sigma} \cdot \tilde{\omega} \cdot \Lambda^{+} \tilde{G}^{\mathbb{C}}_{\sigma},$
where the $\tilde{\omega}$'s are representatives of the double cosets.\\

\item The multiplication
 \begin{equation}
\Lambda^-_* \tilde{G}^{\mathbb{C}}_{\sigma} \times\Lambda^+ \tilde{G}^{\mathbb{C}}_{\sigma}\rightarrow
\Lambda^-_* \tilde{G}^{\mathbb{C}}_{\sigma} \cdot\Lambda^+ \tilde{G}^{\mathbb{C}}_{\sigma}
\end{equation}
is a complex analytic diffeomorphism and the (left) ``big cell"
$ \Lambda^-_* \tilde{G}^{\mathbb{C}}_{\sigma}
\cdot\Lambda^+ \tilde{G}^{\mathbb{C}}_{\sigma}$ is open and dense in
$ \Lambda  \tilde{G}^{\mathbb{C}}_{\sigma}$.\\

\item More precisely, every $g$  in $ \Lambda  \tilde{G}^{\mathbb{C}}_{\sigma}$ can be written in the form
 \begin{equation}g=g_-\cdot \tilde{\omega}\cdot g_+\end{equation}
with $g_{\pm}\in \Lambda^{\pm}  \tilde{G}^{\mathbb{C}}_{\sigma}$,
 and $\tilde{\omega} : S^1 \rightarrow  \tilde{T} \subset Fix^\sigma (\tilde{G}^\C)$ a homomorphism , where $\tilde{T}$ is a maximal compact torus in $\tilde{G}^\C$ fixed pointwise by $\sigma$.\\
\end{enumerate}
\end{theorem}

\begin{proof} The decomposition above has been proven for algebraic loop groups in \cite{Ka-P}.
Our results follow by completeness in the Wiener Topology (see e.g. \cite{DPW}).
\end{proof}

\begin{remark}
\begin{enumerate}
\item Our actual goal is to obtain a Birkhoff decomposition theorem for $(\Lambda G_\sigma^\C)^0$. The restriction to the connected component
is possible, since we will always consider maps from connected surfaces into the loop group which attain the value $I$ at some point of the surface.

This is very fortunate: since we have obtained above  a Birkhoff decomposition for the simply connected  complexified universal group, we will attempt to obtain the desired Birkhoff decomposition by projection.
Since $\Lambda \tilde{G}^{\C}_{\sigma} $ is connected,
applying the natural extension of the natural projection  $\tilde{ \pi}^\C : \tilde{G}^{\mathbb{C}} \rightarrow {G}^{\mathbb{C}}$  to $\Lambda \tilde{G}^{\mathbb{C}}_{\sigma} $ we obtain as image  the connected component  $(\Lambda G_\sigma^\C)^0$ of $\Lambda G_{\sigma}^\C$ .
We thus obtain the desired Birkhoff decomposition by projecting the terms on the right side. But since the groups  $\Lambda^{+} \tilde{G}^{\mathbb{C}}_{\sigma}$
and $\Lambda^{-} \tilde{G}^{\mathbb{C}}_{\sigma}$ are connected their images under the projection are $\Lambda^{+}_\FC {G}^{\mathbb{C}}_{\sigma}$
and $\Lambda^{-}_\FC {G}^{\mathbb{C}}_{\sigma}$ respectively.
 From this the Birkhoff Decomposition Theorem for $(\Lambda G_\sigma^\C)^0$ follows. The special case of primary interest in this paper will be discussed in detail in the following remark.

\item  Let $J$ denote a nondegenerate quadratic form in $\R^{n+4}$
and  $SO(J,\C)$
the corresponding real special orthogonal group. Let $SO(J,\C)$ denote the
complexified special orthogonal group. Then $SO(J,\C)$ is connected and
has fundamental group $\pi_{1}(SO(J,\C))\cong \mathbb{Z}/2\mathbb{Z}$.
Moreover, if $\sigma$ is an involutive inner automorphism, we have
$\Lambda SO(J,\C)\cong \Lambda SO(J,\C)_{\sigma}$. Therefore
 \begin{equation}
\pi_0 ( \Lambda SO(J,\C)_{\sigma}) \cong\pi_0 ( \Lambda SO(J,\C))\cong \pi_{1}(SO(J,\C))\cong \mathbb{Z}/2\mathbb{Z}.
\end{equation}
The loop group $\Lambda SO^+(1,\tilde{n},\mathbb{C})_{\sigma}$ thus has two connected components.
Finally, choosing $\sigma, K$ and $K^\C$ in the Willmore setting, the group $K^\C$ has two connected components. Therefore also
$\Lambda^+ {G}^{\mathbb{C}}_{\sigma}$ and
$\Lambda^{-} {G}^{\mathbb{C}}_{\sigma}$ have two connected components.

\item Much of the above is contained in \cite{PS}, Section 8.5 (see also \cite{SW}). Note, however, that our real group $G=SO^+(1,n+3)$ is not compact.
\end{enumerate}
\end{remark}

\ \\
{\em Proof of Theorem \ref{thm-birkhoff-0}.}\
  Consider the universal cover $\pi: Spin(1,n+3, \mathbb{C})\rightarrow SO(1,n+3,\mathbb{C})$.
Then $\pi$ induces a homomorphism from  $\Lambda Spin(1,n+3, \mathbb{C})_\sigma$ to
$(\Lambda SO(1,n+3, \mathbb{C})_\sigma)^0$, the identity component of  $\Lambda SO(1,n+3, \mathbb{C})_\sigma$.
 Projecting the decomposition of  Theorem \ref{thm-birkhoff} with $\tilde G^\C = Spin(1,n+3,\C)$ to $ SO(1,n+3, \mathbb{C})$, we obtain the Birkhoff factorization Theorem \ref{thm-birkhoff-0}.
\hfill $\Box$


\subsection{Iwasawa Decomposition}

From here on we will write, for convenience, $\Lambda G^0_\sigma$ for
$(\Lambda G_\sigma)^0$.
For our geometric applications we also need a second loop group decomposition. Ideally we would like to be able to write
any $g\in (\Lambda G^\mathbb{C})^0_\sigma$ in the form $g=hv_+$ with  $h\in (\Lambda G)^0_\sigma$ and
 $ v_+ \in (\Lambda^+ G^\mathbb{C})^0_\sigma =
\Lambda^+_\mathcal{C} G^\mathbb{C}_\sigma$.  Unfortunately, this is not always possible.

For the discussion of this situation we start again by considering  the universal cover $ \tilde{\pi}^\C      :\tilde{G}^\C \rightarrow G^\C$. Then $\tau$, the anti-holomorphic involution of $G^\C$ defining $G$, and $\sigma$ have natural
lifts, denoted  by $\tilde{\tau}$ and $\tilde{\sigma}$, to $ \tilde{G}^\C$.
The  fixed point group
$\tilde{K}^\C$ of $\tilde{\sigma}$ is connected and projects onto $(K^\C)^0$.
The fixed point group of $\tilde{\tau}$ in $ \tilde{G}^\C$ is generally not  connected, like in the Willmore surface  case, where  the real elements ${Fix}^{\tau}(\tilde{G}^\C) = Spin(1,n+3)$  of $\tilde{G}^\C = Spin(1,n+3, \C)$ form a non-connected group.
But it suffices to consider its connected component
$ ( {Fix}^{\tau}(\tilde{G}^\C) )^0 =  \tilde{G}$ which projects onto $G$ under $\tilde{\pi}: \tilde{G} \rightarrow G$.

From here on we will write, for convenience, $\Lambda G^0_\sigma$ for
$(\Lambda G_\sigma)^0$. Then we trivially obtain the disjoint union
\begin{equation}\label{general-Iwasawa-univ}
\Lambda \tilde{G}^{\C}_{\sigma} =
\bigcup \Lambda \tilde{G}_{\sigma}\cdot \tilde{\delta} \cdot
\Lambda^{+} \tilde{G}^{\mathbb{C}}_{\sigma},
\end{equation}
where the $\tilde{\delta}$'s simply parametrize the different double cosets.
 Note that in this equation all groups are connected.
We can (and will) assume that $\tilde{\delta} = e$ occurs.
For the corresponding double coset, since the corresponding Lie algebras add to give the full loop algebra, we obtain:

\begin{theorem}
The multiplication $\Lambda \tilde{G}_{\sigma}\times \Lambda^{+} \tilde{G}^{\mathbb{C}}_{\sigma}\rightarrow
\Lambda \tilde{G}^{\mathbb{C}}_{\sigma}$ is a real analytic map onto the connected open subset
$ \Lambda \tilde{G}_{\sigma} \cdot \Lambda^{+} \tilde{G}^{\mathbb{C}}_{\sigma}      = \mathcal{I}^{\mathcal{U}}_e \subset\Lambda \tilde{G}^{\mathbb{C}}_{\sigma}$.

\end{theorem}

From this the Iwasawa Decomposition Theorem \ref{thm-Iwasawa-0} for $(\Lambda G^\C_\sigma)^0$ follows after an application of the natural projection as above.

\begin{remark}\
\begin{enumerate}
\item We have seen above that the Iwasawa cell with middle term $I$ is open.
But also the Iwasawa cell with middle term $\delta_0 = diag(-1,1,1,1,-1,...,1)$ is open. To verify this we consider $\delta_0^{-1} \Lambda so(1, n+3)_\sigma \delta_0
\oplus \Lambda^+so(1,n+3,\C)$ and observe that the first summand is equal to $\Lambda so(1, n+3)$. As a consequence,
\[\delta_0^{-1} \Lambda SO^+(1, n+3)^0_\sigma \delta_0
\cdot  \Lambda_\mathcal{C}^+SO(1,n+3,\C)_\sigma\] is open and the claim follows.

\item  Using work of Peter Kellersch \cite{Ke1} it seems to be possible to show that there are exactly two open Iwasawa cells in this case. We will not need such a statement in this paper.
\end{enumerate}
\end{remark}

\subsection{On a complementary  solvable  subgroup of  $ SO^+(1,3) \times SO(n)$ in $ SO(1,n+3,\C)$}

We consider $K^{\mathbb{C}}$, the connected complex subgroup of $SO(1,n+3,\C)$ with Lie algebra
 $\mathfrak{so}(1,3,\C)\times \mathfrak{so}(n,\C)$ considered as
Lie algebra of matrices acting on
$\mathbb{C}^4\oplus\mathbb{C}^{n}$. (Hence we consider the ``basic" representations
 of these Lie algebras as introduced in section 1.2.1 ). Clearly
 $K^{\mathbb{C}}=SO(1,3,\mathbb{C})\times SO(n,\mathbb{C}).$

\begin{theorem} There exist connected solvable subgroups $S_1 \subset SO^+(1,3,\C)$ and $S_2 \subset SO(n,\C)$ such that
 \begin{equation}
\left(SO^+(1,3) \times SO(n)\right) \times (S_1 \times S_2) \rightarrow
\left(SO^+(1,3) \cdot S_1\right) \times \left( SO(n) \cdot S_2\right)
\end{equation}
is a real analytic diffeomorphism onto an open subset of $K^\C$.
\end{theorem}
\begin{proof}
Since $SO(n)$ is a connected maximal compact subgroup of $SO(n,\mathbb{C})$,
 in $SO(n,\mathbb{C})$ we have the classical Iwasawa decomposition
 $SO(n,\mathbb{C})=SO(n)\cdot B,$  where $B$ is a solvable subgroup of $SO(n,\mathbb{C})$ satisfying
 $SO(n)\cap B=\{I\}$.

It thus suffices to consider $SO(1,3,\mathbb{C})$ and to prove the existence of a (connected solvable) subgroup $S_1$ of $SO(1,3,\C)$ such that
 \begin{equation}\label{eq-sos}
     \mathcal{S}: SO^+(1,3)\times S_1 \rightarrow SO^+(1,3)\cdot S_1
 \end{equation}
 is a real analytic diffeomorphism and
$
  SO^+(1,3)\cdot S_1
 $
  is open in $SO(1,3,\mathbb{C})$.
 Note, since the map $\mathcal{S}$ is clearly analytic and surjective, it suffices, as we will see below,  to verify that it is also open and that  $SO^+(1,3)\cap S_1=\{I\}$ holds.

 At any rate, we  need to find
a solvable Lie subalgebra $\mathfrak{s}_1$ of
$\mathfrak{so}(1,3,\mathbb{C})$
satisfying
\begin{equation} \mathfrak{so}(1,3) + \mathfrak{s}_1=\mathfrak{so}(1,3,\mathbb{C}),\ \
 \mathfrak{so}(1,3)\cap\mathfrak{s}_1=\{0\}.
\end{equation}

Set
\[\mathfrak{s}_1=\left\{\left.\left(
    \begin{array}{cccc}
      0 & i a_{12} & a_{13} & ia_{13} \\
      ia_{12} & 0 & a_{23} & ia_{23} \\
      a_{13} & -a_{23}  & 0 & ia_{34} \\
      i a_{13} & -ia_{23} & -ia_{34} & 0 \\
    \end{array}
  \right)\right|\ a_{12},a_{34}\in\R, a_{13}, a_{23}\in\C\right\}.
\]
We see that $\mathfrak{so}(1,3) \cap\mathfrak{s}_1=\{0\}$ and $\mathfrak{so}(1,3)^{\C}=\mathfrak{so}(1,3)\oplus\mathfrak{s}_1$ hold.
 It is straightforward to see  that $\mathfrak{s}_1$ is a solvable Lie algebra.
Let $S_1$ be the connected Lie subgroup of $SO(1,3,\C)$ with Lie algebra $Lie(S_1)=\mathfrak{s}_1$. So we have that the map $ \mathcal{S}$ is  a local  diffeomorphism near the identity element  by  Chapter II, Lemma 2.4 of \cite{Helgason}. This also implies that the map $\mathcal{S}$ is open.

 Next we finally show that  $SO^+(1,3)\cap S_1=I$ holds.
 We recall  that the exponential map $\exp:\mathfrak{so}(1,3)\rightarrow SO^+(1,3)$ is surjective.  Then every element of $SO^+(1,3)\cdot S_1$ has the form
$\exp( \mathfrak{A}) \exp( \mathfrak{B})\exp(\mathfrak{C}),$ with $\mathfrak{A}\in\mathfrak{so}(1,3)$,
  $\mathfrak{B}$ contained in the abelian subalgebra of  the $2\times2-$ diagonal blocks in $\mathfrak{s}_1 $, and
 $\mathfrak{C}$  in the nilpotent subalgebra of $\mathfrak{s}_1 $  consisting of the  ``off-diagonal''  blocks
(Note that for every  off-diagonal block $Q$ in $\mathfrak{s}_1$ we have $Q^2 = 0$).   Let  $\exp( \mathfrak{A}) \exp( \mathfrak{B})\exp(\mathfrak{C})\in SO^+(1,3)\cap S_1$.  Then $ \exp( \mathfrak{B})\exp(\mathfrak{C})= \exp(\overline{ \mathfrak{B}})\exp(\overline{\mathfrak{C}})$ and
  \[\exp(\overline{ \mathfrak{B}})^{-1}\exp( \mathfrak{B})= \exp(\overline{\mathfrak{C}})\exp(\mathfrak{C})^{-1}\]
follows.
  As a consequence, $\exp(\overline{ \mathfrak{B}})^{-1}\exp( \mathfrak{B})= \exp(\overline{\mathfrak{C}})\exp(\mathfrak{C})^{-1}=I_4$, i.e.,
  $\exp( \mathfrak{B})=\exp(\overline{ \mathfrak{B}})$ and $\exp(\mathfrak{C})=\exp(\overline{\mathfrak{C}})$.  The definition of    $ \mathfrak{s}_1 $ now implies  $\exp(\mathfrak{B})=\exp(\mathfrak{C})=I_4$.

 To see that the inverse map is real analytic we take a small neighbourhood in  $SO^+(1,3) \cdot  S_1$ of the form $gVs$, where $V$ is a small neighbourhood of the identity $I$. Since locally near $I$ our map is a real analytic diffeomorphism, the claim follows.
\end{proof}

\begin{remark}
We point out that the set $SO^+(1,3)S_1$ is not all of
$SO(1,3,\C)$. For example
\[\left(
               \begin{array}{cccc}
                 \frac{\sqrt{2}}{2} & 0 & \frac{i\sqrt{2}}{2} & 0 \\
               0  & 1 & 0 & 0 \\
                 \frac{i\sqrt{2}}{2} & 0 &\frac{\sqrt{2}}{2} & 0 \\
                0 & 0 & 0 & 1 \\
               \end{array}
             \right)
\]
 is an element of $SO^+(1,3,\C)$ which is not contained in $ SO^+(1,3)S_1$.
\end{remark}


\section{Appendix B: Proof of Theorem \ref{normalizationlemma} }

Let $U$ be a contractible open subset of some Riemann surface $M$. Then the Maurer-Cartan form of any strongly conformally harmonic map
$f: M \rightarrow G/K$ is  real analytic on $U$.
Moreover, the matrix $B_1$ in \eqref{eq-B0} satisfies $B_1^t I_{1,3} B_1 = 0$
 by Definition \ref{stronglyconfharm}.  In particular, the columns of $B_1$ are orthogonal  complex null vectors relative to the quadratic form
defined by $I_{1,3}$. Our goal is to find a simple canonical form of $B_1.$

The first case to consider is, where $B_1$ consists of one column.
It is easy to verify that every non-vanishing fixed complex null vector $b \in \C^4$ can be mapped by $SO^+(1,3)$ into the space
\[\mathcal{N} = \C (1,-1,0,0)^t\ \hbox{ or into }\ \mathcal{N}_{\pm}= \C (0,0,1, \pm i)^t\] according to whether the real part of $b$ is lightlike (possibly including $0$ ) or spacelike respectively.  For a real analytic complex valued null vector function $b$ it is not possible, in general, to map $b$ by some real analytic matrix function $A \in SO^+(1,3)$ into one of these spaces only.
But if $B_1 = b$ corresponds to a non-trivial strongly conformally harmonic map, then one can map $b$ into the sum $\mathcal{N} \oplus  \mathcal{N}_+$ or  into $\mathcal{N} \oplus  \mathcal{N}_-$.

We start by proving the desired result in the case, where $b$ never vanishes on $U$.

\begin{lemma}\label{lemma-null}
Let $U$ be a contractible open subset of $\C$ and $b: U \rightarrow \C^4_1\backslash\{0\}$
a real analytic null vector.
Then there exists a real analytic map $A:U \rightarrow SO^+(1,3)$ such that
the function $Ab$ is contained in $\mathcal{N} \oplus  \mathcal{N}_+$, i.e. $Ab$ has the form $(p, -p, q,  i  q)^t$.
\end{lemma}
\begin{proof}
The proof is particularly easy if one realizes $\C^4 \cong Mat(2, \C)$ with quadratic form
\[\langle X,X^\prime \rangle =
 X_{11} X^\prime_{22}  - X_{12} X^\prime_{21} + X_{22} X^\prime_{11}  - X_{21} X^\prime_{12} .\]
In this realization the non-vanishing complex null vectors are exactly all matrices of rank 1, i.e. all non-vanishing matrices of determinant $0$.

As real form we choose $\R^4_1 \cong Herm(2,\C)$, the space of $2 \times 2-$ complex hermitian matrices.

The group $SO(4,\C) \cong (SL(2, \C) \times SL(2,\C) ) / \{ \pm I \}$ acts on $Mat(2,\C)$ by $(g,h).X = gXh^{-1}$. Then $SO^+(1,3) \cong SL(2,\C) / \{\pm I\}$ acts by $g.X = gX \bar{g}^t$.

In the spirit of what was said before the statement of the lemma, we want to transform  any real analytic map $X$ defined in $U$ with values in $Mat(2,\C)$
 into the complex space $\C E_{11} \oplus \C E_{21}$ by the operation
$X \rightarrow gX \bar{g}^t$, where $g \in SL(2, \C)$ is defined on $U$ and real analytic.

Now it is an easy exercise to verify that for any $z_0 \in U$ there is a matrix function, $q_\delta$, defined on some open neighbourhood $U_\delta \subset U$ of $z_0$ such that $X \bar{g}_\delta^t$ has  on $U_\delta$  an identically vanishing second column, if $det(X) \equiv 0$ on $U$.  Of course, then also
$g_\delta X \bar{g}_\delta^t$ has identically vanishing second column.

Next we consider $h_{\alpha \beta} = q_\alpha q_\beta^{-1}$. These matrix functions are defined on $U_\alpha \cap U_\beta$  and form a cocycle  relative to the covering given by the $U_\delta$. Moreover, this cocycle consists of upper triangular matrices of determinant $1$. Therefore, since $U$ is contractible, this cocycle is a co-boundary. Therefore  there exist upper triangular matrices $h_\delta$  of determinant $1$  and  defined on $U_\delta$ satisfying $ q_\alpha q_\beta^{-1}= h_\alpha h_\beta^{-1}$.
As a consequence  $g = h_\alpha^{-1} q_\alpha$ is defined on $U$ and
 the  second column  of  $gX\bar{g}^t$ vanishes identically on $U$.
\end{proof}

Now it is fairly straightforward to prove Theorem 3.4.
If the maximal rank of $B_1$ is $1$, then the argument would be easy, if all columns of $B_1$  would be real analytic multiples of one of the columns, say, the first column of  $B_1$. The actual argument follows in a sense the same idea, but is a bit more sophisticated. If the maximal rank of $B_1$ is $2$, then in the complex vector space spanned by  two generically linearly independent columns of $B_1$ one constructs a real vector which then implies quite directly  what we want in view of the condition $B_1^t I_{1,3} B_1 = 0$.\\

{\em Proof of Theorem \ref{normalizationlemma}:}

First we mention some result which is true for all harmonic maps into a symmetric space, namely that any such harmonic map can be constructed by the loop group method from ``holomorphic potentials''. The proof is as in \cite{DPW} and is not related in any way to the specific properties of conformally harmonic maps which we investigate.

Therefore, let's consider the holomorphic potential of the harmonic map $f$ (for a discussion we refer to  Section 4.3).
 By Theorem \ref{Holo-Willmore}, let
\begin{equation}\label{eq-potential-h}
\xi=(\lambda^{-1}\xi_{-1}+\sum_{j\geq0}\lambda^j\xi_j)dz,\  \hbox{ with }\ \xi_{-1}=\left(
    \begin{array}{cc}
      0 & R_1 \\
      -R_1^tI_{1,3} & 0 \\
    \end{array}
  \right),
\end{equation}
be the corresponding holomorphic potential on $U$. Then there exist some real analytic matrices $S_1 \in SO^+(1,3,\C)$ and $S_2\in SO(n,\C)$, such that $B_1=S_1R_1S_2$  holds.

Our claim is equivalent to that there exists some real analytic matrix function $A:U\rightarrow SO^+(1,3)$ such that $AB_1$ has the form desired.

It is easy to see that it suffices to prove this special form for $Q_1 = S_1 R_1$.
Let's write $Q_1$ as a matrix of column vectors, $Q_1 = (q_1,...,q_n).$ Since we assume w.l.g. $B_1\neq 0,$ also $Q_1 \neq 0.$  Hence one of the columns of $Q_1$ does not vanish. Let's assume w.l.g. that the first column $q_1$ of $Q_1$ does not vanish identically. Then the corresponding first column $r_1$ of $R_1$ does not vanish identically. Since $r_1$ is holomorphic, one can factor out some holomorphic (product) function $h_1$ such that $r_1 = h_1 \hat{r}_1$, where $\hat{r}_1$ is holomorphic and never vanishes on $U$.
As a consequence, $q_1 = h_1 \hat{q}_1$, where the globally defined and real analytic map $\hat{q}_1$ never vanishes.

 From Lemma \ref{lemma-null} we obtain now that there exists some real analytic matrix function $A:U\rightarrow SO^+(1,3)$ such that $A \hat{q}_1$ has the desired form
\[A \hat{q}_1=aE_{11}+bE_{21}.\]
Hence also $ A q_1 = h_1  A\hat{q}_1$ has the desired form.\vspace{2mm}

 Let's assume next that $B_1$ has maximal rank $1$. Then  we claim that each column of $Q_1$ is a multiple of $\hat{q}_1$ and
 this multiple is holomorphic on $U$.
As a consequence, $AB_1$ has the desired form.

To prove the claim above, note that the relation between $A S_1 r_1$ and $A S_1 r_j$ can already be found between $r_1$ and $r_j$. By the argument above we can write
$r_1 = h_1 \hat{r}_1$ and $r_j = h_j \hat{r}_j$ with $\hat{r_1}$ and $\hat{r}_j$
never vanishing on $U$. Let $U'$ denote the discrete subset of points in $U$, where none of the occurring, not identically vanishing functions/vector entries, vanish. On this set one can show that an entry of $\hat{r}_1$ does not vanish identically if and only if the corresponding entry of $\hat{r}_j$ does not vanish identically.
Now it is easy to verify that $\hat{r}_j$ is a holomorphic multiple of $\hat{r}_1$. Whence the statement above. \vspace{2mm}

Next let's assume that the maximal rank of $B_1$ is $2$.
In this case we apply the argument given above for $q_1$  to each column of $B_1$, i.e. we write $q_j = h_j \hat{q}_j$, where $\hat{q}_j$ never vanishes on $U$. Note, the case $q_j \equiv 0$ corresponds to $h_j \equiv 0$ and $\hat{q}_j = const \neq 0$.
 We will also assume w.l.g. that the second column of $B_1$ does not vanish identically. Hence $\hat{q}_1$ and $\hat{q}_2$  never vanish  on $U$ and are linearly independent on an open and dense subset $\tilde{U}$ of $U$.

 For the following argument we realize again $\C^4 $ by $Mat(2,\C)$. As before we can apply the theorem above to $\hat{q}_1$ and can assume w.l.g. that $\hat{q}_1$ is a $2 \times 2-$matrix for which the second column is $0$. We will use the notation $\hat{q}_1 = aE_{11} + bE_{21}$ and note that by assumption $|a|^2 + |b| ^2 $ never vanishes on $U$.

  If $ab\equiv0$, then $a\equiv 0$ or $b\equiv0$ on $U$.
The nilpotency condition  $L^t I_{1,3} L=0$ for
$Q_1 = S_1 R_1$ implies that the claim of the theorem holds, after one more (constant) gauging if necessary.

 If $ab\neq 0$, then after applying a constant
$SL(2,\C )-$matrix, if necessary, we can assume w.l.g. that $a \neq 0$ and $b \neq 0$ on the open and dense subset $\tilde{U}$ of $U$.

 In this case, using that $\hat{q}_1$ and $\hat{q}_2$ are perpendicular to each other and to themselves,  it is straightforward to see by a computation on $\tilde{U}$ that $\hat{q}_2$
is  either of the form
$  \hat{a} E_{11}+\hat{b} E_{21} $,  or of the form \[ \hat{a} (aE_{11}+bE_{21})+\hat{b}( aE_{12}+bE_{22}) \hbox{ with } \hat{b} \neq0,\] where the coefficient functions are real analytic on  $\tilde{U}$.
For the first case, we are done, since the coefficients clearly extend to functions defined on $U$.

 In the second case one can show by a simple computation that the complex vector space spanned by $\hat{q}_1$ and $\hat{q}_2$ contains the hermitian matrix
\[ w_1 = |a|^2 E_{11} + \bar{a} b E_{21} + a \bar{b} E_{12} |b|^2.\]
Clearly, this matrix is defined on all of $U$. Moreover, the $SL(2, \C )-$matrix
$g = c_0 (\bar{b} E_{11} - a E_{12} + \bar{a} E_{21} + b E_{22} )$, with $c_0 = 1/{\sqrt{|a|^2 + |b|^2}}$ is a real analytic function on $U$ which transforms
$w_1$ into the matrix  $w_2 = (|a|^2 + |b|^2) E_{11}.$ As a consequence, after this transformation the complex vector space spanned by $q_1$ and $q_2$ contains the constant matrix function $q_0 = E_{11}$.

 By the construction carried out so far, the vectors $q_0, q_1,...$ all are perpendicular to each other and to themselves. In particular, $\langle q_0,q_j\rangle = 0$ and  $\langle q_j,q_j\rangle =0$  for $j>0$ implies by a straightforward computation that each of the matrices $q_j, j>0,$ has a vanishing second column or a vanishing second row. But the relation $\langle q_1, q_k\rangle =0, k>1,$ implies that all $q_k$ have the same type as $q_1$. Hence all
$q_j, j \geq 1,$ are   contained in either $\mathcal{N} \oplus  \mathcal{N}_+$  or $\mathcal{N} \oplus  \mathcal{N}_-$.

   \hfill   $\Box$

\begin{corollary} \label{corollary-B1-vanish}
Let $h_0$ denote the greatest common divisor of the holomorphic
functions
$h_j, j = 1, \dots, n,$ defined in the proof above, then $B_1 = h_0 \hat{B}_1$ and  $\hat{B}_1(z) \neq 0$ for all $z \in U$.
\end{corollary}

{\small{\bf Acknowledgements}\ \
This work was started when the second named author visited the Department of Mathematics of Technische Universit\"{a}t  M\"{u}nchen, and  the Department of Mathematics of Tuebingen University. He would like to express his sincere gratitude for both the hospitality and financial support. The second named author is thankful to Professor Changping Wang and Xiang Ma for their suggestions and encouragement. The second named author is also thankful to the ERASMUS MUNDUS TANDEM Project for the financial supports to visit the TU M\"{u}nchen. This work was partly supported by the Project 11201340 and 11571255 of NSFC.}

{\footnotesize

\def\refname{Reference}

}
\vspace{2mm}
{\small
Josef Dorfmeister

Fakult\" at f\" ur Mathematik,

TU-M\" unchen, Boltzmann str. 3,

D-85747, Garching, Germany

{\em E-mail address}: dorfm@ma.tum.de\\

Peng Wang

Department of Mathematics

Tongji University, Siping Road 1239

Shanghai, 200092, P. R. China

{\em E-mail address}: {netwangpeng@tongji.edu.cn}
}
\end{document}